\documentclass{article}

\usepackage{amsmath,amsthm,amssymb,mathtools,enumerate}

\newtheorem{theorem}{Theorem}[section]
\newtheorem{observation}[theorem]{Observation}
\newtheorem{lemma}[theorem]{Lemma}
\newtheorem{proposition}[theorem]{Proposition}
\newtheorem{corollary}[theorem]{Corollary}

\numberwithin{equation}{section}

\theoremstyle{definition}
\newtheorem{definition}[theorem]{Definition}

\newtheorem{remark}[theorem]{Remark}
\newtheorem{example}[theorem]{Example}

\newcommand\R{\mathbb{R}}
\newcommand\C{\mathbb{C}}
\newcommand\Q{\mathbb{Q}}

\newcommand\N{\mathbb{N}}

\newcommand{\cA}{\mathcal{A}}
\newcommand{\cB}{\mathcal{B}}
\newcommand{\cC}{\mathcal{C}}
\newcommand{\cD}{\mathcal{D}}

\newcommand{\cF}{\mathcal{F}}

\newcommand{\cK}{\mathcal{K}}
\newcommand{\cL}{\mathcal{L}}
\newcommand{\cM}{\mathcal{M}}
\newcommand{\cN}{\mathcal{N}}
\newcommand{\cP}{\mathcal{P}}
\newcommand{\cQ}{\mathcal{Q}}
\newcommand{\cR}{\mathcal{R}}
\newcommand{\cU}{\mathcal{U}}
\newcommand{\cV}{\mathcal{V}}
\newcommand{\cO}{\mathcal{O}}

\newcommand{\bA}{\mathbf{A}}
\newcommand{\bD}{\mathbf{D}}
\newcommand{\bS}{\mathbf{S}}
\newcommand{\bW}{\mathbf{W}}
\newcommand{\bX}{\mathbf{X}}
\newcommand{\bY}{\mathbf{Y}}
\newcommand{\bZ}{\mathbf{Z}}

\newcommand{\rT}{\mathrm{T}}

\DeclareMathOperator{\id}{id}

\DeclareMathOperator{\re}{Re}
\DeclareMathOperator{\im}{Im}

\DeclareMathOperator{\tr}{tr}
\DeclareMathOperator{\sa}{sa}

\DeclareMathOperator{\qf}{qf}
\DeclareMathOperator{\tp}{tp}
\DeclareMathOperator{\orb}{orb}
\DeclareMathOperator{\Ent}{Ent}
\DeclareMathOperator{\Th}{Th}

\DeclareMathOperator{\acl}{acl}

\DeclarePairedDelimiter{\norm}{\lVert}{\rVert}
\DeclarePairedDelimiter{\ip}{\langle}{\rangle}

\begin{document}
	
	\title{Covering entropy for types in tracial $\mathrm{W}^*$-algebras}
	\author{David Jekel}
	
	\maketitle	
	
	\begin{abstract}    
		We study embeddings of tracial $\mathrm{W}^*$-algebras into a ultraproduct of matrix algebras through an amalgamation of free probabilistic and model-theoretic techniques.  Jung implicitly and Hayes explicitly defined \emph{$1$-bounded entropy} $h$ through the asymptotic covering numbers of Voiculescu's microstate spaces, that is, spaces of matrix tuples $(X_1^{(N)},X_2^{(N)},\dots)$ having approximately the same $*$-moments as the generators $(X_1,X_2,\dots)$ of a given tracial $\mathrm{W}^*$-algebra.  We study the analogous covering entropy for microstate spaces defined through formulas that use suprema and infima, not only $*$-algebra operations and the trace--formulas which arise in the model theory of tracial $\mathrm{W}^*$-algebras initiated by Farah, Hart, and Sherman.  By relating the new theory with the original $1$-bounded entropy, we show that if $\mathcal{M}$ is a separable tracial $\mathrm{W}^*$-algebra with $h(\cN:\cM) \geq 0$, then there exists an embedding of $\cM$ into a matrix ultraproduct $\cQ = \prod_{n \to \cU} M_n(\C)$ such that $h(\cN:\cQ)$ is arbitrarily close to $h(\cN:\cM)$.  We deduce that if all embeddings of $\cM$ into $\cQ$ are automorphically equivalent, then $\cM$ is strongly $1$-bounded and in fact has $h(\cM) \leq 0$.
	\end{abstract}
	
	\newpage
	
	\tableofcontents
	
	\newpage
	
	
	\section{Introduction}
	
	\subsection{Overview}
	
	The study of $\mathrm{W}^*$-algebras or von Neumann algebras is a deep and challenging subject with many connections to fields as diverse as ergodic theory, geometric group theory, random matrix theory, quantum information, and model theory.  Our present goal is to bring together two of these facets---the model theory of tracial $\mathrm{W}^*$-algebras studied in \cite{FHS2013,FHS2014,FHS2014b,AGKE2022} and Voiculescu's free entropy theory which, roughly speaking, quantifies the amount of matrix approximations for the generators of $\mathrm{W}^*$-algebra (see e.g.\ \cite{VoiculescuFE2,VoiculescuFE3,GePrime,Jung2007S1B,Hayes2018}).  On the free entropy side, we will work in the framework of Hayes' $1$-bounded entropy $h$ \cite{Hayes2018} which arose out of the work of Jung \cite{Jung2007S1B}; for history and motivation, refer to \cite[\S 2]{HJNS2021}.  The sibling paper \cite{JekelModelEntropy} develops the analog of Voiculescu's free microstate entropy for the setting of model-theoretic types.
	
	We adapt the framework of $1$-bounded entropy \cite{Jung2007S1B,Hayes2018} to capture data about the generators' model-theoretic \emph{type} and not only their \emph{non-commutative law}.  The non-commutative law of a tuple $(X_j)_{j \in \N}$ from $\cM = (M,\tau)$ encodes the joint moments $\tau(p(X))$ for non-commutative $*$-polynomials $p$; laws thus describe tracial von Neumann algebras with chosen generators up to generator-preserving isomorphism.  The \emph{type} includes the values of more complicated formulas that involve not only the addition, adjoint, product, and trace operations, but also taking suprema and infima in auxiliary variables over an operator norm ball in $\cM$.  For instance, the type would include the value of the formula
	\[
	\sup_{Z \in D_1^{\cM}} \inf_{Y \in \overline{B}_{\cM}(0,1)} [\tau(X_1YX_2Z)^2 + \tau(X_2^2Y^5)],
	\]
	where $D_1^{\cM}$ denotes the closed unit ball with respect to operator norm.  The entropy $\Ent^{\cU}(\mu)$ of a type $\mu$ is defined in terms of the exponential growth rates of the covering numbers of microstate spaces (spaces of matrix tuples with approximately the same type as our chosen generators, as in Voiculescu's work), just like Jung and Hayes' $1$-bounded entropy except with types instead of laws.  However, we prefer the term ``covering entropy'' rather than ``$1$-bounded entropy'' as a more intrinsic description of the definition.  The superscript $\cU$ denotes the fact that we take limits with respect to a fixed non-principal ultrafilter $\cU$ on $\N$.
	
	Just as in the original definition of the $1$-bounded entropy, a key property of the covering entropy $\Ent^{\cU}(\mu)$ is that it is invariant under change of coordinates (see \S \ref{subsec:invariance}).  More precisely, if $\bX$ and $\bY$ are tuples from $\cM$ with $\mathrm{W}^*(\bX) = \mathrm{W}^*(\bY)$, then their types $\tp^{\cM}(\bX)$ and $\tp^{\cM}(\bY)$ have the same covering entropy (Corollary \ref{cor:invariance}).  This allows us to define the entropy $\Ent^{\cU}(\cN:\cM)$ of a separable tracial $\mathrm{W}^*$-algebra $\cN \subseteq \cM$ as the entropy of the type of any generating set.  As suggested in \cite{HJNS2021}, we streamline the proof of this invariance property using the result that every tuple $\bY$ from $\mathrm{W}^*(\bX)$ can be expressed as $\mathbf{f}(\bX)$ for some quantifier-free definable function (see \cite[\S 13]{JekelThesis}).  More generally, we can extend the definition of $\Ent^{\cU}(\cN:\cM)$ to the case where $\cN \subseteq \cM$ is not separable by setting it to be the supremum of $\Ent^{\cU}(\cN_0:\cM)$ over separable $\mathrm{W}^*$-subalgebras $\cN_0 \subseteq \cN$ or equivalently, the supremum of $\Ent^{\cU}(\tp^{\cM}(\bX))$ over all tuples $\bX \in L^\infty(\cN)^{\N}$ (see Definition \ref{def:W*entropy}).
	
	The covering entropy $\Ent^{\cU}(\cN: \cM)$ can be viewed intuitively as a measurement of the amount of tracial $\mathrm{W}^*$-embeddings of $\cN$ into the matrix ultraproduct $\cQ = \prod_{n \to \cU} M_n(\C)$ that extend to elementary embeddings of $\cM$ (compare \S \ref{subsec:embeddings1}).  This is the analog of the idea that the $1$-bounded entropy $h(\cN:\cM)$ of $\cN$ in the presence of $\cM$ quantifies the amount of $\mathrm{W}^*$-embeddings of $\cN$ into $\cQ$ that extend to any embedding of $\cM$.  Thus, our work is motivated in part by the study of embeddings into ultraproducts, which is one theme of recent work on von Neumann algebras  \cite{Popa2019,Goldbring2020,IS2021,AKE2021,AGKE2022,Gao2020}.
	
	We make a precise connection between $\Ent^{\cU}(\cN:\cM)$ and $1$-bounded entropy as follows.  There is a canonical projection $\pi_{\qf}$ from the space of types to the space of non-commutative laws, since a non-commutative law describes the evaluation of quantifier-free formulas (rather than all logical formulas) in a tuple $\bX$.  Given a non-commutative law (or quantifier-free type) $\mu$, the $1$-bounded entropy $h^{\cU}(\mu)$ can be expressed through the following variational principle (Corollary \ref{cor:qfvariational}):
	\begin{equation} \label{eq:introvariational}
		h^{\cU}(\mu) = \sup_{\nu \in \pi_{\qf}^{-1}(\mu)} \Ent^{\cU}(\nu).
	\end{equation}
	Thus, the $1$-bounded entropy is the quantifier-free version of the entropy for types.
	
	In a similar way, the $1$-bounded entropy of $\cN$ in the presence of $\cM$ is the version using existential types.  Entropy in the presence is described using microstates for a tuple $\bX$ in $\cN$ such that there exist compatible microstates for a tuple $\bY$ that generates $\cM$.  In the model-theoretic framework, the existence of such microstates for $\cM$ is described through the evaluation of existential formulas in the original generators and their microstates (see \S \ref{subsec:inthepresence}).  Similar to the quantifier-free setting, there is a projection $\pi_{\exists}$ from the space of types into the space of existential types, and a similar variational principle expressing the covering entropy of an existential type $\mu$ as the supremum of $\Ent^{\cU}(\nu)$ over full types $\nu \in \pi_{\exists}^{-1}(\mu)$ (Lemma \ref{lem:existentialversusfull}).
	
	Altogether these ingredients allow us to prove the following result about ultraproduct embeddings, which is restated and proved in Theorem \ref{thm:main2}:
	
	\begin{theorem} \label{thm:main}
		Let $c \in \R$.  Let $\cN \subseteq \cM$ be an inclusion of separable tracial $\mathrm{W}^*$-algebras $h^{\cU}(\cN:\cM) > c$.  Then there exists an embedding $\iota$ of $\cM$ into the matrix ultraproduct $\cQ = \prod_{n \to \cU} M_n(\C)$ such that $\Ent^{\cU}(\iota(\cN): \cQ) > c$, hence also $h^{\cU}(\iota(\cN): \cQ) > c$.
	\end{theorem}
	
	The hypotheses of the theorem hold for instance when $\cN = \cM$ is a nontrivial free product by \cite[Proposition 6.8]{VoiculescuFE3} and \cite[Corollary 3.5]{Jung2007S1B} and \cite[Proposition A.16]{Hayes2018} since $h^{\cU}(\cM:\cM) = \infty$.  They also hold for the von Neumann algebras of groups with non-approximately-inner cocycles by \cite[Theorem 3]{Shlyakhtenko2009} and \cite[Corollary 3.5]{Jung2007S1B} and \cite[Proposition A.16]{Hayes2018}.
	
	In particular, since there exists $\cM$ with $h^{\cU}(\cM:\cM) = \infty$, the theorem implies that there exist types in $\cQ$ with arbitrarily large covering entropy, and therefore, $h^{\cU}(\cQ:\cQ) = \infty$.  Similarly, the entropy $\Ent^{\cU}(\cQ:\cQ)$ given by Definition \ref{def:W*entropy} is infinite.
	
	\begin{corollary}
		Let $\cU$ be a free ultrafilter on $\N$ and let $\cQ = \prod_{n \to \cU} M_n(\C)$.  Then $h^{\cU}(\cQ:\cQ) = \infty$ and $\Ent^{\cU}(\cQ:\cQ) = \infty$.  Hence, $\cQ$ is not strongly $1$-bounded.
	\end{corollary}
	
	The following corollary of Theorem \ref{thm:main} was communicated to me by Ben Hayes.
	
	\begin{corollary} \label{cor:irreducibleembedding}
		Let $\cN \subseteq \cM$ be an inclusion of separable tracial $\mathrm{W}^*$-algebras $h(\cM) > 0$ such that $\cN$ is a $\mathrm{II}_1$ factor (it has trivial center).  Then there exists a free ultrafilter $\cV$ and an embedding $\iota: \cM \to \prod_{n \to \cV} M_n(\C)$ such that $\cN' \cap \prod_{n \to \cV} M_n(\C) = \C$.
	\end{corollary}
	
	\begin{proof}
		The $1$-bounded entropy $h(\cM)$ is the supremum of $h^{\cV}(\cM)$ over free ultrafilters $\cV$.  Hence, there exists some free ultrafilter $\cU$ such that $h^{\cV}(\cM) > 0$ and by Theorem \ref{thm:main} there is an embedding $\iota_0: \cM \to \cQ = \prod_{n \to \cU} M_n(\C)$ with $h^{\cU}(\iota_0(\cM):\cQ) > 0$.
		
		A general fact about $1$-bounded entropy is that if $\cA \subseteq \cB$ and $\cA' \cap \cB$ is diffuse, then $h(\cA:\cB) \leq 0$.  Indeed, if $\cA' \cap \cB$ is diffuse, it contains a diffuse amenable subalgebra $\cC$.  Let
		\[
		\cN = \mathrm{W}^*(u \in \cQ \text{ unitary: } u \cC u^* \cap \cC \text{ is diffuse}),
		\]
		(this is known as the step 1 wq-normalizer of $\cC$ and was introduced in \cite{GP2015}).  Note that $\cA \subseteq \cN$.  Hence, by \cite[Property 1, p.\ 10]{Hayes2018}
		\[
		h^{\cU}(\cA:\cB) \leq h^{\cU}(\cN: \cB).
		\]
		By \cite[Theorem 2.8 and Proposition 3.2]{Hayes2018},
		\[
		h^{\cU}(\cN: \cB) = h^{\cU}(\cC: \cB).
		\]
		Then using \cite[Property 1, p.\ 10]{Hayes2018} again,
		\[
		h^{\cU}(\cC: \cB) \leq h^{\cU}(\cC:\cC) = h^{\cU}(\cC),
		\]
		which is zero since $\cC$ is amenable.  Hence, $h(\cA:\cB) \leq 0$.
		
		By contrapositive, since in our case $h(\iota(\cM):\cQ) \geq h^{\cU}(\iota_0(\cM):\cQ) > 0$, then $\iota_0(\cM)' \cap \cQ$ is not diffuse.  Therefore, it contains a minimal projection $p$.  Let $p \cQ p$ be the compression of $\cQ$ by $p$ equipped with the trace $\tau_{p\cQ p}(x) = \tau_{\cQ}(pxp)/\tau(p)$, and let $\iota: \cM \to p\cQ p$ be the map $\iota(x) = p\iota_0(x)p = \iota_0(x)$.  Since $p$ commutes with $\iota_0(\cM)$, it follows that $\iota_0$ is a $*$-homomorphism, and since $\cM$ is a II$_1$ factor and hence has a unique trace, the map $\iota$ must be trace-preserving.  Because $p$ was a minimal projection in $\iota_0(\cM)' \cap \cQ$, we know $\iota(\cM)' \cap p\cQ p$ has no nontrivial projections and hence is $\C$.
		
		Finally, note that $p \cQ p$ is a matrix ultraproduct $\prod_{n \to \cV} M_n(\C)$ for some ultrafilter $\cV$.  Indeed, by stability of projections there exist projections $p_n$ in $M_n(\C)$ such that $p$ is the equivalence class of $(p_n)_{n \in \N}$ in $\cQ$.  Let $k(n)$ be the rank of $p_n$.  One can check that $p\cQ p = \prod_{n \to \cU} p_n M_n(\C) p_n \cong \prod_{n \to \cU} M_{k(n)}(\C)$, which is simply a matrix ultraproduct for a different ultrafilter $\cV$.
	\end{proof}
	
	As shown in \cite[Theorem 1.2]{JekelModelEntropy}, the analogous result holds for free entropy rather than $1$-bounded entropy \emph{without} having to change the ultrafilter $\cU$ to the ultrafilter $\cV$.

	\subsection{Embeddings into Ultraproducts}
	
	Our results relate to recent work and questions about embeddings into ultraproducts.  Jung \cite{Jung2007embeddings} used the study of microstates to show that a separable tracial $\mathrm{W}^*$-algebra $\cA$ is amenable if and only if all embeddings of $\cA$ into $\cR^{\cU}$ are unitarily conjugate.  Atkinson and Kunnawalkam Elayavalli \cite{AKE2021} strengthened this result by showing that $\cA$ is amenable if and only if all embeddings of $\cA$ into $\cR^{\cU}$ are ucp-conjugate (meaning they are conjugate by an automorphism of $\cR^{\cU}$ that lifts to a sequence of unital completely positive maps $\cR \to \cR$).  Atkinson, Goldbring, and Kunnawalkam Elayavalli \cite{AGKE2022} later showed that if a separable $\mathrm{II}_1$ factor $\cM$ is Connes-embeddable and all embeddings of $\cM$ into $\cM^{\cU}$ are automorphically conjugate, then $\cM \cong \cR$.
	
	One can ask similar questions for embeddings into the ultraproduct $\cQ = \prod_{n \to \cU} M_n(\C)$ for some fixed free ultrafilter $\cU$.  In \cite{AKE2021}, the authors showed that if $\cA$ is a separable Connes-embeddable tracial $\mathrm{W}^*$-algebra and the space of unitary orbits of embeddings $\cA \to \cQ$ is separable in a certain metric, then $\cA$ must be amenable.  In particular, if all embeddings $\cA \to \cQ$ are unitarily conjugate, then $\cA$ is amenable.  It is an open question whether this result still holds when ``unitarily conjugate'' is replaced by ``automorphically conjugate.''  However, Theorem \ref{thm:main} implies the following result, which was pointed out to me by Srivatsav Kunnawalkam Elayavalli:
	
	\begin{corollary} \label{cor:Jungproperty}
		Let $\cA$ be a tracial $\mathrm{W}^*$-algebra.  Suppose that any two embeddings $\cA \to \cQ = \prod_{n \to \cU} M_n(\C)$ are conjugate by an automorphism of $\cQ$.  Then $h^{\cU}(\cA) \leq 0$.
	\end{corollary}
	
	\begin{proof}
		We proceed by contradiction.  Suppose that $h^{\cU}(\cA) > 0$.  By Theorem \ref{thm:main}, there exists an embedding $\alpha: \cA \to \cQ$ with $h^{\cU}(\alpha(\cA): \cQ) > 0$.  Moreover, since $\cA$ is Connes-embeddable, so is $\cR \overline{\otimes} \cA$, so there exists some embedding $\beta: \cR \overline{\otimes} \cA \to \cQ$.  In particular, $\beta(\cR) \subseteq \beta(\cA)' \cap \cQ$.  If we assume for contradiction that $\alpha$ and $\beta|_{\cA}$ are conjugate by an automorphism, then $\alpha(\cA)' \cap \cQ$ also contains a copy of $\cR$, so in particular, $\alpha(\cA)' \cap \cQ$ is diffuse.  As pointed out in the proof of Corollary \ref{cor:irreducibleembedding}, this implies that $h(\alpha(\cA):\cQ) = 0$, which contradicts our choice of $\alpha$.
	\end{proof}
	
	Intuitively, the corollary says that if the space of embeddings modulo automorphic conjugacy is trivial, then the space of embeddings modulo unitary conjugacy is not too large, since $h^{\cU}(\cA)$ quantifies the ``amount'' of embeddings $\cA \to \cQ$ up to unitary conjugacy.  The conclusion that $h^{\cU}(\cA) = 0$ is a weakening of amenability since by Jung's theorem \cite{Jung2007embeddings} amenability is equivalent to the space of embeddings modulo unitary conjugacy being trivial.
	
	We remark that the free entropy techniques used here to study embeddings into $\cQ$ cannot be directly applied to study embeddings into $\cR^{\cU}$.  For instance, Theorem \ref{thm:main} does not make sense with $\cQ$ replaced by $\cR^{\cU}$.  Indeed, $\cR^{\cU}$ has property Gamma by \cite[\S 3.2.2]{FHS2014b}, and every tracial $\mathrm{W}^*$-algebra with property Gamma has $1$-bounded entropy zero (this is a special case of \cite[Corollary 4.6]{Hayes2018} and it is shown explicitly in \cite[\S 1.2, Example 4]{HJNS2021}).  Thus, $h(\cR^{\cU}) = 0$ and therefore, for any subalgebra $\cM$ of $\cR^{\cU}$, we also have $h^{\cU}(\cM: \cR^{\cU}) = 0$ by \cite[\S 2, Property 1]{Hayes2018}.  Hence, Theorem \ref{thm:main} would not hold with $\cR^{\cU}$ instead of $\cQ$.  By contrast, many other operator-theoretic and model-theoretic techniques are more easily applied to $\cR^{\cU}$ than to $\cQ$ since $\cR^{\cU}$ is an ultrapower; see for instance \cite{Goldbring2020,Goldbring2021enforceable,Goldbring2021nonembeddable}.
	
	\subsection{Outline}
	
	In large part, our goal is to establish communication between the free probabilistic and model theoretic subgroups of operator algebras, and to show that many of the notions in free probability (such as non-commutative laws, microstates spaces in the presence, and relative microstate spaces) arise naturally from the model-theoretic framework.  Therefore, we strive to make the exposition largely self-contained and use model-theoretic language throughout.
	
	We start out by explaining the model-theoretic framework for operator algebras in \S \ref{sec:modeltheory}.  In particular, we give a more detailed explanation than current literature of the languages and structures for multiple sorts and multiple domains of quantification for each.  Next, in \S \ref{sec:predicates}, we give a self-contained development of definable predicates and functions of infinite tuples, including the result that every element of a tracial $\mathrm{W}^*$-algebra can be realized by applying a quantifier-free definable function to the generators which was observed in \cite{JekelThesis,HJNS2021}.
	
	\S \ref{sec:entropy} develops the framework of covering entropy for types.  We show the invariance of entropy under change of coordinates in \S \ref{subsec:invariance}, describe the relationship with ultraproduct embeddings in \S \ref{subsec:embeddings1}, and finally show that adding variables in the (model-theoretic) algebraic closure of given tuple $\bX$ does not change its entropy in \S \ref{subsec:algebraic}.
	
	In \S \ref{sec:qfentropy}, we describe the quantifier-free and existential versions of entropy, showing that they agree with the $1$-bounded entropy of Hayes.  We conclude the proof Theorem \ref{thm:main} there.
	
	In the appendix \S \ref{sec:conditional}, we describe a generalization to conditional (or ``relative'') entropy, which focuses on quantifying the embeddings $\cN \to \cQ$ which restrict to a fixed embedding $\iota: \cA \to \cQ$ on a given $\mathrm{W}^*$-subalgebra $\cA$.  The existential version of the conditional covering entropy was studied by Hayes explicitly for $\cA$ diffuse abelian and implicitly for $\cA$ diffuse amenable \cite{Hayes2018}, in which case it agrees with the unconditional version.  However, the conditional covering entropy (for full, quantifier-free, or existential types) makes sense for any diffuse $\cA$ with a specified embedding $\alpha: \cA \to \cQ$ (though, as far as we know, it may depend on the embedding $\alpha$).  Moreover, conditional entropy is natural from the model-theoretic perspective, since it arises from replacing formulas in the original language with formulas that have coefficients from the subalgebra $\cA$.

	\subsection{Acknowledgements}
	
	This was work was partially funded by the NSF postdoc grant DMS-2002826.  I thank the organizers and hosts of the 2017 workshop on the model theory of operator algebras at University of California, Irvine, and of the 2018 long program Quantitative Linear Algebra at the Institute of Pure and Applied Mathematics at UCLA for conferences that greatly contributed to my knowledge of operator algebras, entropy, and model theory.  Special thanks to Ben Hayes for pointing out Corollary \ref{cor:irreducibleembedding}, Srivatsav Kunnawalkam Elayavalli for pointing out Corollary \ref{cor:Jungproperty}, and Jennifer Pi for reading the paper in detail and finding many typos.  Thanks to Isaac Goldbring, Ben Hayes, and Adrian Ioana for their comments on the paper and advice on references.  Finally, thanks to the anonymous referees for numerous emendations to the paper.
	
	\section{Continuous model theory for tracial $\mathrm{W}^*$-algebras} \label{sec:modeltheory}
	
	This section sketches the setup of continuous model theory, or model theory for metric structures \cite{BYBHU2008,BYU2010} and its application to operator algebras in \cite{FHS2013,FHS2014,FHS2014b}.  We strive to present a self-contained exposition for two reasons:  First, some readers may not be familiar with the model-theoretic terminology.  Second, we are following the treatment in \cite{FHS2014} which introduces ``domains of quantification'' to cut down on the number of ``sorts,'' which means that some of the statements need to modified from their original form in \cite{BYBHU2008}.
	
	\subsection{Background on operator algebras} \label{subsec:operatoralgebrasbackground}
	
	We start by giving some basic terminology and background on operator algebras.  For further detail and history, we suggest consulting the references \cite{KadisonRingroseI,Dixmier1969,Sakai1971,TakesakiI,Blackadar2006,Zhu1993}.
	
	\textbf{$\mathrm{C}^*$-algebras:} 
	
	\begin{enumerate}[(1)]
		\item A (unital) algebra over $\C$ is a unital ring $A$ with a unital inclusion map $\C \to A$.
		\item A (unital) $*$-algebra is an algebra $A$ equipped with a conjugate linear involution $*$ such that $(ab)^* = b^* a^*$.
		\item A unital $\mathrm{C}^*$-algebra is a $*$-algebra $A$ equipped with a complete norm $\norm{\cdot}$ such that $\norm{ab} \leq \norm{a} \norm{b}$ and $\norm{a^*a} = \norm{a}^2$ for $a, b \in A$.
	\end{enumerate}
	
	A collection of fundamental results in $\mathrm{C}^*$-algebra theory establishes that $\mathrm{C}^*$-algebras can always be represented as algebras of operators on Hilbert spaces.  If $H$ is a Hilbert space, the algebra of bounded operators $B(H)$ is a $\mathrm{C}^*$-algebra.  Conversely, every unital $\mathrm{C}^*$-algebra can be embedded into $B(H)$ by some unital and isometric $*$-homomorphism $\rho$.  By isometric, we mean that $\norm{\rho(a)} = \norm{a}$, where $\norm{\rho(a)}$ is the operator norm on $B(H)$ and $\norm{a}$ is the given norm on the $\mathrm{C}^*$-algebra $A$.
	
	\textbf{$\mathrm{W}^*$-algebras:} A \emph{von Neumann algebra} is a $*$-subalgebra of $B(H)$ (for some Hilbert space $H$) that is closed in the \emph{strong operator topology}, the topology of pointwise convergence as functions on $H$.  A \emph{$\mathrm{W}^*$-algebra} is a $\mathrm{C}^*$-algebra that admits a predual (that is, it is the dual of some Banach space).  A deep result of Sakai showed that for a $\mathrm{C}^*$-algebra $A$ is a $\mathrm{W}^*$-algebra if and only if it is isomorphic to a von Neumann algebra; moreover, the weak-$\star$ topology on a $\mathrm{W}^*$-algebra $A$ is uniquely determined by its $\mathrm{C}^*$-algebra structure \cite[Corollary 1.13.3]{Sakai1971}.
	
	\textbf{Tracial $\mathrm{W}^*$-algebras:} A \emph{tracial $\mathrm{W}^*$-algebra} is a $\mathrm{W}^*$-algebra $M$ together with a linear map $\tau: M \to \C$ satisfying:
	\begin{itemize}
		\item \emph{positivity}: $\tau(x^*x) \geq 0$ for all $x \in M$
		\item \emph{unitality}: $\tau(1) = 1$
		\item \emph{traciality}: $\tau(xy) = \tau(yx)$ for $x, y \in A$
		\item \emph{faithfulness}:  $\tau(x^*x) = 0$ implies $x = 0$ for $x \in A$.
		\item \emph{weak-$\star$ continuity}:  $\tau: M \to \C$ is weak-$\star$ continuous.
	\end{itemize}
	We call $\tau$ a \emph{faithful normal tracial state}.
	
	\textbf{The standard representation:}  Given a tracial $\mathrm{W}^*$-algebra $(M,\tau)$, we can form a Hilbert space $L^2(M,\tau)$ as the completion of $M$ with respect to the inner product $\ip{x,y}_\tau = \tau(x^*y)$; if $x \in M$, then we denote the corresponding element of $L^2(M,\tau)$ by $\widehat{x}$.  There is a unique unital $*$-homomorphism $\pi_\tau: M \to B(L^2(M,\tau))$ satisfying $\pi_\tau(x) \widehat{y} = \widehat{xy}$ for $x, y \in M$.  Now $\pi_\tau$ is a $*$-homomorphism isometric with respect to the operator norm, and its image is a von Neumann algebra.  The construction of $L^2(M,\tau)$ and $\pi_\tau$ is a special case of the \emph{GNS (Gelfand-Naimark-Segal) construction} and $\pi_\tau$ is also known as \emph{the standard representation} of $(M,\tau)$.  Note the convergence of a net $x_i$ to $x$ in $M$ with respect to the strong operator topology in $B(L^2(M,\tau))$ implies convergence of $\widehat{x}_i = \pi_\tau(x_i) \widehat{1}$ to $\widehat{x} =\pi_\tau(x) \widehat{1}$ in $L^2(M,\tau)$.  (It turns out that the converse is true if $(x_i)_{i \in I}$ is bounded in operator norm, but we will not need to use this fact directly.)
	
	\textbf{$*$-polynomials and generators:} Given an index set $I$, we denote by $\C\ip{x_i,x_i^*: i \in I}$ the free unital algebra (or non-commutative polynomial algebra) generated by indeterminates $x_i$ and $x_i^*$ for $i \in I$.  We equip $\C\ip{x_i,x_i^*: i \in I}$ with the unique $*$-operation sending $x_i$ to $x_i^*$, thus making it into a $*$-algebra.  If $A$ is a unital $*$-algebra and $(a_i)_{i \in I}$ a collection of elements, there is a unique unital $*$-homomorphism $\rho: \C\ip{x_i,x_i^*: i \in I} \to A$ mapping $x_i$ to $a_i$ for each $i \in I$.  We refer to the elements of $\C\ip{x_i,x_i^*: i \in I}$ as \emph{non-commutative $*$-polynomials}, and if $p \in \C\ip{x_i,x_i^*: i \in I}$ and $\rho: \C\ip{x_i,x_i^*: i \in I} \to A$ is as above, we denote $\rho(p)$ by $p(a_i: i \in I)$.  Moreover, the image of $\rho$ is the \emph{$*$-algebra generated by} $(a_i)_{i \in I}$.
	
	If $A$ is a $\mathrm{C}^*$-algebra and $(a_i)_{i \in I}$ is a collection of elements of $I$, then the \emph{$\mathrm{C}^*$-algebra generated by $(a_i)_{i \in I}$} is the norm-closure of the $*$-algebra generated by $(a_i)_{i \in I}$.  Similarly, if $M$ is a von Neumann algebra and $(a_i)_{i \in I}$ is a collection of elements of $M$, then the \emph{von Neumann subalgebra} or \emph{$\mathrm{W}^*$-subalgebra generated by $(a_i)_{i \in I}$} is the strong operator topology closure of the $*$-algebra generated by $(a_i)_{i \in I}$.  In particular, we say that $(a_i)_{i \in I}$ \emph{generates $M$} if the strong operator topology closure is all of $M$.
	
	\subsection{Languages and structures}
	
	Next, let us sketch the setup of continuous model theory, or model theory for metric structures \cite{BYBHU2008,BYU2010}.  We will follow the treatment in \cite{FHS2014} which introduces ``domains of quantification'' to cut down on the number of ``sorts'' neeeded.
	
	A \emph{language} $\mathcal{L}$ consists of:
	\begin{itemize}
		\item A set $\mathcal{S}$ whose elements are called \emph{sorts}.
		\item For each $S \in \mathcal{S}$, a privileged relation symbol $d_S$ (which will represent a metric) and a set $\mathcal{D}_S$ whose elements are called \emph{domains of quantification for $S$}.
		\item For each $S \in \mathcal{S}$ and $D, D' \in \mathcal{D}_S$ an assigned constant $C_{D,D'}$.
		\item A countably infinite set of \emph{variable symbols} for each sort $S$.  We denote the variables by $(x_i)_{i \in \N}$.
		\item A set of \emph{function symbols}.
		\item For each function symbol $f$, an assigned tuple $(S_1,\dots,S_n)$ of sorts called the \emph{domain}, another sort $S$ called the \emph{codomain}.  We call $n$ the \emph{arity of $f$}.
		\item For each function symbol $f$ with domain $(S_1,\dots,S_n)$ and codomain $S$, and for every $\mathbf{D} = (D_1,\dots,D_n) \in \mathcal{D}_{S_1} \times \dots \times \mathcal{D}_{S_n}$, there is an assigned $D_{f,\mathbf{D}} \in \mathcal{D}_S$ (representing a range bound), and assigned moduli of continuity $\omega_{f,\mathbf{D},1}$, \dots, $\omega_{f,\mathcal{D},n}$. (Here ``modulus of coninuity'' means a continuous increasing, zero-preserving function $[0,\infty) \to [0,\infty)$).
		\item A set of \emph{relation symbols}.
		\item For each relation symbol $R$, an assigned domain $(S_1,\dots,S_n)$ as in the case of function symbols.
		\item For each relation symbol $R$ and for every $\mathbf{D} = (D_1,\dots,D_n) \in \mathcal{D}_{S_1} \times \dots \times \mathcal{D}_{S_n}$, an assigned bound $N_{R,\mathbf{D}} \in [0,\infty)$ and assigned moduli of continuity $\omega_{R,\mathbf{D},1}$, \dots, $\omega_{R,\mathcal{D},n}$.
	\end{itemize}
	
	Given a language $\mathcal{L}$, an \emph{$\mathcal{L}$-structure} $\cM$ assigns an object to each symbol in $\mathcal{L}$, called the \emph{interpretation} of that symbol, in the following manner:
	\begin{itemize}
		\item Each sort $S \in \mathcal{S}$ is assigned a metric space $S^{\cM}$, and the symbol $d_S$ is interpreted as the metric $d_S^{\cM}$ on $S^{\cM}$.
		\item Each domain of quantification $D \in \mathcal{D}_S$ is assigned a subset $D^{\cM} \subseteq S^{\cM}$, such that $D^{\cM}$ is complete for each $D$, $S^{\cM} = \bigcup_{D \in \mathcal{D}_S} D^{\cM}$, and $\sup_{X \in D, Y \in D'} d_S^{\cM}(X,Y) \leq C_{D,D'}$.
		\item Each function symbol $f$ with domain $(S_1,\dots,S_n)$ and codomain $S$ is interpreted as a function $f^{\mathcal{M}}: S_1^{\cM} \times \dots \times S_n^{\cM} \to S^{\cM}$.  Moreover, for each $\mathbf{D} = (D_1,\dots,D_n) \in \mathcal{D}_{S_1} \times \dots \times \mathcal{D}_{S_n}$, the function $f^{\cM}$ maps $D_1^{\cM} \times \dots \times D_n^{\cM}$ into $D_{f,\mathbf{D}}^{\cM}$.  Finally, $f^{\cM}$ restricted to $D_1^{\cM} \times \dots \times D_n^{\cM}$ is uniformly continuous in the $i$th variable with modulus of continuity of $\omega_{f,\mathbf{D},i}$.
		\item Each relation symbol $R$ with domain $(S_1,\dots,S_n)$ is interpreted as a function $R^{\cM}: S_1^{\cM} \times \dots \times S_n^{\cM} \to \R$.  Moreover, for each $\mathbf{D} = (D_1,\dots,D_n) \in \mathcal{D}_{S_1} \times \dots \times \mathcal{D}_{S_n}$, $f^{\cM}$ is bounded by $N_{R,\mathbf{D}}$ on $M(D_1) \times \dots \times M(D_n)$ and uniformly continuous in the $i$th argument with modulus of continuity of $\omega_{R,\mathbf{D},i}$.
	\end{itemize}
	
	The \emph{language $\mathcal{L}_{\tr}$ of tracial $\mathrm{W}^*$-algebras} can be described as follows.  We will also simultaneously describe how a tracial $\mathrm{W}^*$-algebra $(M,\tau)$ gives rise to an $\mathcal{L}_{\tr}$-structure $\mathcal{M}$, that is, how each symbol will be interpreted.
	\begin{itemize}
		\item A single sort, to be interpreted as the $\mathrm{W}^*$-algebra $M$.  If $\cM = (M,\tau)$ is a tracial $\mathrm{W}^*$-algebra, we denote the interpretation of this sort by $L^\infty(\cM)$ because of the intuition of tracial $\mathrm{W}^*$-algebras as non-commutative measure spaces.
		\item Domains of quantification $\{D_r\}_{r \in (0,\infty)}$, to be interpreted as the operator norm balls of radius $r$ in $M$.
		\item The metric symbol $d$, to be interpreted as the metric induced by $\norm{\cdot}_{2,\tau}$.
		\item A binary function symbol $+$, to be interpreted as addition.
		\item A binary function symbol $\cdot$, to be interpreted as multiplication.
		\item A unary function symbol $*$, to be interpreted as the adjoint operation.
		\item For each $\lambda \in \C$, a unary function symbol, to be interpreted as multiplication by $\lambda$.
		\item Function symbols of arity $0$ (in other words constants) $0$ and $1$, to be interpreted as additive and multiplicative identity elements.
		\item Two unary relation symbols $\re \tr$ and $\im \tr$, to be interpreted the real and imaginary parts of the trace $\tau$.
		\item For technical reasons explained in \cite{FHS2014}, we also introduce for each $d$-variable non-commutative polynomial $p$ a symbol $t_p: L^\infty(\cM)^d$ representing the evaluation of $p$, along with the appropriate range bounds $N_{t_p,\mathbf{r}}$ given by the supremum of $\norm{p(X_1,\dots,X_d)}$ over all $(X_1,\dots,X_d)$ in a tracial $\mathrm{W}^*$-algebra $\cM$.
	\end{itemize}
	Each function and relation symbol is assigned range bounds and moduli of continuity that one would expect, e.g.\ multiplication is supposed to map $D_r \times D_{r'}$ into $D_{rr'}$ with $\omega_{(D_r,D_r'),1}^{\cdot}(t) = r't$ and $\omega_{(D_r,D_r'),2}^{\cdot} = rt$.
	
	Although not every $\mathcal{L}_{\tr}$-structure comes from a tracial $\mathrm{W}^*$-algebra, one can formulate axioms in the language such that any structure satisfying these axioms comes from a tracial $\mathrm{W}^*$-algebra \cite[\S 3.2]{FHS2014}.  In order to state this result precisely, we first have to explain formulas and sentences.
	
	\subsection{Syntax: Terms, formulas, conditions, and sentences}
	
	\emph{Terms} in a language $\mathcal{L}$ are expressions obtained by iteratively composing the function symbols and variables.  For example, if $x_1$, $x_2$, \dots are variables in a sort $S$ and $f: S \times S \to S$ 
	and $g: S \times S \to S$ are function symbols, then $f(g(x_1,x_2),x_1)$ is a term.  Each term has assigned range bounds and moduli of continuity in each variable which are the natural ones computed from those of the individual function symbols making up the composition.  Any term $f$ with variables $x_1 \in S_1$, \dots, $x_k \in S_k$ and output in $S$ can be interpreted in an $\mathcal{L}$-structure as a function $S_1^{\cM} \times \dots \times S_k^{\cM} \to S^{\cM}$. For example, in the language $\mathcal{L}_{\tr}$, the terms are expressions obtained from iterating scalar multiplication, addition, multiplication, and the $*$-operation on variables and the unit symbol $1$.  If $(M,\tau)$ is a tracial $\mathrm{W}^*$-algebra, then the interpretation of a term in $\mathcal{M}$ is a function represented by a $*$-polynomial.
	
	\emph{Basic formulas} in a language are obtained by evaluating relation symbols on terms.  In other words, if $T_1$, \dots, $T_k$ are terms valued in sorts $S_1$, \dots, $S_k$, and $R$ is a relation $S_1 \times \dots \times S_k \to \R$, then $R(T_1,\dots,T_n)$ is a basic formula.  The basic formulas have assigned range bounds and moduli of continuity similar to the function symbols.  In an $\mathcal{L}$-structure $\cM$, a basic formula $\phi$ is interpreted as a function $\phi^{\cM}: S_1^{\cM} \times \dots \times S_k^{\cM} \to \R$.  In $\mathcal{L}_{\tr}$, a basic formula can take the form $\re \tr(f)$ or $\im \tr(f)$ where $f$ is an expression obtained by iterating the algebraic operations.  Thus, when evaluated in a tracial $\mathrm{W}^*$-algebra, it corresponds to the real or imaginary part of the trace of a non-commutative $*$-polynomial.
	
	\emph{Formulas} are obtained from basic formulas by iterating several operations:
	\begin{itemize}
		\item Given a formulas $\phi_1$, \dots, $\phi_n$ and $F: \R^n \to \R$ continuous, $F(\phi_1,\dots,\phi_n)$ is a formula.
		\item If $\phi$ is a formula, $D$ is a domain of quantification for some sort $S$, and $x$ is one of our variables in $S$, then $\inf_{x \in D} \phi$ and $\sup_{x \in D} \phi$ are formulas.
	\end{itemize}
	Each occurrence of a variable in $\phi$ is either \emph{bound} to a quanitifer $\sup_{x \in D}$ or $\inf_{x \in D}$, or else it is \emph{free}.  We will often write $\phi(x_1,\dots,x_n)$ for a formula to indicate that the free variables are $x_1$, \dots, $x_n$.
	
	All these formulas also have assigned range bounds and moduli of continuity. The moduli of continuity of $F(\phi_1,\dots,\phi_n)$ are obtained by composition from the moduli of continuity of $F$ and $\phi_j$ as in \cite[\S 2 appendix and Theorem 3.5]{BYBHU2008}.  Next, if $\phi: S_1 \times \dots \times S_n \to S$ and $D \in \mathcal{D}_{S_n}$
	\[
	\psi(x_1,\dots,x_{n-1}) = \sup_{x_n \in D} \phi(x_1,\dots,x_{n-1},x_n),
	\]
	then
	\[
	\omega_{\psi,(D_1,\dots,D_{n-1}),j} = \omega_{\phi,(D_1,\dots,D_{n-1},D),j}.
	\]
	Each formula has an interpretation in every $\mathcal{L}$-structure $\mathcal{M}$, defined by induction on the complexity of the formula.  If $\phi = F(\phi_1,\dots,\phi_n)$, then $\phi^{\cM} = F(\phi_1^{\cM}, \dots, \phi_n^{\cM})$.  Similarly, if $\psi(x_1,\dots,x_{n-1}) = \sup_{x_n \in D} \psi(x_1,\dots,x_n)$, then
	\[
	\psi^{\cM}(X_1,X_2,\dots,X_n) = \sup_{X_n \in D^{\cM}} \phi^{\cM}(X_1,\dots,X_{n-1}).
	\]
	Here $X_1$, \dots, $X_n$ are elements of the sorts in the $\mathcal{L}$-structure $\cM$, rather than formal variables.
	
	\begin{example}
		In $\mathcal{L}_{\tr}$, some terms are
		\[
		x_1 x_2, \quad (x_1 x_2) x_3 + (x_2^* x_3)(x_1 x_3^*)^*.
		\]
		A basic formula is
		\[
		\re \tr(x_1 x_2 + x_3^*(x_2x_1)^*).
		\]
		Another formula is
		\[
		\re \tr(x_1 x_2) + e^{\im \tr(x_1^*(x_2 x_3^*) \re \tr(x_4)}.
		\]
		We can also write a formula
		\[
		\sup_{x_1 \in D_2} [\re \tr(x_1 x_2) + e^{\im \tr(x_1^*(x_2 x_3^*) \re \tr(x_4)}],
		\]
		which will be interpreted as the supremum of the previous formula over $x_1$ in the ball of radius $2$.  In this formula, $x_1$ is bound to the quantifier $\sup_{x_1 \in D_2}$ and the variables $x_2$ and $x_3$ are free.
	\end{example}
	
	For convenience, we will assume that our formulas do not have two copies of the same variable (i.e. if a variable is bound to a quantifier, there is no other variable of the same name that is free or bound to a different quantifier).  For instance, in the formula
	\[
	\im \tr(x_1) \sup_{x_1 \in D_1} \re \tr(x_1x_2 + x_3 x_1^*),
	\]
	the first occurrence of $x_1$ is free while the latter two occurrences are bound to the quantifier $\sup_{x_1 \in D_1}$, but we can rewrite this formula equivalently as
	\[
	\im \tr(x_1) \sup_{y_1 \in D_1} \re \tr(y_1x_2 + x_3 y_1^*).
	\]
	We will typically denote the free variables by $(x_i)_{i \in \N}$ and the bound variables by $(y_i)_{i \in \N}$.  Lowercase letters will be used for formal variables while uppercase letters will be used for individual operators in operator algebras (or more generally elements of an $
	\cL$-structure).
	
	
	\subsection{Theories, models, and axioms}
	
	A \emph{sentence} is a formula with no free variables.  If $\phi$ is a sentence, then the interpretation $\phi^{\cM}$ in an $\mathcal{L}$-structure is simply a real number.
	
	A \emph{theory} $\rT$ in a language $\mathcal{L}$ is a set of sentences.  We say that an $\mathcal{L}$-structure $\mathcal{M}$ \emph{models} the theory $\rT$, or $\mathcal{M} \models \rT$ if $\phi^{\cM} = 0$ for all $\phi \in \rT$.
	
	If $\mathcal{M}$ is an $\mathcal{L}$-structure, then the \emph{theory of $\cM$}, denoted $\Th(\mathcal{M})$ is the set of sentences $\phi$ such that $\phi^{\cM} = 0$.  As observed in \cite{FHS2014}, the theory of $\cM$ also uniquely determines the values $\phi^{\cM}$ of all sentences $\phi$ since $\phi - c$ is a sentence for every constant $c \in \R$.
	
	More generally, if $\mathcal{C}$ is a class of $\mathcal{L}$-structures, then $\Th(\cC)$ is the set of all sentences $\phi$ such that $\phi^{\cM} = 0$ for all $\cM$ in $\cC$.  The class $\cC$ is said to be \emph{axiomatizable} if every $\mathcal{L}$-structure that models $\Th(\cC)$ is actually in $\cC$.
	
	It is shown in \cite[\S 3.2]{FHS2014} that the class of $\cL_{\tr}$-structures that represent actual tracial $\mathrm{W}^*$-algebras is axiomatizable.  The axioms, roughly speaking, encode the fact that $\cM$ is a $*$-algebra, the fact that $\tau$ is a tracial state, the fact that $\norm{xy}_{L^2(\cM)} \leq r \norm{y}_{L^2(\cM)}$ for $x \in D_r^{\cM}$, the relationship between the distance and the trace, the fact that $D_r^{\cM}$ is contained in $D_{r'}^{\cM}$ for $r < r'$ (that is, $\sup_{x \in D_r} \inf_{y \in D_{r'}} d(x,y) = 0$), and the fact that $t_p$ agrees with the evaluation of the non-commutative polynomial $p$.
	
	The theory of tracial $\mathrm{W}^*$-algebras will be denoted $\rT_{\tr}$.  It is also shown in \cite{FHS2014} that $\mathrm{II}_1$ factors (infinite-dimensional tracial $\mathrm{W}^*$-algebras with trivial center) are axiomatizable by a theory $\rT_{\mathrm{II}_1}$.
	
	\subsection{Ultraproducts}
	
	An important construction for continuous model theory and for $\mathrm{W}^*$-algebras is the ultraproduct.  Ultraproducts are a way of constructing a limiting object out of arbitrary sequences (or more generally indexed families) of objects.  In order to force limits to exist, one uses a device called an ultrafilter.
	
	Let $I$ be an index set.  An \emph{ultrafilter} $\cU$ on $I$ is a collection of subsets of $I$ such that
	\begin{itemize}
		\item $\varnothing \not \in \cU$.
		\item If $A \subseteq B \subseteq I$ and $A \in \cU$, then $B \in \cU$.
		\item If $A$, $B \in \cU$, then $A \cap B \in \cU$.
		\item For each $A \subseteq I$, either $A \in \cU$ or $A^c \in \cU$.
	\end{itemize}
	If $\cU$ is an ultrafilter on $I$, $\Omega$ is a topological space, and $f: I \to \Omega$ is a function, then we say that
	\[
	\lim_{i \to \cU} f(i) = w
	\]
	if for every neighborhood $O \ni w$ in $\Omega$, the preimage $f^{-1}(O)$ is an element of $\cU$.  Now if $i \in I$, there is an ultrafilter $\cU_i := \{A \subseteq I: i \in A\}$, which is called a \emph{principal ultrafilter}.  All other ultrafilters are called \emph{non-principal} or \emph{free ultrafilters}.
	
	The set of all ultrafilters on $I$ can be identified with the Stone-{\v C}ech compactification of $I$, where $I$ is given the discrete topology (see e.g.\ \cite[\S 3]{HS2011}).  The principal ultrafilters correspond to the points of the original space $I$.  In particular, this means that if $\Omega$ is a compact Hausdorff space and $f: I \to \Omega$ is a function, then $\lim_{i \to \cU} f(i)$ exists in $\Omega$.
	
	Now consider a language $\cL$.  Let $I$ be an index set, $\cU$ an ultrafilter on $I$, and for each $i \in I$, let $\cM$ be an $\cL$-structure.  The ultraproduct $\prod_{i \to \cU} \cM_i$ is the $\cL$-structure $\cM$ defined as follows (see \cite[\S 5]{BYBHU2008}):  For each sort $S$, consider tuples $(X_i)_{i \in I}$ where $X_i \in S^{\cM_i}$.
	\begin{itemize}
		\item Let's call $(X_i)_{i \in I}$ \emph{confined} if there exists $D \in \cD_S$ such that $X_i \in D^{\cM_i}$ for all $i$.
		\item Let's call $(X_i)_{i \in I}$ and $(Y_i)_{i \in I}$ \emph{equivalent} if $\lim_{i \to \cU} d_S^{\cM_i}(X_i,Y_i) = 0$.
		\item For a confined tuple $(X_i)_{i \in I}$, let $[X_i]_{i \in I}$ denote its equivalence class.
	\end{itemize}
	We define $S^{\cM}$ to be the set of equivalence classes of confined tuples $(X_i)_{i \in I}$.  The metric $d_S^{\cM}$ on $S^{\cM}$ is then given by
	\[
	d_S^{\cM}([X_i]_{i \in I}, [Y_i]_{i \in I}) = \lim_{i \to \cU} d_S^{\cM_i}(X_i,Y_i).
	\]
	This is independent of the choice of representative for the equivalence classes because of the triangle inequality, and it is finite because if $X_i \in D^{\cM_i}$ and $Y_i \in (D')^{\cM_i}$ for all $i$, then $d_S^{\cM_i}(X_i,Y_i) \leq C_{D,D'}$.  Then $S^{\cM}$ is a metric space and $S^{\cM} = \bigcup_{D \in \cD_S} D^{\cM}$, where $D^{\cM}$ is the set of classes $[X_i]_{i \in I}$ with $X_i \in D^{\cM_i}$ for all $i$.  Moreover, $D^{\cM}$ is automatically complete \cite[Proposition 5.3]{BYBHU2008}.
	
	Each function symbol $f: S_1 \times \dots \times S_n \to S$ receives its interpretation $f^{\cM}$ through
	\[
	f^{\cM}([X_{1,i}]_{i \in I},\dots,[X_{n,i}]_{i \in I}) = [f^{\cM_i}(X_{1,i},\dots,X_{n,i})]_{i \in I},
	\]
	which is well-defined because of the uniform continuity of $f$ on each domain of quantification, and similarly, each relation receives its interpretation in $\cM$ through
	\[
	R^{\cM}([X_{1,i}]_{i \in I},\dots,[X_{n,i}]_{i \in I}) = \lim_{i \to \cU} R^{\cM_i}(X_{1,i},\dots,X_{n,i}).
	\]
	One can verify by the same reasoning as \cite[\S 5]{BYBHU2008} that $\cM$ is indeed an $\cL$-structure.
	
	One of the reasons ultraproducts are so important is because of the following result, known as (the continuous analog of) \L os's theorem.  See \cite[Theorem 5.4]{BYBHU2008}.
	
	\begin{theorem} \label{thm:Los}
		Let $\cL$ be a language, $\cU$ an ultrafilter on an index set $I$, and $\cM_i$ an $\cL$-structure for each $i \in I$.  Let $\cM = \prod_{i \to \cU} \cM_i$.  If $\phi$ is a formula with free variables $x_1 \in S_1$, \dots, $x_n \in S_n$, then for any $[X_{1,i}]_{i \in I} \in S_1^{\cM}$, \dots, $[X_{n,i}]_{i \in I} \in S_n^{\cM}$, we have
		\[
		\phi^{\cM}([X_{1,i}]_{i \in I}, \dots, [X_{n,i}]_{i \in I}) = \lim_{i \to \cU} \phi^{\cM_i}(X_{1,i},\dots,X_{n,i}).
		\]
	\end{theorem}
	
	\begin{corollary}
		In the situation of the previous theorem, if $\rT$ is an $\cL$-theory, and if $\cM_i \models \rT$ for all $i$, then $\cM \models \rT$.
	\end{corollary}
	
	In particular, this shows that an ultraproduct of $\cL_{\tr}$-structures that are tracial $\mathrm{W}^*$-algebras will also be an $\cL_{\tr}$-structure that is a tracial $\mathrm{W}^*$-algebra.  One can verify that the model-theoretic ultraproduct agrees in this case with the ultraproduct of tracial $\mathrm{W}^*$-algebras.
	
	\section{Definable predicates and functions} \label{sec:predicates}
	
	This section describes types, definable predicates, and definable functions.  The material in \S \ref{subsec:types} - \S \ref{subsec:qftypes} is largely a mixture of folklore and adaptations of \cite{BYBHU2008}; our main contribution is to write down the results in the setting of \emph{infinite tuples} and \emph{domains of quantification}.  In \S \ref{subsec:qfdefinablefunctions}, we give a characterization of quantifier-free definable functions in $\cL_{\tr}$ based on \cite[\S 13]{JekelThesis} and \cite[\S 2]{HJNS2021}.
	
	\subsection{Types} \label{subsec:types}
	
	\begin{definition}
		Let $\mathbf{S} = (S_j)_{j \in \N}$ be an $\N$-tuple of sorts in $\mathcal{L}$.  Let $\cF_{\mathbf{S}}$ be the space of $\cL$-formulas with free variables $(x_j)_{j \in \N}$ with $x_j$ from the sort $S_j$.  If $\cM$ is an $\cL$-structure and $\mathbf{X} \in \prod_{j \in \N} S_j^{\cM}$, then the \emph{type} of $\mathbf{X}$ is the map
		\[
		\tp^{\cM}(\mathbf{X}): \quad \cF_{\mathbf{S}} \to \R: \quad \phi \mapsto \phi^{\cM}(\mathbf{X}).
		\]
	\end{definition}
	
	\begin{definition}
		Let $\mathbf{S} = (S_j)_{j \in \N}$ be an $\N$-tuple of sorts in $\mathcal{L}$, and let $\rT$ be an $\cL$-theory.  If $\mathbf{D} \in \prod_{j \in \N} \mathcal{D}_{S_j}$, then we denote by $\mathbb{S}_{\bD}(\rT)$ the set of types $\tp^{\cM}(\bX)$ of all $\bX \in \prod_{j \in \N} D_j^{\cM}$ for all $\cM \models \rT$.
	\end{definition}
	
	\begin{definition}
		If $\bS$ is an $\N$-tuple of $\cL$-sorts, the set $\cF_{\mathbb{S}}$ of formulas defines a real vector space.  For each $\cL$-structure $\cM$ and $\bX \in \prod_{j \in \N} S_j^{\cM}$, the type $\tp^{\cM}(\bX)$ is a (real) linear map $\cF_{\bS} \to \R$.  Thus, for each $\cL$-theory $\rT$ and $\bD \in \prod_{j \in \N} \cD_{S_j}$, the space $\mathbb{S}_{\bD}(\rT)$ is a subset of the dual $\cF_{\mathbb{S}}^\dagger$.  We equip $\mathbb{S}_{\bD}(\rT)$ with the weak-$\star$ topology (also known as the \emph{logic topology}).
	\end{definition}
	
	The following observation is well known; see \cite[Corollary 5.12, Proposition 8.6]{BYBHU2008}
	
	\begin{observation}
		$\mathbb{S}_{\bD}(\rT)$ is compact in the weak-$\star$ topology.
	\end{observation}
	
	\begin{proof}
		Each formula $\phi$ has a range bound $N_{\phi,\bD}$ such that $|\phi^{\cM}(\bX)| \leq N_{\phi,\bD}$ for all $\cL$-structures $\cM$ and all $\bX \in \prod_{j \in \N} D_j$.  Thus, $\mathbb{S}_{\bD}(\rT)$ is a subset of $\prod_{\phi \in \cF_{\cD}} [-N_{\phi,\bD},N_{\phi,\bD}]$ with the product topology, which is compact by Tychonoff's theorem.
		
		Moreover, $\mathbb{S}_{\bD}(\rT)$ is a closed subset.  While closedness can be expressed in terms of nets, it can also be expressed in terms of ultralimits.  A set $A$ is closed if and only if for every $I$ and $f: I \to A$ and ultrafilter $\cU$, the limit $\lim_{i \to \cU} f(i)$ exists in $A$.  It then follows from Theorem \ref{thm:Los} that if $[X_i]_{i \in I}$ is an element of an ultraproduct $\cM$ of $\cL$-structures $\cM_i$, then $\tp^{\cM}([X_i]_{i \in I}) = \lim_{i \to \cU} \tp^{\cM_i}(X_i)$.
	\end{proof}
	
	Although many times authors choose to work with $\mathbb{S}_{\bD}(\rT)$ for each $\bD$, we find it convenient to specify a topology on the entire space of types $\mathbb{S}_{\bS}(\rT)$ that extends the topology on each $\mathbb{S}_{\bD}(\rT)$, so that our later results can be stated about $\mathbb{S}_{\bS}(\rT)$ globally.  The topology on $\mathbb{S}_{\bD}(\rT)$ is given by a categorical colimit of the topologies on $\mathbb{S}_{\bD}(\rT)$.
	
	\begin{definition}
		For a language $\cL$, tuple $\bS$ of sorts, and theory $\rT$, let $\mathbb{S}_{\bS}(\rT)$ denote the space of $\bS$-types for all $\cM \models \rT$.  Note that $\mathbb{S}_{\bS}(\rT)$ is the union of all $\mathbb{S}_{\bD}(\rT)$ for all $\bD \in \prod_{j \in \N} \cD_{S_j}$.  We say that $\cO \subseteq \mathbb{S}_{\bS}(\rT)$ is open if $\cO \cap \mathbb{S}_{\bD}(\rT)$ is open for every $\bD \in \prod_{j \in \N} \cD_{S_j}$; this defines a topology on $\mathbb{S}_{\bD}(\rT)$, which we will also call the \emph{logic topology}.
	\end{definition}
	
	\begin{observation}
		For a language $\cL$, tuple $\bS$ of sorts, theory $\rT$, and $\bD \in \prod_{j \in \N} \cD_{S_j}$, the inclusion map $\mathbb{S}_{\bD}(\rT) \to \mathbb{S}_{\bS}(\rT)$ is a topological embedding.
	\end{observation}
	
	\begin{proof}
		Note that $\mathbb{S}_{\bS}(\rT)$ is Hausdorff; indeed, $\mu$ and $\nu$ are two distinct types, then there exists a formula $\phi$ with $\mu(\phi) \neq \nu(\phi)$.  One can check that the sets
		\begin{align*}
			U &= \{\sigma \in \mathbb{S}_{\bS}(\rT): |\sigma(\phi) - \mu(\phi)| < |\sigma(\phi) - \nu(\phi)|\}, \\
			V &= \{\sigma \in \mathbb{S}_{\bS}(\rT): |\sigma(\phi) - \nu(\phi)| < |\sigma(\phi) - \mu(\phi)|\}
		\end{align*}
		are open and they separate $\mu$ and $\nu$.
		
		Continuity of the inclusion map $\mathbb{S}_{\bD}(\rT) \to \mathbb{S}_{\bS}(\rT)$ follows from the definition of open sets in $\mathbb{S}_{\bD}(\rT)$.  Then since $\mathbb{S}_{\bD}(\rT)$ is compact and $\mathbb{S}_{\bS}(\rT)$ is Hausdorff, the map is a topological embedding.
	\end{proof}
	
	\begin{observation} \label{obs:typecontinuity}
		For a language $\cL$, tuple $\bS$ of sorts, theory $\rT$, and topological space $\Omega$, a function $\psi: \mathbb{S}_{\bS}(\rT) \to \Omega$ is continuous if and only if $\psi|_{\mathbb{S}_{\bD}(\rT)}$ is continuous for every $\bD \in \prod_{j \in \N} \cD_{S_j}$.
	\end{observation}
	
	\begin{proof}
		This follows from the definition of open sets in $\mathbb{S}_{\bS}(\rT)$.
	\end{proof}
	
	\subsection{Definable predicates}
	
	Next, we describe definable predicates, which are certain limits of formulas.  It will turn out that definable predicates correspond precisely to continuous functions $\mathbb{S}_{\bS}(\rT) \to \R$, and thus they are a natural completion of the space of formulas in the setting of continuous model theory.  Our approach to the definition will be semantic rather than syntactic, defining these objects immediately in terms of their interpretations.
	
	\begin{definition} \label{def:definablepredicate}
		Let $\mathcal{L}$ be a language and $\rT$ an $\mathcal{L}$-theory.  A \emph{definable predicate relative to $\rT$} is a collection of functions $\phi^{\cM}: \prod_{j \in \N} S_j^{\cM} \to \R$ (for each $\cM \models \rT$) such that for every collection of domains $\mathbf{D} = (D_j)_{j \in \N}$ and every $\epsilon > 0$, there exists a finite $F \subseteq \N$ and an $\mathcal{L}$-formula $\psi(x_j: j \in F)$ such that whenever $\cM \models \rT$ and $\mathbf{X} \in \prod_{j \in \N} D_j^{\cM}$, we have
		\[
		|\phi^{\cM}(\mathbf{X}) - \psi^{\cM}(X_j: j \in F)| < \epsilon.
		\]
	\end{definition}
	
	In other words a definable predicate is an object that can be uniformly approximated by a formula on any product of domains of quantification, where the approximation works uniformly for all models of the theory $\rT$.  This is done relative to $\rT$ because, for instance, in the study of tracial $\mathrm{W}^*$-algebras we do not care if the definable predicate makes sense to evaluate on arbitrary $\mathcal{L}_{\tr}$-structures, only those which actually come from tracial $\mathrm{W}^*$-algebras.
	
	Note that every formula defines a definable predicate.  However, two formulas as defined in the previous section (where the range bounds and moduli of continuity are part of the definition) may reduce to the same definable predicate (especially given the restriction that we work relative to a given theory $\rT$).
	
	The next proposition describes definable predicates as continuous functions on the space of types.   This is an adaptation of \cite[Theorem 9.9]{BYBHU2008} to the setting with domains of quantification.
	
	\begin{proposition} \label{prop:defpredcontinuous}
		Let $\cL$ be a language and $\rT$ an $\cL$-theory.  Let $\phi$ be a collection of functions $\phi^{\cM}: \prod_{j \in \N} S_j^{\cM} \to \R$ for each $\cM \models \rT$.  The following are equivalent:
		\begin{enumerate}[(1)]
			\item $\phi$ is a definable predicate relative to $\rT$.
			\item There exists a continuous $\gamma: \mathbb{S}_{\bS}(\rT) \to \R$ such that $\phi^{\cM}(\bX) = \gamma(\tp^{\cM}(\bX))$ for all $\cM \models \rT$ and $\bX \in \prod_{j \in \N} S_j^{\cM}$.
		\end{enumerate}
	\end{proposition}
	
	\begin{proof}
		(1) $\implies$ (2).  First, suppose that $\phi$ is a formula.  Then by definition of type, $\phi^{\cM}(\bX)$ only depends on the type of $\bX$ in $\cM$, and hence $\phi^{\cM}(\bX) = \gamma(\tp^{\cM}(\bX))$ for some $\gamma: \mathbb{S}_{\bS}(\rT) \to \R$.  For each $\bD \in \prod_{j \in \N} \cD_{S_j}$, the restriction of $\gamma$ to a map $\mathbb{S}_{\bD}(\rT) \to \R$ is continuous by definition of the weak-$\star$ topology.  Hence, by Observation \ref{obs:typecontinuity}, $\gamma$ is a continuous function $\mathbb{S}_{\bS}(\rT) \to \R$.
		
		Now let $\phi$ be a general definable predicate.  Fix $\bD \in \prod_{j \in \N} \cD_{S_j}$.  Then taking $\epsilon = 1/n$ in Definition \ref{def:definablepredicate}, there exists a formula $\phi_{\bD,n}$ depending on finitely many of the variables $x_j$, such that
		\begin{equation} \label{eq:defpredapprox}
			|\phi^{\cM}(\mathbf{X}) - \phi_{\bD,n}^{\cM}(\mathbf{X})| < \frac{1}{n}
		\end{equation}
		for all $\cM \models \rT$ and $\bX \in \prod_{j \in \N} D_j^{\cM}$.  By the previous paragraph, there exists a continuous $\gamma_{\bD,n}: \mathbb{S}_{\bD}(\rT) \to \R$ such that $\phi_{\bD,n}^{\cM}(\bX) = \gamma_{\bD,n}(\tp^{\cM}(\bX))$ for all $\cM \models \rT$ and $\bX \in \prod_{j \in \N} D_j^{\cM}$.  By \eqref{eq:defpredapprox},
		\[
		\sup_{\mu \in \mathbb{S}_{\bD}(\rT)} |\gamma_{\bD,n}(\mu) - \gamma_{\bD,m}(\mu)| \leq \frac{1}{n} + \frac{1}{m},
		\]
		which implies that the sequence $\gamma_{\bD,n}$ converges as $n \to \infty$ to a continuous $\gamma_{\bD}: \mathbb{S}_{\bD}(\rT) \to \R$.  Also, by \eqref{eq:defpredapprox},
		\[
		\phi^{\cM}(\bX) = \gamma_{\bD}(\tp^{\cM}(\bX)))
		\]
		for $\cM \models \rT$ and $\bX \in \prod_{j \in \N} D_j^{\cM}$.  This in turn implies that $\gamma_{\bD}$ and $\gamma_{\bD'}$ agree on $\mathbb{S}_{\bD}(\rT) \cap \mathbb{S}_{\bD'}(\rT)$ for any $\bD$ and $\bD' \in \prod_{j \in \N} \cD_{S_j}$.  Thus, for some function $\gamma: \mathbb{S}_{\bS}(\rT) \to \R$, we have $\gamma_{\bD} = \gamma|_{\bD}$ for $\bD \in \prod_{j \in \N} \cD_{S_j}$.   By Observation \ref{obs:typecontinuity}, $\gamma$ is continuous on $\mathbb{S}_{\bS}(\rT)$.
		
		(2) $\implies$ (1).  Assume there exists $\gamma: \mathbb{S}_{\bS}(\rT) \to \R$ continuous such that $\phi^{\cM}(\bX) = \gamma(\tp^{\cM}(\bX))$ for all $\cM \models \rT$ and $\bX \in \prod_{j \in \N} S_j^{\cM}$.  Fix $\bD \in \prod_{j \in \N} \cD_{S_j}$.  Let $\cA$ be the set of functions $\mathbb{S}_{\bD}(\rT) \to \R$ given by the application of formulas $\phi \in \cF_{\bS}$.  Then $\cA$ is a subalgebra of $C(\mathbb{S}_{\bD}(\rT),\R)$ since formulas are closed under sums, products, and scalar multiplication by real numbers.  Moreover, $\cA$ separates points because by definition two types are the same if they agree on all formulas.  Therefore, since $\gamma|_{\mathbb{S}_{\bD}(\rT)}$ is continuous, the Stone-Weierstrass theorem implies that there exists a formula $\psi$ depending on finitely many of the variables $x_j$ such that $|\phi^{\cM}(\mathbf{X}) - \psi^{\cM}(\bX)| < \epsilon$ whenever $\cM \models \rT$ and $\mathbf{X} \in \prod_{j \in \N} D_j^{\cM}$.
	\end{proof}
	
	\begin{lemma} \label{lem:defpredweakstar}
		If $\bS$ is an $\N$-tuple of types and $\bD \in \prod_{j \in \N} \cD_{S_j}$, then the logic topology on $\mathbb{S}_{\bD}(\rT)$ agrees with the weak-$\star$ topology obtained by viewing $\mathbb{S}_{\bD}(\rT)$ as a subspace of the dual of the vector space of definable predicates.
	\end{lemma}
	
	\begin{proof}
		We defined the logic topology as the weak-$\star$ topology generated by the pairing of types with formulas in variables $x_j \in S_j$ for $j \in \N$.  Since every formula gives a definable predicate, the weak-$\star$ topology obtained from the pairing with definable predicates is at least as strong as the logic topology.  On other hand, for each $\bD \in \prod_{j \in \N} \cD_{S_j}$, every definable predicate can be approximated uniformly by $\prod_{j \in \N} D_j^{\cM}$ for all $\cM \models \rT$, and hence the pairing with each definable predicate $\phi$ defines a map $\mathbb{S}_{\bD}(\rT) \to \R$ that is continuous with respect to the logic topology, and hence the logic topology is at least as strong as the weak-$\star$ topology obtained from pairing with definable predicates.
	\end{proof}
	
	Just like formulas, definable predicates are uniformly continuous on any product of domains of quantification.  But to say this properly, we should clarify what ``uniform continuity'' means for a function of infinitely many variables. If $\Omega_j$ is a metric space, then $\prod_{j \in \N} \Omega_j$ with the product topology is metrizable but without a canonical choice of metric.  However, we will say that $\phi: \prod_{j \in \N} \Omega_j \to \R$ is uniformly continuous if for every $\epsilon > 0$, there exists a finite $F \subseteq \N$ and $\delta > 0$, such that
	\[
	d_j(x_j,y_j) < \delta \text{ for } j \in F \implies |\phi(x) - \phi(y)| < \epsilon.
	\]
	In other words, uniform continuity is defined with respect to the product \emph{uniform structure} on $\prod_{j \in \N} \Omega_j$ (see for instance \cite{James1990} for background on uniform structures).
	
	\begin{observation} \label{obs:defpredunifcont}
		If $\phi = (\phi^{\cM})$ is a definable predicate over $\mathcal{L}$ relative to $\rT$, then $\phi$ satisfies the following uniform continuity property:
		
		For every $\mathbf{D} \in \prod_{j \in \N} \mathcal{D}_{S_j}$ and $\epsilon > 0$, there exists a finite $F \subseteq \N$ and $\delta > 0$ such that, for every $\mathcal{M} \models \rT$ and $\mathbf{X}, \mathbf{Y} \in \prod_{j \in \N} D_j^{\cM}$,
		\[
		d^{\cM}(X_j,Y_j) < \delta \text{ for all } j \in F \implies |\phi^{\cM}(\mathbf{X}) - \phi^{\cM}(\mathbf{Y})| < \epsilon.
		\]
		
		Moreover, for every $\mathbf{D} \in \prod_{j \in \N} \mathcal{D}_{S_j}$, there exists a constant $C$ such that $|\phi^{\cM}| \leq C$ for all $\cM \models \rT$.
	\end{observation}
	
	By construction, this result holds for formulas in finitely many $X_j$'s, and it holds for general definable predicates by the principle that uniform continuity and boundedness are preserved under uniform limits.
	
	Another useful property is that definable predicates are closed under the same types of operations as formulas.  In fact, we can use infinitary rather than finitary operations.  Point (1) here is an adaptation of \cite[Proposition 9.3]{BYBHU2008}.
	
	\begin{lemma} \label{lem:defpredoperations} ~
		\begin{enumerate}[(1)]
			\item If $F: \R^{\N} \to \R$ is continuous (where $\R^{\N}$ has the product topology) and $(\phi_j)_{j \in \N}$ are definable predicates $\prod_{j \in \N} S_j \to \R$ in $\mathcal{L}$ relative to $\rT$, then $F((\phi_j)_{j\in\N})$ is a definable predicate.
			\item If $\phi$ is a definable predicate $\prod_{j \in \N} S_j \times \prod_{j \in \N} S_j' \to \R$ in $\mathcal{L}$ relative to $\rT$ and $\mathbf{D}' \in \prod_{j \in \N} \cD_{S_j'}$, then
			\[
			\psi^{\cM}(\mathbf{X},\mathbf{Y}) := \inf_{\mathbf{Y} \in \prod_{j \in \N} (D_j')^{\cM}} \phi(\mathbf{X},\mathbf{Y})
			\]
			is also definable predicate in $\mathcal{L}$ relative to $\rT$.
		\end{enumerate}
	\end{lemma}
	
	\begin{proof}~
		\begin{enumerate}[(1)]
			\item This follows from \ref{prop:defpredcontinuous} and the fact that continuity is preserved by composition.
			\item Fix $\bD \in \prod_{j \in \N} \cD_{S_j}$ and $\epsilon > 0$.  Then there exist a formula $\phi_0$ whose free variables are a finite subset of the $x_j$'s and $y_j$'s, such that $|\phi^{\cM} - \phi_0^{\cM}| < \epsilon$ on $\prod_{j \in \N} D_j^{\cM} \times \prod_{j \in \N} (D_j')^{\cM}$ for all $\cM \models \rT$.  Note that
			\[
			\psi_0^{\cM}(\mathbf{X}) = \inf_{\mathbf{Y} \in \prod_{j \in \N} D_j^{\cM}} \phi_0^{\cM}(\mathbf{X},\mathbf{Y})
			\]
			also defines a formula because the infimum is effectively over only finitely many $Y_j$'s.  Also, $|\psi_0^{\cM} - \psi^{\cM}| \leq \epsilon$ on $\prod_{j \in \N} D_j^{\cM}$ for all $\cM \models \rT$.  Therefore, $\psi$ is a definable predicate. \qedhere
		\end{enumerate}
	\end{proof}
	
	We conclude with a brief remark on separability since we will use the separability of $\cL_{\tr}$ in the sequel.  For a $\cL$-theory $\rT$, we equip $C(\mathbb{S}_{\bS}(\rT))$ with the locally convex topology generated by the family of seminorms
	\[
	\phi \mapsto \norm{\phi|_{\mathbb{S}_{\bD}(\rT)}}_{C(\mathbb{S}_{\bD}(\rT))}
	\]
	for $\bD \in \prod_{j \in \N} \cD_{S_j}$.  In other words, a net $\phi_i$ to $\phi$ converges in this topology if and only if $\norm{(\phi - \phi_i)|_{\mathbb{S}_{\bD}(\rT)}}_{C(\mathbb{S}_{\bD}(\rT))} \to 0$ for all $\bD$.
	
	\begin{definition}
		A language $\cL$ is \emph{separable} if
		\begin{enumerate}[(1)]
			\item $\cL$ has countably many sorts.
			\item For every $\N$-tuple $\mathbb{S}$ of sorts, the space $C(\mathbb{S}_{\bS}(\varnothing))$ is separable, where $\varnothing$ denotes the empty theory.
		\end{enumerate}
	\end{definition}
	
	\begin{observation} \label{obs:weak*metrizable}
		If $\cL$ is a separable language, $\rT$ is an $\cL$-theory, and $\bD$ is an $\N$-tuple of domains from an $\N$-tuple of sorts $\bS$, then $\mathbb{S}_{\bD}(\rT)$ is metrizable.
	\end{observation}
	
	\begin{proof}
		By separability of $\cL$, there is a dense sequence $(\phi_n)_{n \in \N}$ in $C(\mathbb{S}_{\bS}(\varnothing))$.  Since the restriction maps $C(\mathbb{S}_{\bS}(\varnothing)) \to C(\mathbb{S}_{\bD}(\varnothing))$ and $C(\mathbb{S}_{\bD}(\varnothing)) \to C(\mathbb{S}_{\bD}(\rT))$ are continuous, $(\phi_n)$ also defines a dense subset in $C(\mathbb{S}_{\bD}(\rT))$.  For each $n$, there exists a constant $K_n$ such that $|\mu(\phi_n)| \leq K_n$ for all $\mu \in \mathbb{S}_{\bD}(\rT)$; this holds because $\phi_n$ can be uniformly approximated on $\bD$ by formulas, which are also uniformly bounded.  Then we may define a metric on $\mathbb{S}_{\bD}(\rT)$ by
		\[
		d(\mu,\nu) = \sum_{n \in \N} \frac{1}{2^n K_n} |\mu(\phi_n) - \nu(\phi_n)|.
		\]
		The verification that this induces the weak-$*$ topology is routine.  The types $\mathbb{S}_{\bD}(\rT)$ induce linear functionals on $C(\mathbf{S}_{\bD}(\rT))$, or in other words, $\mathbb{S}_{\bD}(\rT)$ is contained in the unit ball of the dual of $C(\mathbb{S}_{\bD}(\rT))$ so convergence on a dense subset of $C(\mathbb{S}_{\bD}(\rT))$ is equivalent to convergence on all of $C(\mathbb{S}_{\bD}(\rT))$.
	\end{proof}
	
	\begin{observation} \label{obs:Ltrseparable}
		The language $\cL_{\tr}$ of tracial $\mathrm{W}^*$-algebras is separable.
	\end{observation}
	
	\begin{proof}
		Consider $\cL_{\tr}$ formulas obtained using only scalar multiplication by numbers in $\Q[i]$ rather than $\C$ and using only suprema and infima over $D_r$ for $r \in \Q \cap (0,\infty)$.  There are only countably many such formulas, and one can show that these formulas are dense in the space of definable predicates.
	\end{proof}
	
	\subsection{Definable functions}
	
	Although definable functions are often defined only for finite tuples, it is useful for the theory of covering entropy to work with infinite tuples as both the input and the output functions.  The following ``functional'' description of definable functions makes it easy to prove properties relating them with definable predicates and the type space.
	
	\begin{definition} \label{def:definablefunction}
		Let $\bS$ and $\bS'$ be $\N$-tuples of sorts in the language $\cL$.  A \emph{definable function} $f: \prod_{j \in \N} S_j \to \prod_{j \in \N} S_j'$ relative to the $\cL$-theory $\rT$ is a collection of maps $\mathbf{f}^{\cM}: \prod_{j \in \N} S_j^{\cM} \to \prod_{j \in \N} (S_j')^{\cM}$ for $\cM \models \rT$ satisfying the following conditions:
		\begin{enumerate}[(1)]
			\item For each $\bD \in \prod_{j \in \N} \cD_{S_j}$, there exists $\bD' \in \prod_{j \in \N} \cD_{S_j'}$ such that for every $\cM \models \rT$, $\mathbf{f}^{\cM}$ maps $\prod_{j \in \N} D_j^{\cM}$ into $\prod_{j \in \N} (D_j')^{\cM}$.
			\item Whenever $\tilde{\bS}$ is another tuple of sorts and $\phi$ is a definable predicate relative to $\rT$ in the free variables $x_j' \in S_j'$ and $\tilde{x}_j \in \tilde{S}_j$ for $j \in \N$, then $\phi(\mathbf{f}(\mathbf{x}),\tilde{\mathbf{x}})$ is a definable predicate in the variables $\mathbf{x} = (x_j)_{j\in \N}$ and $\tilde{\mathbf{x}} = (\tilde{x}_j)_{j\in\N}$.
		\end{enumerate}
	\end{definition}
	
	The next proposition gives a more down-to-earth characterization of definable functions which can be more easily checked in examples.  This is in fact typically used as the definition \cite[Definition 9.22]{BYBHU2008}.
	
	\begin{proposition} \label{prop:deffunc2}
		Let $\bS$ and $\bS'$ be $\N$-tuples of types in the language $\cL$ and let $\rT$ be an $\cL$-theory.  Let $\mathbf{f}: \prod_{j \in \N} S_j \to \prod_{j \in \N} S_j'$ be a collection of maps $\mathbf{f}^{\cM}: \prod_{j \in \N} S_j^{\cM} \to \prod_{j \in \N} (S_j')^{\cM}$ for $\cM \models \rT$ satisfying (1) of Definition \ref{def:definablefunction}.  Then $\mathbf{f}$ is a definable function if and only if, for each $k \in \N$, the map $\phi_k(\mathbf{x},y) = d(f_k(\mathbf{x}),y)$ is a definable predicate on $\prod_{j \in \N} S_j \times S_k'$.
	\end{proposition}
	
	\begin{proof}
		($\implies$) Let $\tilde{\bS} = \bS'$, and consider the definable predicate $\phi(\mathbf{x}',\tilde{\mathbf{x}}) = d(x_k',\tilde{x}_k)$.  Taking $\tilde{\bS} = \bS'$ in Definition \ref{def:definablefunction} (2), we see that if $\mathbf{f}$ is a definable function, then $\phi(\mathbf{f}(\mathbf{x}),\tilde{\mathbf{x}}) = d(f_k(\mathbf{x}),\tilde{x}_k)$ is a definable predicate.  So substituting $y$ for $\tilde{x}_k$, we have proved the claim.
		
		($\impliedby$) In order to verify (2) of Definition \ref{def:definablefunction}, let $\tilde{\bS}$ be an $\N$-tuple of sorts, and let $\phi(\mathbf{x}',\tilde{\mathbf{x}})$ be a definable predicate on $\prod_{j \in \N} S_j' \times \prod_{j \in \N} \tilde{S}_j$.  We need to show that $\psi(\mathbf{x},\tilde{\mathbf{x}}) = \phi(\mathbf{f}(\mathbf{x}),\tilde{\mathbf{x}})$ is a well-defined definable predicate relative to $\rT$.  Thus, to check Definition \ref{def:definablepredicate}, fix $\bD \in \prod_{j \in \N} \cD_{S_j}$ and $\tilde{\bD} \in \prod_{j \in \N} \cD_{\tilde{S}_j}$ and $\epsilon > 0$.  Since we assumed that Definition \ref{def:definablefunction} (1) holds, there exists $\bD'$ such that $f$ maps $\prod_{j \in \N} D_j$ into $\prod_{j \in \N} D_j'$.
		
		By Definition \ref{def:definablepredicate}, there exists a formula $\eta$ depending on finitely many of the variables $x_j'$ and $\tilde{x}_j$ that approximates $\phi$ within $\epsilon/2$ on $\prod_{j \in \N} (D_j')^{\cM} \times \prod_{j \in \N} (\tilde{D}_j)^{\cM}$.  Let $F'$ be the set of indices $j$ such that $\eta$ depends on $x_j'$.  For $t > 0$, let
		\[
		\psi_t(\mathbf{x},\tilde{\mathbf{x}}) = \inf_{y_j \in D_j': j \in F} \left[ \eta(y_j: j \in F, \tilde{\mathbf{x}}) + \frac{1}{t} \sum_{j \in F} d(\mathbf{f}_j(\mathbf{x}), y_j) \right],
		\]
		which is a definable predicate by our assumption on $\mathbf{f}$ and by Lemma \ref{lem:defpredoperations} (2).
		
		We want to show that $\psi_t$ is close to $\phi$ when $t$ is sufficiently small.  We automatically have $\psi_t^{\cM}(\bX,\tilde{\bX}) \leq \eta^{\cM}(\mathbf{f}(\bX),\tilde{\bX})$ for $\bX \in \prod_{j \in \N} (D_j')^{\cM}$ and $\tilde{\bX} \in \prod_{j \in \N} (\tilde{D}_j)^{\cM}$ when $\cM \models \rT$ since $\mathbf{f}_j(\mathbf{X})$ is a value of $\mathbf{Y}$ participating in the infimum.  To get a bound in the other direction, first observe that since $\eta$ is a formula, $|\eta|$ is bounded on $\prod_{j \in \N} D_j' \times \prod_{j \in \N} \tilde{D}_j$ by some constant $C$.  We then observe using the triangle inequality that
		\[
		\eta^{\cM}(\mathbf{Y},\tilde{\mathbf{X}}) + \frac{1}{t} \sum_{j \in F} d^{\cM}(Y_j,f_j^{\cM}(\mathbf{X})) \geq \eta^{\cM}(\mathbf{f}^{\cM}(\mathbf{X}),\tilde{\mathbf{X}})
		\]
		unless $\sum_{j \in F} d^{\cM}(Y_j,f_j^{\cM}(\mathbf{X})) < 2Ct$, and therefore the infimum is witnessed by $\mathbf{Y}$ such that $\sum_{j \in F} d^{\cM}(Y_j,f_j^{\cM}(\mathbf{X})) < 2Ct$.  Furthermore, by the uniform continuity property of the formula $\eta$ (Observation \ref{obs:defpredunifcont}), there exists $t$ such that if $\sum_{j \in F} d^{\cM}(Y_j,f_j^{\cM}(\mathbf{X})) < 2Ct$, then
		\[
		|\eta^{\cM}(\mathbf{Y},\tilde{\mathbf{X}}) - \eta^{\cM}(\mathbf{f}(\mathbf{X}),\tilde{\mathbf{X}})| < \frac{\epsilon}{2}.
		\]
		It follows that for this value of $t$,
		\[
		\eta^{\cM}(\mathbf{f}(\bX),\tilde{\bX}) - \frac{\epsilon}{2} \leq \psi_t^{\cM}(\mathbf{X},\tilde{\mathbf{X}}) \leq \eta^{\cM}(\mathbf{f}(\mathbf{X}),\tilde{\mathbf{X}})
		\]
		hence also
		\[
		|\psi_t^{\cM}(\mathbf{X},\tilde{\mathbf{X}}) - \psi^{\cM}(\mathbf{X},\tilde{\mathbf{X}})| < \epsilon,
		\]
		for $\cM \models \rT$ and $\bX \in \prod_{j \in \N} D_j^{\cM}$ and $\tilde{\bX} \in \prod_{j \in \N} (\tilde{D}_j)^{\cM}$.  Since $\psi$ can be approximated in this way by definable predicates, it is a definable predicate, which proves the claim of the proposition.
	\end{proof}
	
	\begin{corollary} \label{cor:termdefinable}
		If $f_k: \prod_{j \in \N} S_j \to S_k'$ is a term for each $k \in \N$, then $\mathbf{f} = (f_j)_{j \in \N}$ is a definable function relative to any theory $\rT$.
	\end{corollary}
	
	\begin{proof}
		By definition a term maps a product of domains of quantification into some domain of quantification, which verifies (1) of Definition \ref{def:definablefunction}.  Moreover, for each $k$, $d(f_k(\mathbf{x}),y_k)$ is a formula, hence a definable predicate, so by the previous proposition $\mathbf{f}$ is a definable function.
	\end{proof}
	
	Similar to definable predicates, definable functions are automatically uniformly continuous with respect to $d$ on each product of domains of quantification.  This is a straightforward generalization of \cite[Proposition 9.23]{BYBHU2008}.
	
	\begin{lemma} \label{lem:deffuncunifcont}
		Let $\cL$ be a language, $\rT$ an $\cL$-theory, $\mathbf{S}$ and $\mathbf{S}'$ $\N$-tuples of sorts, and $\mathbf{f}: \prod_{j \in \N} S_j \to \prod_{j \in \N} S_j'$ a definable function.  Then for every $\bD \in \prod_{j \in \N} \cD_{S_j}$ and $F \subseteq \N$ finite and $\epsilon > 0$, there exists $E \subseteq \N$ finite and $\delta > 0$ such that
		\[
		d_{S_i}^{\cM}(X_i,Y_i) < \delta \text{ for all } i \in E \implies d_{S_j'}^{\cM}(f_j^{\cM}(\mathbf{X}),f_j^{\cM}(\mathbf{Y})) < \epsilon \text{ for all } j \in F.
		\]
		whenever $\cM \models \rT$ and $\mathbf{X}, \mathbf{Y} \in \prod_{j \in \N} D_j^{\cM}$.
	\end{lemma}
	
	\begin{proof}
		Let $\bD' \in \prod_{j \in \N} \cD_{S_j}$ such that $\mathbf{f}^{\cM}$ maps $\prod_{j \in \N} D_j^{\cM}$ into $\prod_{j \in \N} (D_j')^{\cM}$ for all $\cM \models \rT$.  Fix $\epsilon > 0$ and $F \subseteq \N$ finite.  Then by Lemma \ref{lem:defpredoperations} and Proposition \ref{prop:deffunc2}, the object
		\[
		\phi^{\cM}(\mathbf{X},\mathbf{Y}) = \max_{j \in F} d_{S_j'}^{\cM}(f_j^{\cM}(\mathbf{X}),Y_j)
		\]
		is a definable predicate relative to $\rT$.  Hence, by Observation \ref{obs:defpredunifcont} there exists $E \subseteq \N$ finite and $\delta > 0$ such that
		\[
		d_{S_i}^{\cM}(X_i',X_i) < \delta \text{ for } i \in E \text{ and } d_{S_j'}^{\cM}(Y_j,Y_j') < \delta \text{ for } j \in F \implies |\phi^{\cM}(\bX,\bY) - \phi^{\cM}(\bX',\bY')| < \epsilon
		\]
		for $\mathbf{X}$, $\bX' \in \prod_{j \in \N} D_j^{\cM}$ and $\bY$, $\bY' \in \prod_{j \in \cM} (D_j')^{\cM}$ and $\cM \models \rT$.  Taking $\bY = \mathbf{f}(\bX)$, we see that
		\[
		|\phi^{\cM}(\bX,\mathbf{f}^{\cM}(\bX)) - \phi^{\cM}(\bX',\mathbf{f}^{\cM}(\bX))| = \max_{j \in F} d_{S_j'}^{\cM}(f_j(\bX'), f_j(\bX)) < \epsilon
		\]
		whenever $\cM \models \rT$ and $\bX$, $\bX' \in \prod_{j \in \N} D_j$ and $\max_{i \in E} d_{S_i}^{\cM}(X_i,X_i') < \delta$, which is the desired uniform continuity property.
	\end{proof}
	
	Next, we describe the relationship between definable functions and types.
	
	\begin{lemma}~
		Let $\bS$ and $\bS'$ be $\N$-tuples of sorts, and let $\mathbf{f}: \prod_{j \in \N} S_j \to \prod_{j \in \N} S_j'$ be a definable function relative to $\rT$.
		\begin{enumerate}[(1)]
			\item If $\phi$ is a definable predicate in the variables $x_j' \in S_j'$ for $j \in \N$, then $\phi \circ \mathbf{f}$ is a definable predicate.
			\item If $\cM \models \rT$ and $\bX \in \prod_{j \in \N} S_j^{\cM}$, then $\tp^{\cM}(\mathbf{f}^{\cM}(\bX))$ is uniquely determined by $\tp^{\cM}(\bX)$ and $\mathbf{f}$.
			\item Let $\mathbf{f}_*: \mathbb{S}_{\bS}(\rT) \to \mathbb{S}_{\bS'}(\rT)$ be the map such that $\tp^{\cM}(\mathbf{f}^{\cM}(\bX)) = \mathbf{f}_* \tp^{\cM}(\bX)$.  Then $\mathbf{f}_*$ is continuous with respect to the logic topology.
		\end{enumerate}
	\end{lemma}
	
	\begin{proof}
		(1) Considering another $\N$-tuple $\tilde{\bS}$ of sorts, we may view $\phi(\mathbf{x})$ as a definable predicate in $(\mathbf{x},\tilde{\mathbf{x}})$, and hence by Definition \ref{def:definablefunction}, $\phi \circ \mathbf{f}$ is a definable predicate.
		
		(2) For every definable predicate $\phi$ in $\mathbf{x}'$, $\phi \circ \mathbf{f}$ is a definable predicate, and hence $\phi^{\cM} \circ \mathbf{f}^{\cM}(\bX)$ is uniquely determined by $\tp^{\cM}(\bX)$ for all $\cM \models \rT$ and $\bX \in \prod_{j \in \N} S_j^{\cM}$.  Since this is true for every definable predicate $\phi$ in variables $\mathbf{x}'$, it follows that $\tp^{\cM}(\mathbf{f}(\bX))$ is uniquely determined by $\mathbf{f}$ and $\tp^{\cM}(\bX)$.
		
		(3) Let $\cO$ be an open set in $\mathbb{S}_{\bS'}(\rT)$.  By definition of the topology on $\mathbb{S}_{\bS}(\rT)$, in order to show that $(\mathbf{f}_*)^{-1}(\cO)$ is open, is suffices to show that $(\mathbf{f}_*)^{-1}(\cO) \cap \mathbb{S}_{\bD}(\rT)$ is open for every $\bD \in \prod_{j \in \N} \cD_{S_j}$.
		
		For any such $\bD$, by Definition \ref{def:definablefunction} (1), there exists $\bD' \in \prod_{j \in \N} \cD_{S_j'}$ such that $\mathbf{f}^{\cM}$ maps $\prod_{j \in \N} D_j^{\cM}$ into $\prod_{j \in \N} (D_j')^{\cM}$ for all $\cM \models \rT$.  This implies that $\mathbf{f}_*$ maps $\mathbb{S}_{\bD}(\rT)$ into $\mathbb{S}_{\bD'}(\rT)$.  Hence, $(\mathbf{f}_*)^{-1}(\cO) \cap \mathbb{S}_{\bD}(\rT) = (\mathbf{f}_* |_{\mathbb{S}_{\bD}(\rT)})^{-1}(\cO \cap \mathbb{S}_{\bD'}(\rT))$, so to show that this set is $(\mathbf{f}_*)^{-1}(\cO) \cap \mathbb{S}_{\bD}(\rT)$ is open, it suffices to check continuity of $\mathbf{f}_*$ as a map $\mathbb{S}_{\bD}(\rT) \to \mathbb{S}_{\bD'}(\rT)$.
		
		By Lemma \ref{lem:defpredweakstar}, the topology on $\mathbb{S}_{\bD'}(\rT)$ is generated by the pairings of types with every definable predicate $\phi$ in the variables $\mathbf{x}'$.  If $\phi$ is such a definable predicate, then $\phi \circ \mathbf{f}$ is a definable predicate in $\mathbf{x}$ by (1), and therefore, it is continuous with respect to the logic topology on $\mathbb{S}_{\bD}(\rT)$.  Thus, the map $\mathbf{f}_*: \mathbb{S}_{\bD}(\rT) \to \mathbb{S}_{\bD'}(\rT)$ is continuous as desired.
	\end{proof}
	
	Finally, we verify that definable functions are closed under composition.
	
	\begin{observation}
		Let $\bS$, $\bS'$, and $\bS''$ be $\N$-tuples of sorts in the language $\cL$.  If $\mathbf{f}: \prod_{j \in \N} S_j \to \prod_{j \in \N} S_j'$ and $\mathbf{g}: \prod_{j \in \N} S_j' \to \prod_{j \in \N} S_j''$ are definable functions, then so is $\mathbf{g} \circ \mathbf{f}$.
	\end{observation}
	
	\begin{proof}
		If $\bD \in \prod_{j \in \N} \cD_{S_j}$, then since $\mathbf{f}$ is definable, there exists $\bD' \in \prod_{j \in \N} \cD_{S_j'}$ such that $\mathbf{f}^{\cM}$ maps $\prod_{j \in \N} D_j^{\cM}$ into $\prod_{j \in \N} (D_j')^{\cM}$ for every $\cM \models \rT$.  Similarly, there exists $\bD'' \in \prod_{j \in \N} \cD_{S_j''}$ such that $\mathbf{g}^{\cM}$ maps $\prod_{j \in \N} (D_j')^{\cM}$ into $\prod_{j \in \N} (D_j'')^{\cM}$.  Hence, $(\mathbf{g} \circ \mathbf{f})^{\cM}$ maps $\prod_{j \in \N} D_j^{\cM}$ into $\prod_{j \in \N} (D_j'')^{\cM}$, so that $\mathbf{g} \circ \mathbf{f}$ satisfies (1) of Definition \ref{def:definablefunction}.
		
		Let $\tilde{\bS}$ be another $\N$-tuple of sorts and let $\phi$ be a definable predicate in the variables $x_j'' \in S_j''$ for $j \in \N$ and $\tilde{x}_j$ for $j \in \N$.  By the definability of $\mathbf{g}$, $\psi(\mathbf{x}',\tilde{\mathbf{x}}) := \phi(\mathbf{g}(\mathbf{x}'),\tilde{\mathbf{x}})$ is also a definable predicate.  Then by definability of $\mathbf{f}$, $\psi(\mathbf{f}(\mathbf{x}), \tilde{\mathbf{x}}) = \phi((\mathbf{g} \circ \mathbf{f})(\mathbf{x}), \tilde{\mathbf{x}})$ is a definable predicate.  Therefore, $\mathbf{g} \circ \mathbf{f}$ satisfies (2) of Definition \ref{def:definablefunction}, so it is a definable function.
	\end{proof}
	
	\subsection{Quantifier-free types and definable predicates} \label{subsec:qftypes}
	
	Quantifier-free formulas, that is, formulas defined without suprema or infima, are the simplest kind of formula and have special significance in our study of tracial $\mathrm{W}^*$-algebras.
	
	\begin{definition}
		\emph{Quantifier-free formulas} are formulas obtained through the application of relations to terms and iterative application of continuous functions $\R^n \to \R$, that is, formulas obtained without using $\sup$ and $\inf$ operations.  If $\bS$ is an $\N$-tuple of sorts, we denote the set of quantifier-free formulas in variables $x_j \in S_j$ for $j \in \N$ by $\cF_{\qf,\bS}$.
	\end{definition}
	
	Quantifier-free types, the space of quantifier-free types, and quantifier-free definable predicates are defined in the same ways as the analogous objects for types, to wit:
	
	\begin{definition}
		Let $\bS$ be an $\N$-tuple of sorts in the language $\cL$.  If $\cM$ is an $\cL$-structure, then the \emph{quantifier-free type} of $\bX \in \prod_{j \in \N} S_j^{\cM}$ is the map
		\[
		\tp_{\qf}^{\cM}(\bX): \cF_{\bS} \to \R: \phi \mapsto \phi^{\cM}(\bX).
		\]
	\end{definition}
	
	\begin{definition}~
		\begin{itemize}
			\item If $\rT$ is an $\cL$-theory and $\bS$ is an $\N$-tuple of sorts, then $\mathbb{S}_{\qf,\bS}(\rT)$ will denote the set of types $\tp_{\qf}^{\cM}(\bX)$ for $\bX \in \prod_{j \in \N} S_j^{\cM}$ and $\cM \models \rT$.
			\item If $\bD \in \prod_{j \in \N} \bD_{\bS}$, then $\mathbb{S}_{\qf,\bD}(\rT)$ will denote the set of types $\tp^{\cM}(\bX)$ for $\cM \models \rT$ and $\bX \in \prod_{j \in \N} D_j^{\cM}$.
			\item We equip $\mathbb{S}_{\qf,\bD}(\rT)$ with the weak-$\star$ topology as a subset of the dual of $\cF_{\qf,\bS}$.
			\item We equip $\mathbb{S}_{\qf,\bS}(\rT)$ with the topology such that $\cO$ is open if and only if $\cO \cap \mathbb{S}_{\qf,\bD}(\rT)$ is open for every $\bD \in \prod_{j \in \N} \cD_{S_j}$.  We call this the \emph{(quantifier-free) logic topology}.
		\end{itemize}
	\end{definition}
	
	\begin{definition}
		Let $\rT$ be an $\cL$-theory and $\bS$ is an $\N$-tuple of sorts.  A \emph{quantifier-free definable predicate} is collection of functions $\phi^{\cM}: \prod_{j \in \N} S_j^{\cM} \to \R$ for $\cM \models \rT$ such that for every $\bD \in \prod_{j \in \N} \cD_{S_j}$ and $\epsilon > 0$, there exists a quantifier-free formula $\psi$ in finitely many of the variables $x_j \in S_j$, such that
		\[
		|\phi^{\cM}(\bX) - \psi^{\cM}(\bX)| < \epsilon
		\]
		for $\bX \in \prod_{j \in \N} D_j^{\cM}$ for $\cM \models \rT$.
	\end{definition}
	
	The following can be verified in the same way as for types, when $\bS$ is an $\N$-tuple of sorts and $\rT$ is an $\cL$-theory:
	\begin{itemize}
		\item For each $\bD \in \prod_{j \in \N} \cD_{S_j}$, the space $\mathbb{S}_{\qf,\bD}(\rT)$ is a compact Hausdorff space.
		\item $\phi$ is a quantifier-free definable predicate if and only if $\phi^{\cM}(\bX) = \omega(\tp^{\cM}(\bX))$ for some continuous $\omega: \mathbb{S}_{\qf,\bS}(\rT) \to \R$.
		\item If $\phi_j$ is a quantifier-free definable predicate for $j \in \N$ and $F: \R^{\N} \to \R$ is continuous, then $F((\phi_j)_{j \in \N})$ is a quantifier-free definable predicate.
	\end{itemize}
	Furthermore, the quantifier-free type space and the type space can be related as follows.
	
	\begin{observation}
		Let $\bS$ be an $\N$-tuple of sorts in $\cL$ and $\rT$ an $\cL$-theory.  Let $\pi: \mathbb{S}_{\bS}(\rT) \to \mathbb{S}_{\qf,\bS}(\rT)$ be the map that sends a type (as a linear map $\cF_{\bS} \to \R$) to its restriction to $\cF_{\qf,\bS}$.  Then $\pi(\mathbb{S}_{\bD}(\rT)) = \mathbb{S}_{\qf,\bD}(\rT)$ for each $\bD \in \prod_{j \in \N} \cD_{S_j}$, and $\pi$ is a topological quotient map.
	\end{observation}
	
	\begin{proof}
		First, $\pi$ is a continuous map $\mathbb{S}_{\bD}(\rT) \to \mathbb{S}_{\qf,\bD}(\rT)$ by definition of the weak-$\star$ topology.  Then since a set in $\mathbb{S}_{\bS}(\rT)$ is open if and only if its restriction to $\mathbb{S}_{\bD}(\rT)$ is open, and the same holds for the quantifier-free versions, we deduce that $\pi$ is continuous.  It is immediate from the definitions that $\pi(\mathbb{S}_{\bD}(\rT)) = \mathbb{S}_{\qf,\bD}(\rT)$.  Then because $\mathbb{S}_{\bD}(\rT)$ is compact and $\mathbb{S}_{\qf,\bD}(\rT)$ is Hausdorff, $\pi$ defines a topological quotient map $\mathbb{S}_{\bD}(\rT) \to \mathbb{S}_{\qf,\bD}(\rT)$.  Finally, using the definition of open sets in $\mathbb{S}_{\bS}(\rT)$ and $\mathbb{S}_{\qf,\bS}(\rT)$, we deduce that $\cO \subseteq \mathbb{S}_{\qf,\bS}(\rT)$ is open if and only if $\pi^{-1}(\cO)$ is open, hence $\pi$ is a topological quotient map.
	\end{proof}
	
	\begin{remark} \label{rem:qfnorm}
		A convenient feature of $\cL_{\tr}$ is that $\pi^{-1}(\mathbb{S}_{\qf,\mathbf{r}}(\rT_{\tr})) = \mathbb{S}_{\mathbf{r}}(\rT_{\tr})$.  Indeed, suppose that $\cM \models \rT_{\tr}$ and $\bX \in L^\infty(\cM)^{\N}$ with $\tp_{\qf}^{\cM}(\bX) \in \mathbb{S}_{\mathbf{r}}(\rT_{\tr})$.  Then the operator norm of $X_j$ can be recovered from $\tp_{\qf}^{\cM}(\bX)$ through
		\[
		\norm{X_j} = \lim_{k \to \infty} (\re \tr^{\cM}((X_j^*X_j)^k))^{1/2k},
		\]
		hence $\mathbf{X} \in \prod_{j \in \N} D_{r_j}^{\cM}$, so that $\tp^{\cM}(\bX) \in \mathbb{S}_{\mathbf{r}}(\rT_{\tr})$.
	\end{remark}
	
	\subsection{Quantifier-free definable functions in $\mathcal{L}_{\tr}$} \label{subsec:qfdefinablefunctions}
	
	Quantifier-free definable functions are defined analogously to definable functions.
	
	\begin{definition} \label{def:qfdeffunc}
		$\rT$ be an $\cL$-theory, and let $\bS$ and $\bS'$ be $\N$-tuples of sorts.  A \emph{quantifier-free definable function} $\mathbf{f}: \prod_{j \in \N} S_j \to \prod_{j \in \N} S_j'$ is a collection of functions $\mathbf{f}^{\cM}: \prod_{j \in \N} S_j^{\cM} \to \prod_{j \in \N} (S_j')^{\cM}$ for all $\cM \models \rT$ satisfying the following conditions:
		\begin{enumerate}[(1)]
			\item For each $\bD \in \prod_{j \in \N} \cD_{S_j}$, there exists $\bD' \in \prod_{j \in \N} \cD_{S_j'}$ such that for every $\cM \models \rT$, $\mathbf{f}^{\cM}$ maps $\prod_{j \in \N} D_j^{\cM}$ into $\prod_{j \in \N} (D_j')^{\cM}$.
			\item Whenever $\tilde{\bS}$ is another tuple of sorts and $\phi$ is a quantifier-free definable predicate relative to $\rT$ in the free variables $x_j' \in S_j'$ and $\tilde{x}_j \in \tilde{S}_j$ for $j \in \N$, then $\phi(\mathbf{f}(\mathbf{x}),\tilde{\mathbf{x}})$ is a quantifier-free definable predicate in the variables $\mathbf{x} = (x_j)_{j\in \N}$ and $\tilde{\mathbf{x}} = (\tilde{x}_j)_{j\in\N}$.
		\end{enumerate}
	\end{definition}
	
	\begin{example} \label{ex:qfterm}
		If $f_j$ is a term in a finite subset of the variables $x_j$, then $\mathbf{f} = (f_j)_{j \in \N}$ is a quantifier-free definable function relative to any $\cL$-theory $\rT$.  To see this, suppose that $\phi$ is a quantifier-free definable predicate and $\bD$ is a tuple of domains of quantification.  Let $\mathbf{f}$ map $\bD$ into $\bD'$.  As a quantifier-free definable predicate, $\phi$ can, for any given $\epsilon > 0$, be approximated on $\bD'$ by a quantifier-free formula $\psi$ with error  smaller than $\epsilon$ on $\prod_{j \in \N} (D_j')^{\cM}$ for all $\cM \models \rT$.  Then $\psi \circ \mathbf{f}$ is a quantifier-free formula that approximates $\phi \circ \mathbf{f}$ within $\epsilon$ on $\prod_{j \in \N} D_j^{\cM}$ for all $\cM \models \rT$.
	\end{example}
	
	The following facts about quantifier-free definable functions are verified just as in the case of definable functions:
	\begin{itemize}
		\item Quantifier-free definable functions are closed under composition.
		\item If $\phi$ is a quantifier-free definable predicate and $\mathbf{f}$ is a quantifier-free definable function, then $\phi \circ \mathbf{f}$ is a quantifier-free definable predicate.
		\item For each definable function $\mathbf{f}: \prod_{j \in \N} S_j \to \prod_{j \in \N} S_j'$, there is a continuous map $\mathbf{f}_*: \mathbb{S}_{\bS}(\rT) \to \mathbb{S}_{\bS'}(\rT)$ given by $\tp_{\qf}^{\cM}(\mathbf{f}(\bX)) = \mathbf{f}_* \tp_{\qf}^{\cM}(\bX)$.
		\item If $\mathbf{f}$ is a quantifier-free definable function, then for each $k \in \N$, the object $\phi(\mathbf{x},y) = d_{S_k'}(f_k(\mathbf{x}), y)$ is a quantifier-free definable predicate.
		\item Hence, $\mathbf{f}$ is a quantifier-free definable function, then it is a definable function by Proposition \ref{prop:deffunc2}.
		\item Thus, a quantifier-free definable function is uniformly continuous on every product of domains of quantification.
	\end{itemize}
	The proof of ($\impliedby$) in Proposition \ref{prop:deffunc2} uses quantifiers (specifically infima) and thus does not directly adapt to the quantifier-free setting.  This is why we argued that terms are quantifier-free functions directly in Example \ref{ex:qfterm} rather than as in Corollary \ref{cor:termdefinable}.
	
	In the special case of $\cL_{\tr}$ and $\rT_{\tr}$, we have the following characterizations of quantifier-free definable functions.  Recall that $\cL_{\tr}$ has one type $S$ and the domains of quantification are given by $D_r$ for each $r > 0$.  Given $\mathbf{r} = (r_j)_{j \in \N} \in (0,\infty)^{\N}$, we write $\mathbb{S}_{\qf,\mathbf{r}}(\rT_{\tr})$ for the set of quantifier-free types of $\N$-tuples in $\prod_{j \in \N} D_{r_j}$ in $\cL_{\tr}$ relative to $\rT_{\tr}$.  A variant of this theorem was proved in the author's Ph.D. thesis \cite[Proposition 13.6.4]{JekelThesis}.
	
	\begin{theorem} \label{thm:qfdeffunc2}
		Let $\mathbf{f}$ be a collection of functions $\mathbf{f}^{\cM}: L^\infty(\cM)^{\N} \to L^\infty(\cM)^{\N}$ for each $\cM \models \rT_{\tr}$.  Suppose that for every $\mathbf{r} \in (0,\infty)^{\N}$, there exists $\mathbf{r}' \in (0,\infty)^{\N}$ such that $\mathbf{f}^{\cM}$ maps $\prod_{j \in \N} D_{r_j}^{\cM}$ into $\prod_{j \in \N} D_{r_j'}^{\cM}$; assume that for each $\mathbf{r}$ a corresponding $\mathbf{r}'$ has been chosen, which we will refer to below.  Then the following are equivalent.
		\begin{enumerate}[(1)]
			\item $\mathbf{f}$ is a quantifier-free definable function in $\cL_{\tr}$ relative to $\rT_{\tr}$.
			\item For each $k \in \N$, the object $\phi(\mathbf{x},y) = d_{S_k'}(f_k(\mathbf{x}), y)$ is a quantifier-free definable predicate.
			\item For each $k \in \N$, $\mathbf{r} \in (0,\infty)^{\N}$, and quantifier-free type $\mu \in \mathbb{S}_{\qf,\bD_{\mathbf{r}}}(\rT_{\tr})$ and $\epsilon > 0$, there exists a term $g$ and an open neighborhood $\cO$ of $\mu$ in $\mathbb{S}_{\qf,\mathbf{r}}(\rT_{\tr})$ such that, for all $\cM \models \rT$ and $\bX \in \prod_{j \in \N} D_{r_j}^{\cM}$,
			\[
			g^{\cM}(\bX) \in D_{r_k'}
			\]
			and
			\[
			\tp_{\qf}^{\cM}(\bX) \in \cO \implies d^{\cM}(f_k^{\cM}(\bX), g^{\cM}(\bX)) < \epsilon.
			\]
			\item For each $k \in \N$, $\mathbf{r} \in (0,\infty)^{\N}$, and $\epsilon > 0$, there exist $m \in \N$, quantifier-free formulas $\phi_1$, \dots, $\phi_m$, and terms $g_1$, \dots, $g_m$ such that
			\[
			\sum_{j=1}^m \phi_j^{\cM}(\bX) g_j^{\cM}(\bX) \in D_{r_k'} \text{ and } d^{\cM}\left(f_k^{\cM}(\bX), \sum_{j=1}^m \phi_j^{\cM}(\bX) g_j^{\cM}(\bX) \right) < \epsilon
			\]
			whenever $\cM \models \rT$ and $\bX \in \prod_{j \in \N} D_{r_j}^{\cM}$.
		\end{enumerate}
	\end{theorem}
	
	\begin{proof}
		(1) $\implies$ (2) follows as in Proposition \ref{prop:deffunc2}.
		
		(2) $\implies$ (3).  Fix $\mathbf{r}$ and $\mu$ and $\epsilon$.  Let $\bX$ be an $\N$-tuple from a tracial $\mathrm{W}^*$-algebra $\cM$ which has quantifier-free type $\mu$.  Note that if $\phi$ is a quantifier-free formula and $\cN$ is a $\mathrm{W}^*$-subalgebra of $\cM$ containing $\bX$, then $\phi^{\cN}(\bX) = \phi^{\cM}(\bX)$.  Hence, the quantifier-free type of $\bX$ in $\cN$ is the same as the quantifier-free type in $\cM$.  In particular, we can replace $\cM$ with $\cN = \mathrm{W}^*(\bX)$ and thus assume without loss of generality that $\cM = \mathrm{W}^*(\bX)$.
		
		Recall that $\cM$ is faithfully represented on the GNS space $L^2(\cM)$ (see \S \ref{subsec:operatoralgebrasbackground}, the standard representation).  By the Kaplansky density theorem \cite{Kaplansky1953}, \cite[Theorem 5.3.5]{KadisonRingroseI}, the ball of radius $r_k'$ in the $\mathrm{C}^*$-algebra $\mathrm{C}^*(\bX)$ generated by $\bX$ is dense in the ball of radius $r_k'$ in $\cM = \mathrm{W}^*(\bX)$ with respect to the strong operator topology.  Since approximation in the strong operator topology implies approximation in the $2$-norm associated to the trace, it follows that there exists $Z \in \mathrm{C}^*(\bX)$ such that $\norm{Z} \leq r_k'$ and $\norm{Z - Y}_2 = d^{\cM}(Z,Y) < \epsilon / 2$.
		
		Next, we must obtain a term $g$ bounded by $r_k'$ such that $d^{\cM}(g^{\cM}(\bX),Y) < \epsilon$.  Because we want $g$ to be bounded by $r_k'$ on $\prod_{j \in \N} D_{r_j}$ for all $\cM \models \rT_{\tr}$, we view the $*$-polynomials in infinitely many indeterminates as part of a universal $\mathrm{C}^*$-algebra.  For $*$-polynomials $p$ in infinitely many variables $x_j: j \in \N$, let
		\[
		\norm{p}_u = \sup \left\{ \norm{p(\bX)}: \bX \in \prod_{j \in \N} D_{r_j}^{\cM}, \cM \models \rT_{\tr} \right\}.
		\]
		This defines a $\mathrm{C}^*$-norm on $\C\ip{x_j, x_j^*: j \in \N}$.  Let $A$ be the completion of $\C\ip{x_j, x_j^*: j \in \N}$ into a $\mathrm{C}^*$-algebra.  If $\cM \models \rT_{\tr}$ and $\bX \in \prod_{j \in \N} D_j^{\cM}$, then $\norm{p(\bX)} \leq \norm{p}_u$ by definition, so there is a $*$-homomorphism $\pi: A \to \mathrm{C}^*(\bX)$ mapping $x_j \in A$ to $X_j \in \cM$ for each $j \in \N$.  By \cite[II.5.1.5]{Blackadar2006}, there exists $z \in A$ such that $\pi(z) = Z$ and $\norm{z}_A \leq r_k'$.
		
		Now by definition $\C\ip{x_j,x_j^*: j \in \N}$ is dense in $A$.  It follows that every element of the $r_k'$-ball of $A$ can be approximated by non-commutative $*$-polynomials in the $r_k'$-ball.  In particular, there exists some $g \in \C\ip{x_j,x_j^*: j \in \N}$ with $\norm{g - z}_A < \epsilon/2$, and we can also arrange that $\norm{g}_A \leq r_k'$.
		
		Then $g$ is a term such that $\norm{g^{\cN}(\bY)} \leq r_k'$ for all $\bY \in \prod_{j \in \N} D_{r_j}^{\cN}$ for all $\cN \models \rT_{\tr}$ and such that
		\[
		d^{\cM}(f_k^{\cM}(\bX),g^{\cM}(\bX)) < \epsilon.
		\]
		for our particular choice of $\cM = \mathrm{W}^*(\bX)$ with $\tp_{\qf}^{\cM}(\bX) = \mu$.  Now observe that
		\[
		\psi(\mathbf{x}) = d(f_k(\mathbf{x}), g(\mathbf{x}))
		\]
		is a quantifier-free definable predicate since the term $g$ is a quantifier-free definable function.  Let
		\[
		\cO = \{\nu \in \mathbb{S}_{\qf,\mathbf{r}}(\rT_{\tr}): \nu(\psi) < \epsilon \}.
		\]
		Then $\cO \ni \mu$ by our choice of $g$ and $\cO$ is open by continuity of $\psi$.  Moreover, by definition, if $\cN \models \rT$ and $\bY \in \prod_{j \in \N} D_{r_j}$ and $\tp_{\qf}^{\cN}(\bY) \in \cO$, then $d^{\cN}(f_k^{\cN}(\bY),g^{\cN}(\bY)) < \epsilon$.
		
		(3) $\implies$ (4).  Fix $k \in \N$, $\mathbf{r} \in (0,\infty)^{\N}$, and $\delta > 0$.  For each $\mu \in \mathbb{S}_{\mathbf{r}}(\rT_{\tr})$, there exists an open $\cO_\mu \subseteq \mathbb{S}_{\mathbf{r}}(\rT_{\tr})$ and a term $g_\mu$ such that for $\cM \models \rT$ and $\bX \in \prod_{j \in \N} D_{r_j}^{\cM}$,
		\[
		g_\mu^{\cM}(\bX) \in D_{r_k'}
		\]
		and
		\[
		\tp_{\qf}^{\cM}(\bX) \in \cO \implies d^{\cM}(g_\mu^{\cM}(\bX),f_k^{\cM}(\bX)) < \delta.
		\]
		The sets $\cO_\mu$ form an open cover of the compact set $\mathbb{S}_{\qf,\mathbf{r}}(\rT_{\tr})$, and hence there exists $m \in \N$ and $\mu_1$, \dots, $\mu_m$ such that $\cO_{\mu_1}$, \dots, $\cO_{\mu_m}$ cover $\mathbb{S}_{\qf,\mathbf{r}}(\rT_{\tr})$.  Let $\cO_j = \cO_{\mu_j}$, $g_j = g_{\mu_j}$.
		
		Since $\mathbb{S}_{\qf,\mathbf{r}}(\rT_{\tr})$ is a compact Hausdorff space, there exists a continuous partition of unity $\psi_1$, \dots, $\psi_m$ subordinated to the cover $\cO_1$, \dots, $\cO_m$.  In other words, there exist $\gamma_1$, \dots, $\gamma_m \in C(\mathbb{S}_{\qf,\mathbf{r}}(\rT_{\tr}))$ such that $\gamma_j \geq 0$, $\gamma_j|_{\cO_j^c} = 0$, and $\sum_{j=1}^m \gamma_j = 1$.  Therefore, for $\cM \models \rT$ and $\bX \in \prod_{j \in \N} D_{r_j}^{\cM}$,
		\begin{align*}
			d^{\cM}\left( \sum_{j=1}^m \gamma_j(\tp_{\qf}^{\cM}(\bX)) g_j^{\cM}(\bX), f_k^{\cM}(\bX) \right) &= \norm*{ \sum_{j=1}^m \gamma_j(\tp_{\qf}^{\cM}(\bX)) (g_j^{\cM}(\bX) - f_k^{\cM}(\bX)) }_{L^2(\cM)} \\
			&\leq \sum_{j=1}^m \gamma_j(\tp_{\qf}^{\cM}(\bX)) \norm*{g_j^{\cM}(\bX) - f_k^{\cM}(\bX) }_{L^2(\cM)} \\
			&\leq \sum_{j=1}^m \gamma_j(\tp_{\qf}^{\cM}(\bX)) \delta \\
			&= \delta,
		\end{align*}
		and
		\[
		\norm*{ \sum_{j=1}^m \gamma_j(\tp_{\qf}^{\cM}(\bX)) g_j^{\cM}(\bX) }_{L^\infty(\cM)} \leq \sum_{j=1}^m \gamma_j(\tp_{\qf}^{\cM}(\bX)) \norm{g_j^{\cM}(\bX)}_{L^\infty(\cM)} \leq r_k'.
		\]
		Because quantifier-free formulas comprise a dense subset of $C(\mathbb{S}_{\qf}(\rT))$ by the Stone-Weierstrass theorem, there exist quantifier-free formulas $\psi_1$, \dots, $\psi_m$ such that 
		\[
		|\psi_j^{\cM}(\bX) - \gamma_j(\tp^{\cM}(\bX))| \leq \frac{\delta}{m} \text{ for } \cM \models \rT_{\tr} \text{ and } \bX \in \prod_{j \in \N} D_j^{\cM}.
		\]
		It follows that
		\[
		\norm*{ \sum_{j=1}^m \psi_j^{\cM}(\bX) g_j^{\cM}(\bX) }_{L^\infty(\cM)} \leq r_k' \sum_{j=1}^m |\psi_j^{\cM}(\bX)| \leq r_k'(1 + \delta),
		\]
		and
		\[
		\norm*{ \sum_{j=1}^m \gamma_j(\tp_{\qf}^{\cM}(\bX)) g_j^{\cM}(\bX) - \sum_{j=1}^m \psi_j^{\cM}(\bX) g_j^{\cM}(\bX) }_{L^2(\cM)} \leq r_k' \delta.
		\]
		Therefore, let $\phi_j = (1 + \delta)^{-1} \psi_j$.  Then 
		\[
		\norm*{ \sum_{j=1}^m \phi_j^{\cM}(\bX) g_j^{\cM}(\bX) }_{L^\infty(\cM)} \leq (1 + \delta)^{-1} r_k' \sum_{j=1}^m |\psi_j^{\cM}(\bX)| \leq r_k',
		\]
		and
		\[
		\norm*{ \sum_{j=1}^m \phi_j^{\cM}(\bX) g_j^{\cM}(\bX) - \sum_{j=1}^m \psi_j^{\cM}(\bX) g_j^{\cM}(\bX) }_{L^2(\cM)} \leq r_k' (1 - (1 + \delta)^{-1}) = \frac{\delta r_k'}{1 + \delta}.
		\]
		Hence, by the triangle inequality,
		\[
		\norm*{ \sum_{j=1}^m \phi_j^{\cM}(\bX) g_j^{\cM}(\bX) - f^{\cM}(\bX) }_{L^2(\cM)} \leq \delta + \delta r_k' + \frac{\delta r_k'}{1 + \delta}.
		\]
		By choosing $\delta$ sufficiently small, we can guarantee that the right-hand side is smaller than a given $\epsilon > 0$, so the quantifier-free formulas $\phi_1$, \dots, $\phi_m$ have the desired properties for (4).
		
		(4) $\implies$ (2).  Fix $k$ and we will show that $\phi(\mathbf{x},y) = d(f_k(\mathbf{x}),y)$ is a quantifier-free definable predicate.  To this end, fix $\mathbf{r} \in (0,\infty)^{\N}$, $r'> 0$, and $\epsilon > 0$, and we will approximate $\phi$ by a quantifier-free formula on $\prod_{j \in \N} D_j^{\cM} \times D_{r'}$ within $\epsilon$ for $\cM \models \rT_{\tr}$.  Let $m \in \N$ and $\phi_1$, \dots, $\phi_m$ and $g_1$, \dots, $g_m$ be as in (4) for our given $\mathbf{r}$ and $\epsilon$, and let
		\[
		h^{\cM}(\bX) = \sum_{j=1}^m \phi_j^{\cM}(\bX) g_j^{\cM}(\bX).
		\]
		Then for $\cM \models \rT_{\tr}$ and $\bX \in \prod_{j \in \N} D_{r_j}^{\cM}$ and $Y \in D_{r'}$, we have
		\[
		|d^{\cM}(f_k^{\cM}(\bX),Y) - d^{\cM}(h^{\cM}(\bX),Y)| \leq d^{\cM}(f_k^{\cM}(\bX),h^{\cM}(\bX)) < \epsilon.
		\]
		But note that
		\begin{align*}
			d^{\cM}(h^{\cM}(\bX),Y) &= \norm*{\sum_{j=1} \phi_j^{\cM}(\bX) g_j^{\cM}(\bX) - Y}_{L^2(\cM)} \\
			&= \biggl( \sum_{j,k=1}^m \phi_j^{\cM}(\bX) \phi_k^{\cM}(\bX) \re \tr^{\cM}(g_j^{\cM}(\bX)^* g_k^{\cM}(\bX)) \\
			& \qquad - 2 \sum_{j=1}^m \phi_j^{\cM}(\bX) \re \tr^{\cM}(g_j^{\cM} Y) + \re \tr^{\cM}(Y^*Y) \biggr)^{1/2},
		\end{align*}
		which is a quantifier-free formula by inspection.
		
		(3) $\implies$ (1).  Let $\phi$ be a quantifier-free definable predicate, and we will show that $\phi \circ \mathbf{f}$ defines a continuous function on $\mathbb{S}_{\qf,\mathbf{r}}(\rT_{\tr})$ for each $\mathbf{r} \in (0,\infty)^{\N}$ for each $\mathbf{r}$, and hence $\phi \circ \mathbf{f}$ is a quantifier-free definable predicate.  To this end, it suffices to show that for each $\mathbf{r} \in (0,\infty)$, each $\mu \in \mathbb{S}_{\qf,\mathbf{r}}(\rT_{\tr})$, and each $\epsilon > 0$, there exists an neighborhood $\cO$ of $\mu$ in $\mathbb{S}_{\qf,\mathbf{r}}(\rT_{\tr})$ such that
		\[
		|\phi^{\cM} \circ \mathbf{f}^{\cM}(\bX) - \mu(\phi \circ \mathbf{f})| < \epsilon
		\]
		whenever $\cM \models \rT_{\tr}$ and $\tp_{\qf}^{\cM}(\bX) \in \cO$.
		
		Fix $\mathbf{r}$, $\mu$, and $\epsilon$.  Let $\mathbf{r}' \in (0,\infty)^{\N}$ be such that $\mathbf{f}^{\cM}$ maps $\prod_{j \in \N} D_{r_j}^{\cM}$ into $\prod_{j \in \N} D_{r_j'}^{\cM}$ for $\cM \models \rT$.  By the uniform continuity property of definable predicates, there exists $F \subseteq \N$ finite and $\delta > 0$ such that
		\[
		\mathbf{Y}, \mathbf{Y}' \in \prod_{j \in \N} D_{r_j'}^{\cM} \text{ and } \max_{k \in F} d^{\cM}(Y_k,Y_k') < \delta \implies |\phi^{\cM}(\bY) - \phi^{\cM}(\bY')| < \frac{\epsilon}{3}.
		\]
		By (3), for each $k \in F$, choose a term $g_k$ and open $\cO_k \subseteq \mathbb{S}_{\mathbf{r}}(\rT_{\tr})$ such that, for all $\cM \models \rT$ and $\bX \in \prod_{j \in \N} D_{r_j}^{\cM}$,
		\[
		g_k^{\cM}(\bX) \in D_{r_k'}
		\]
		and
		\[
		\tp_{\qf}^{\cM}(\bX) \in \cO_k \implies d^{\cM}(f_k^{\cM}(\bX), g^{\cM}(\bX)) < \delta.
		\]
		For $k \not \in F$, let $g_k = 0$.  Then, by our choice of $\delta$ and $g_k^{\cM}$, for all $\cM \models \rT$ and $\bX \in \prod_{j \in \N} D_{r_j}^{\cM}$,
		\[
		\tp_{\qf}^{\cM}(\bX) \in \bigcap_{k \in F} \cO_k \implies |\phi^{\cM}(\mathbf{f}^{\cM}(\bX)) - \phi^{\cM}(\mathbf{g}^{\cM}(\bX))| < \frac{\epsilon}{3}.
		\]
		Moreover, $\mathbf{g} = (g_k)_{k \in \N}$ is an $\N$-tuple of terms, hence $\mathbf{g}$ is a quantifier-free definable function.  This implies that $\phi \circ \mathbf{g}$ is a quantifier-free definable predicate.  This implies that
		\[
		\cO' := \{\nu \in \mathbb{S}_{\qf,\mathbf{r}}(\rT_{\tr}): |\nu(\phi \circ \mathbf{g}) - \mu(\phi \circ \mathbf{g})| < \epsilon / 3\}
		\]
		is open in $\mathbb{S}_{\qf,\mathbf{r}}(\rT_{\tr})$.  Let
		\[
		\cO := \cO' \cap \bigcap_{k=1}^m \cO_k.
		\]
		Then $\cM \models \rT_{\tr}$ and $\tp_{\qf}^{\cM}(\bX) \in \cO$ implies that
		\begin{multline*}
			|\phi^{\cM} \circ \mathbf{f}^{\cM}(\bX) - \mu(\phi \circ \mathbf{f})| \\
			\leq |\phi^{\cM} \circ \mathbf{f}^{\cM}(\bX) - \phi^{\cM} \circ \mathbf{g}^{\cM}(\bX)| + |\phi^{\cM} \circ \mathbf{g}^{\cM}(\bX) - \mu(\phi \circ \mathbf{g})| + |\mu(\phi \circ \mathbf{g}) - \mu(\phi \circ \mathbf{f})| < \epsilon
		\end{multline*}
		as desired.
	\end{proof}
	
	\begin{example} \label{ex:functionalcalculus}
		Suppose $\rho_j \in C(\R)$, and let $f_j^{\cM}(\bX) = \rho_j(\frac{1}{2}(X_j+X_j^*))$ for $\cM \models \rT$, where $\rho_j(\frac{1}{2}(X_j+X_j^*))$ is defined through functional calculus.  By approximating $\rho_j$ uniformly on $[-r_j,r_j]$ by a polynomial $g_j$ and applying the spectral theorem, we can verify Theorem \ref{thm:qfdeffunc2} (3) and hence conclude that $\mathbf{f}$ is a quantifier-free definable function relative to $\rT_{\tr}$.  Similarly, if $f_j(\mathbf{X}) = X_j \rho_j(X_j^*X_j)$, then $\mathbf{f}$ is a quantifier-free definable function, an observation that we will use in the proof of Proposition \ref{prop:deffuncrealization}.  In this way, continuous functional calculus fits into the larger model-theoretic framework of quantifier-free definable functions.
	\end{example}
	
	Building on Theorem \ref{thm:qfdeffunc2} and Example \ref{ex:functionalcalculus}, we can show that every element of the $\mathrm{W}^*$-algebra can be realized as a quantifier-free definable function applied to the generators.  This fact will be use later on to show that covering entropy remains invariant under change of generators for a tracial $\mathrm{W}^*$-algebra.  This is a version of \cite[Proposition 13.6.6]{JekelThesis} and \cite[Proposition 2.4]{HJNS2021}, and the idea behind the proof is a ``forced limit'' construction (see \cite[\S 3.2]{BYU2010}, \cite[\S 9, definable predicates]{BYBHU2008}) applied to quantifier-free definable functions rather than quantifier-free definable predicates.
	
	\begin{proposition} \label{prop:deffuncrealization}
		If $\mathcal{M} = (M,\tau)$ is a tracial $\mathrm{W}^*$-algebra and $\mathbf{X} \in \prod_{j \in \N} D_{r_j}$ generates $M$ and $\bY \in \prod_{j \in \N} D_{r_j'}^{\cM}$, then there exists a quantifier-free definable function $\mathbf{f}$ in $\cL_{\tr}$ relative to $\rT_{\tr}$ such that $\bY = \mathbf{f}(\mathbf{X})$.  In fact, $\mathbf{f}$ can be chosen so that $f_k^{\cM}$ maps $\prod_{j \in \N} L^\infty(\cM)$ into $\prod_{j \in \N} D_{r_j'}^{\cM}$ for all $\cM \models \rT$.
	\end{proposition}
	
	\begin{proof}
		Arguing as in (2) $\implies$ (3) of Theorem \ref{thm:qfdeffunc2}, for each $k \in \N$  and $m \in \N$, there exists a non-commutative polynomial $g_{k,m}$ such that
		\[
		\norm{g_{k,m}^{\cN}(\bX')}_{L^\infty(\cN)} \leq r_k'
		\]
		for $\cN \models \rT_{\tr}$ and $\bX' \in \prod_{j \in \N} D_{r_j}^{\cN}$ and
		\[
		d^{\cM}(g_{k,m}^{\cM}(\bX),Y_k) < \frac{1}{2^{m+1}}.
		\]
		Hence also
		\[
		d^{\cM}(g_{k,m}^{\cM}(\bX),g_{k,m+1}^{\cM}(\bX)) < \frac{3}{2^{m+2}} < \frac{1}{2^m}.
		\]
		Let $F_m: \R \to \R$ be the continuous function such that $F_m(t) = 0$ for $t \leq 3/ 2^{m+2}$ and $F_m(t) = 1$ for $t \geq 1/2^m$ and $F_m$ is affine on $[3/2^{m+2},1/2^m]$.  Then
		\[
		\phi_{k,m}(\mathbf{x}) = F_k(d(g_{k,m}(\mathbf{x}),g_{k,m+1}(\mathbf{x})))
		\]
		is a quantifier-free formula.  Moreover, by construction, $\phi_{k,m}^{\cM}(\mathbf{X}) = 1$ for our given $\cM$ and $\bX$, while at the same time $\phi_{k,m}^{\cN}(\bX')$ is zero whenever $\norm{g_{k,m}^{\cN}(\bX') - g_{k,m+1}^{\cN}(\bX'))}_{L^2(\cN)} > 1/2^m$ for any $\cN \models \rT_{\tr}$ and $\bX' \in \prod_{j \in \N} D_{r_j}^{\cN}$.  Let
		\[
		\psi_{k,m} = \phi_{k,1} \phi_{k,2} \dots \phi_{k,m}.
		\]
		Then $\psi_{k,m}$ satisfies the same properties that we just showed for $\phi_{k,m}$ with the additional property that $\psi_{k,m+1} \leq \psi_{k,m}$.
		
		For $\cN \models \rT_{\tr}$ and $\bX' \in L^\infty(\cN)^{\N}$, define
		\begin{align*}
			f_{k,m}^{\cN}(\bX') &:= g_{k,1}^{\cN}(\bX') + \sum_{j=1}^{m-1} \psi_{k,j}^{\cN}(\bX') (g_{k,j+1}^{\cN}(\bX') - g_{k,j}^{\cN}(\bX')) \\
			&= (1 - \psi_{k,1})^{\cN}(\bX') g_{k,1}^{\cN}(\bX') + \sum_{j=2}^{m-1} (\psi_{k,j-1} - \psi_{k,j})^{\cN}(\bX') g_{k,j}^{\cN}(\bX') + \psi_{k,m-1}^{\cN}(\bX') g_{k,m}^{\cN}(\bX').
		\end{align*}
		Then $\mathbf{f}_{\cdot,m} = (f_{k,m})_{k \in \N}$ is a quantifier-free definable function by Theorem \ref{thm:qfdeffunc2} since it is equal to a finite sum of quantifier-free formulas multiplied by terms.  Observe that for $\cN \models \rT_{\tr}$ and $\bX' \in \prod_{j \in \N} D_{r_j}^{\cN}$,
		\[
		\norm{f_{k,m}^{\cN}(\bX')}_{L^\infty(\cN)} \leq (1 - \psi_{k,1})^{\cN}(\bX') r_k' + \sum_{j=2}^{m-1} (\psi_{k,j-1} - \psi_{k,j})^{\cN}(\bX') r_k' + \psi_{k,m-1}^{\cN}(\bX') r_k' = r_k',
		\]
		relying on the fact that $\psi_{k,j} \leq \psi_{k,j-1}$.  Furthermore, for $\bX' \in \prod_{j \in \N} D_{r_j}^{\cN}$, we have
		\[
		\norm{f_{k,m}^{\cN}(\bX') - f_{k,m+1}^{\cN}(\bX')}_{L^2(\cN)} = \psi_{k,m}^{\cN}(\bX') \norm{g_{k,m}^{\cN}(\bX') - g_{k,m+1}^{\cN}(\bX')}_{L^2(\cN)} \leq \frac{1}{2^m}.
		\]
		This implies that for $\mathbf{X}' \in \prod_{j \in \N} D_{r_j}^{\cN}$, the sequence $f_{k,m}^{\cN}(\bX')$ converges as $m \to \infty$ to some $f_k^{\cN}(\bX')$ with
		\[
		\norm{f_k^{\cN}(\bX') - f_{k,m}^{\cN}(\bX')}_{L^2(\cN)} \leq \frac{1}{2^{m-1}}.
		\]
		Furthermore, for our given $\cM$ and $\bX$, we have
		\[
		f_k^{\cM}(\bX) = g_{k,1}^{\cM}(\bX) + \sum_{j=1}^\infty (g_{k+1,j}^{\cM}(\bX) - g_{k,j}^{\cM}(\bX)) = Y_k
		\]
		because we assumed that $g_{k,j}^{\cM}(\bX) \to Y_k$ as $j \to \infty$.
		
		Now this $f_k$ is only well-defined a priori on $\prod_{j \in \N} D_{r_j}$ for our fixed choice of $\mathbf{r}$.  In order to extend it to a global function, we use a cut-off trick based on Example \ref{ex:functionalcalculus}.  Let $\rho_j: \R \to \R$ be the function
		\[
		\rho_j(t) = \begin{cases} 1, & t \leq r_j^2, \\ r_jt^{-1/2}, & t \geq r_j^2. \end{cases}
		\]
		Let $\mathbf{h}$ be given by $h_j^{\cN}(\bX') = X_j' \rho_j((X_j')^*X_j)$.  Then $\mathbf{h}$ is a quantifier-free definable function relative to $\rT_{\tr}$ by Example \ref{ex:functionalcalculus}.  Moreover, if $\cN \models \rT_{\tr}$ and $\bX' \in $, we have
		\begin{multline*}
			\norm{h_j^{\cN}(\mathbf{X}')}_{L^\infty(\cN)}^2 = \norm{h_j^{\cN}(\mathbf{X}')^* h_j^{\cN}(\mathbf{X}')}_{L^\infty(\cN)} \\
			= \norm{\rho_j((X_j')^*X_j')) (X_j')^* X_j' \rho_j((X_j')^* X_j')}_{L^\infty(\cN)} \leq r_j^2,
		\end{multline*}
		since $\rho_j(t)^2 t \leq r_j^2$.  Therefore, $\mathbf{h}^{\cN}$ maps $L^\infty(\cN)^{\N}$ into $\prod_{j \in \N} D_{r_j}^{\cN}$ for all $\cN \models \rT_{\tr}$.  Also, $\mathbf{h}^{\cN}(\bX') = \bX'$ for $\bX' \in \prod_{j \in \N} D_{r_j}^{\cN}$.
		
		Now $\mathbf{f}_{\cdot,m} \circ \mathbf{h}$ is a quantifier-free definable function since it is a composition of quantifier-free definable functions.  Because $f_{k,m}$ converges to $f_k$ uniformly on $\prod_{j \in \N} D_{r_j}$ as $m \to \infty$, we see that $f_{k,m} \circ \mathbf{h}$ converges uniformly to $f_k \circ \mathbf{h}$ globally as $m \to \infty$.  This implies that $\mathbf{f} \circ \mathbf{h}$ is a quantifier-free definable function because quantifier-free functions are closed under limits that are uniform on each product of domains (for instance, using Theorem \ref{thm:qfdeffunc2} (3) or (4)).  Moreover, $\mathbf{f}^{\cM}(\mathbf{h}^{\cM}(\bX)) = \mathbf{f}^{\cM}(\bX) = \bY$ by construction.  Finally, $(\mathbf{f} \circ \mathbf{h})^{\cN}$ maps into $\prod_{j \in \N} D_{r_j'}^{\cN}$ for all $\cN \models \rT_{\tr}$ since $(f_{k,m} \circ \mathbf{h})^{\cN}$ maps $\prod_{j \in \N} D_{r_j}^{\cN}$ into $D_{r_k'}^{\cN}$.
	\end{proof}
	
	\begin{remark} \label{rem:predicateextension}
		We can also deduce from the proof that every continuous function $\gamma$ on $\mathbb{S}_{\qf,\mathbf{r}}(\rT_{\tr})$ extends to a continuous function on $\mathbb{S}_{\qf}(\rT_{\tr})$, namely $\gamma \circ \mathbf{h}_*$ where $\mathbf{h}$ is as in the proof.  In other words, every quantifier-free definable predicate on $\prod_{j \in \N} D_{r_j}$ relative to $\rT_{\tr}$ extends to a global quantifier-free definable predicate.  The same can be said for definable predicates, dropping the word ``quantifier-free'' in this argument.
	\end{remark}
	
	Proposition \ref{prop:deffuncrealization} also leads to a proof of the following fact, which is well-known among $\mathrm{W}^*$-algebraists:
	
	\begin{lemma} \label{lem:lawisomorphism}
		Let $\bX$ be an $\N$-tuple in a tracial $\mathrm{W}^*$-algebra $\cM$ and $\bY$ an $\N$-tuple in a tracial $\mathrm{W}^*$-algebra $\cN$.  Let $\mathrm{W}^*(\bX)$ and $\mathrm{W}^*(\bY)$ be the $\mathrm{W}^*$-subalgebras generated by $\bX$ and $\bY$ with the traces obtained from restricting the traces on $\cM$ and $\cN$ respectively.  Then the following are equivalent:
		\begin{enumerate}[(1)]
			\item $\tp_{\qf}^{\cM}(\bX) = \tp_{\qf}^{\cN}(\bY)$.
			\item There exists a trace-preserving $*$-isomorphism $\sigma: \mathrm{W}^*(\bX) \to \mathrm{W}^*(\bY)$ such that $\sigma(\bX) = \bY$.
		\end{enumerate}
	\end{lemma}
	
	\begin{proof}
		(2) $\implies$ (1).  If such a $*$-isomorphism $\sigma$ exists, then for every $p \in \C\ip{x_i,x_i^*: i \in \N}$, we have $\tau_{\cM}(p(\bX)) = \tau_{\cN}(\sigma(p(\bX))) = \tau_{\cN}(p(\bY))$.  Hence, every atomic formula evaluates to the same thing on $\bX$ and on $\bY$.  Since general quantifier-free formulas are obtained by applying continuous connectives to atomic formulas, it follows by induction on complexity that $\phi^{\cM}(\bX) = \phi^{\cN}(\bY)$ for any quantifier-free formula $\bY$, and hence $\tp_{\qf}^{\cM}(\bX) = \tp_{\qf}^{\cN}(\bY)$.
		
		(1) $\implies$ (2).  Let $\cF_{\qf}(\rT_{\tr})$ be the set of quantifier-free definable functions in $\cL_{\tr}$ with respect to $\rT_{\tr}$.  Since quantifier-free functions are closed under composition, $\cF_{\qf}(\rT_{\tr})$ is a $*$-algebra.  Moreover, the evaluation maps $\alpha: \cF_{\qf}(\rT_{\tr}) \to \cM, f \mapsto f^{\cM}(\bX)$ and $\beta: \cF_{\qf}(\rT_{\tr}) \to \cN, f \mapsto f^{\cN}(\bY)$ are $*$-homomorphisms, and by the previous proposition the images of $\alpha$ and $\beta$ are $\mathrm{W}^*(\bX)$ and $\mathrm{W}^*(\bY)$ respectively.  Since $\re \tr(f)$ and $\im \tr(f)$ are quantifier-free definable predicates, $\tau_{\cM} \circ \alpha(f) = \tr(f)^{\cM}(\bX) = \tr(f)^{\cN}(\bY) = \tau_{\cN} \circ \beta(f)$ for $f \in \cF_{\qf}(\rT_{\tr})$, hence also $\norm{\alpha(f)}_{L^2(\cM)} = \norm{\beta(f)}_{L^2(\cN)}$.  This implies that $\ker \alpha = \ker \beta$.  Therefore, we obtain a $*$-isomorphism $\mathrm{W}^*(\bX) \cong \cF_{\qf}(\rT_{\tr}) / \ker \alpha = \cF_{\qf}(\rT_{\tr}) / \ker \beta \cong \mathrm{W}^*(\bY)$, which is trace-preserving since $\tau_{\cM} \circ \alpha = \tau_{\cN} \circ \beta$.
	\end{proof}
	
	\section{Entropy for types} \label{sec:entropy}
	
	We define a version of Hayes' $1$-bounded entropy for types rather than only quantifier-free types.  Later, in \S \ref{sec:qfentropy}, we will see that Hayes' $1$-bounded entropy of $\cN$ in the presence of $\cM$ (denoted $h^{\cU}(\cN:\cM)$) can be realized as a special case of entropy for a closed subset of the type space.
	
	\subsection{Definition of covering entropy}
	
	If $\cK$ is a subset of the type space $\mathbb{S}(\rT_{\tr})$ and $\mathbf{r} \in (0,\infty)^{\N}$, we define
	\[
	\Gamma_{\mathbf{r}}^{(n)}(\cK) = \left\{\mathbf{X} \in \prod_{j \in \N} D_{r_j}^{M_n(\C)}: \tp^{M_n(\C)}(\bX) \in \cK \right\}.
	\]
	We view this as a microstate space as in Voiculescu's free entropy theory.  We will then define the entropy of $\cO$ through covering numbers of $\Gamma_{\mathbf{r}}^{(n)}(\cO)$ up to unitary conjugation.
	
	\begin{definition}[Orbital covering numbers]
		Given $\Omega \subseteq M_n(\C)^{\N}$ and a finite $F \subseteq \N$ and $\epsilon > 0$, we define $N_{F,\epsilon}^{\orb}(\Omega)$ to be the set of $\mathbf{Y} \in M_n(\C)^{\N}$ such that there exists a unitary $U$ in $M_n(\C)$ and $\mathbf{X} \in \Omega$ such that $\norm{Y_i - UX_iU^*}_2 < \epsilon$ for all $i \in F$.  If $\Omega \subseteq N_{F,\epsilon}^{\orb}(\Omega')$, we say that $\Omega'$ \emph{orbitally $(F,\epsilon)$-covers} $\Omega$.  We denote by $K_{F,\epsilon}^{\orb}(\Omega)$ the minimum cardinality of a set $\Omega'$ that orbitally $(F,\epsilon)$-covers $\Omega$.
	\end{definition}
	
	\begin{definition}
		Fix a non-principal ultrafilter $\cU$ on $\N$.  For a subset $\mathcal{K}$ of the $\mathbb{S}(\rT_{\tr})$ and $F \subseteq I$ finite and $\epsilon > 0$, we define
		\[
		\Ent_{\mathbf{r},F,\epsilon}^{\mathcal{U}}(\cK) = \inf_{\text{open }\cO \supseteq \mathcal{K}} \lim_{n \to \cU} \frac{1}{n^2} \log K_{F,\epsilon}^{\orb}(\Gamma_{\mathbf{r}}^{(n)}(\cO)).
		\]
	\end{definition}
	
	\begin{observation}[Monotonicity] \label{obs:monotonicity}
		Let $\mathcal{K}' \subseteq \mathcal{K} \subseteq \mathbb{S}(\rT_{\tr})$, let $F' \subseteq F \subseteq \N$ finite, let $0 < \epsilon \leq \epsilon'$, and let $\mathbf{r}$, $\mathbf{r}' \in (0,\infty)^{\N}$ with $r_j' \leq r_j$.  Then
		\[
		\Ent_{\mathbf{r}',F',\epsilon'}^{\cU}(\cK') \leq \Ent_{\mathbf{r},F,\epsilon}^{\cU}(\cK).
		\]
		In particular, if $\cO \subseteq \mathbb{S}(\rT_{\tr})$ is open, then
		\[
		\Ent_{\mathbf{r},F,\epsilon}^{\cU}(\cO) = \lim_{n \to \cU} \frac{1}{n^2} \log K_{F,\epsilon}^{\orb}(\Gamma_{\mathbf{r}}^{(n)}(\cO)).
		\]
	\end{observation}
	
	\begin{definition}[Entropy for types]
		For $\mathcal{K} \subseteq \mathbb{S}(\rT)$, define
		\[
		\Ent_{\mathbf{r}}^{\cU}(\cK) := \sup_{\substack{\text{finite } F \subseteq \N \\ \epsilon > 0} } \Ent_{\mathbf{r},F,\epsilon}^{\cU}(\cK).
		\]
		and
		\[
		\Ent^{\cU}(\cK) := \sup_{\mathbf{r} \in (0,\infty)^{\N}} \Ent_{\mathbf{r}}^{\cU}(\cK).
		\]
		Moreover, if $\mu \in \mathbb{S}(\rT_{\tr})$, we define $\Ent^{\cU}(\mu) = \Ent^{\cU}(\{\mu\})$.
	\end{definition}
	
	
	\subsection{Variational principle}
	
	In this section, we show that the covering entropy defines an upper semi-continuous function on the type space, and then deduce a variational principle for the entropy of a closed set, in the spirit of various results in the theory of entropy and large deviations.
	
	\begin{lemma}[Upper semi-continuity] \label{lem:USC}
		For each $\mathbf{r} \in (0,\infty)^{\N}$, $F \subseteq \N$ finite, and $\epsilon > 0$, the function $\mu \mapsto \Ent_{\mathbf{r},F,\epsilon}^{\cU}(\mu)$ is upper semi-continuous on $\mathbb{S}(\rT_{\tr})$.
	\end{lemma}
	
	\begin{proof}
		For each open $\cO \subseteq \mathbb{S}(\rT_{\tr})$, let
		\[
		f_{\cO}(\mu) = \begin{cases} \Ent_{\mathbf{r},F,\epsilon}^{\cU}(\cO), & \mu \in \cO \\ \infty, & \text{ otherwise.} \end{cases}
		\]
		Since $\cO$ is open, $f_{\cO}$ is upper semi-continuous.  Moreover, $\Ent_{\mathbf{r},F,\epsilon}^{\cU}(\mu)$ is the infimum of $f_{\cO}(\mu)$ over all open $\cO \subseteq \mathbb{S}(\rT_{\tr})$, and the infimum of a family of upper semi-continuous functions is upper semi-continuous.
	\end{proof}
	
	\begin{proposition}[Variational principle] \label{prop:variational}
		Let $\mathcal{K} \subseteq \mathbb{S}(\rT_{\tr})$ and let $\mathbf{r} \in (0,\infty)^{\N}$, $F \subseteq \N$ finite, and $\epsilon > 0$.  Then
		\begin{equation} \label{eq:variational1}
			\sup_{\mu \in \cK} \Ent_{\mathbf{r},F,\epsilon}^{\cU}(\mu) \leq \Ent_{\mathbf{r},F,\epsilon}^{\cU}(\cK) \leq \sup_{\mu \in \operatorname{cl}(\cK) } \Ent_{\mathbf{r},F,\epsilon}^{\cU}(\mu).
		\end{equation}
		Hence,
		\begin{equation} \label{eq:variational2}
			\sup_{\mu \in \cK} \Ent^{\cU}(\mu) \leq \Ent^{\cU}(\cK) \leq \sup_{\mu \in \operatorname{cl}(\cK) } \Ent^{\cU}(\mu).
		\end{equation}
	\end{proposition}
	
	\begin{proof}
		If $\mu \in \cK$, then by monotonicity (Observation \ref{obs:monotonicity}), $\Ent_{F,\epsilon}^{\cU}(\{\mu\}) \leq \Ent_{F,\epsilon}^{\cU}(\cK)$.  Taking the supremum over $\mu \in \cK$, we obtain the first inequality of \eqref{eq:variational1}.
		
		For the second inequality of \eqref{eq:variational1}, let $C = \sup_{\mu \in \operatorname{cl}(\cK)} \Ent_{F,\epsilon}^{\cU}(\mu)$.  If $C = \infty$, there is nothing to prove.  Otherwise, let $C' > C$.  For each $\mu \in \operatorname{cl}(\cK) \cap \mathbb{S}_{\mathbf{r}}(\rT_{\tr})$, there exists some open neighborhood $\cO_\mu$ of $\mu$ in $\mathbb{S}(\rT_{\tr})$ such that $\Ent_{\mathbf{r},F,\epsilon}^{\cU}(\cO_\mu) < C'$.   Since $\{\cO_\mu\}_{\mu \in \operatorname{cl}(\cK) \cap \mathbb{S}_{\mathbf{r}}(\rT_{\tr})}$ is an open cover of the compact set $\operatorname{cl}(\cK) \cap \mathbb{S}_{\mathbf{r}}(\rT_{\tr})$, there exist $\mu_1$, \dots, $\mu_k \in \operatorname{cl}(\cK) \cap \mathbb{S}_{\mathbf{r}}(\rT_{\tr})$ such that
		\[
		\cK \cap \mathbb{S}_{\mathbf{r}}(\rT_{\tr}) \subseteq \bigcup_{j=1}^k \cO_{\mu_j}.
		\]
		Let $\cO = \bigcup_{j=1}^k \cO_{\mu_j}$.  Then
		\[
		K_{F,\epsilon}^{\orb}(\Gamma_{\mathbf{r}}^{(n)}(\cO)) \leq \sum_{j=1}^k K_{F,\epsilon}^{\orb}(\Gamma_{\mathbf{r}}^{(n)}(\cO_{\mu_j})) \leq k \max_j K_{F,\epsilon}^{\orb}(\Gamma_{\mathbf{r}}^{(n)}(\cO_{\mu_j})).
		\]
		Thus,
		\[
		\frac{1}{n^2} \log K_{F,\epsilon}^{\orb}(\Gamma_{\mathbf{r}}^{(n)}(\cO)) \leq \frac{1}{n^2} \log k + \max_j \frac{1}{n^2} \log K_{F,\epsilon}^{\orb}(\Gamma_{\mathbf{r}}^{(n)}(\cO_{\mu_j})).
		\]
		Taking the limit as $n \to \cU$,
		\[
		\Ent_{\mathbf{r},F,\epsilon}^{\cU}(\operatorname{cl}(\cK)) \leq  \Ent_{\mathbf{r},F,\epsilon}^{\cU}(\cO) \leq \max_j \Ent_{\mathbf{r},F,\epsilon}^{\cU}(\cO_{\mu_j}) \leq C'.
		\]
		Since $C' > C$ was arbitrary,
		\[
		\Ent_{\mathbf{r},F,\epsilon}^{\cU}(\cK) \leq \Ent_{\mathbf{r},F,\epsilon}(\operatorname{cl}(\cK)) \leq C = \sup_{\mu \in \operatorname{cl}(\cK)} \Ent_{\mathbf{r},F,\epsilon}^{\cU}(\mu),
		\]
		completing the proof of \eqref{eq:variational1}.  Taking the supremum over $F$ and $\epsilon$ and $\mathbf{r}$ in \eqref{eq:variational1}, we obtain \eqref{eq:variational2}.
	\end{proof}
	
	\subsection{Invariance under change of coordinates} \label{subsec:invariance}
	
	Next, we prove certain invariance properties of the covering entropy.  First, $\Ent_{\mathbf{r}}^{\cU}(\mu)$ is independent of $\mathbf{r}$ provided that $\mu \in \mathbb{S}_{\mathbf{r}}(\rT_{\tr})$.  Second, if $\mathrm{W}^*(\bX) = \mathrm{W}^*(\bY)$ inside $\cM$, then $\Ent^{\cU}(\tp^{\cM}(\bX)) = \Ent^{\cU}(\tp^{\cM}(\bY))$, which allows us to define $\Ent^{\cU}(\cN,\cM)$ for a $\mathrm{W}^*$-subalgebra $\cN$ inside $\cM$.  Both of these properties are deduced from the following lemma about push-forward under definable functions.  This is closely related to \cite[Lemma A.8 and Theorem A.9]{Hayes2018}.
	
	\begin{proposition}[Monotonicity under push-forward] \label{prop:pushforward}
		Let $\mathbf{f}$ be a definable function relative to $\rT_{\tr}$, let $\mathbf{r} \in (0,\infty)^{\N}$, and let $\mathbf{r}' \in (0,\infty)^{\N}$ be such that $\mathbf{f}$ maps $\prod_{j \in \N} D_{r_j}$ into $\prod_{j \in \N} D_{r_j'}$.  Let $\cK \subseteq \mathbb{S}_{\mathbf{r}}(\rT_{\tr})$ be closed.  Then
		\[
		\Ent_{\mathbf{r}'}^{\cU}(\mathbf{f}_*(\cK)) \leq \Ent_{\mathbf{r}}^{\cU}(\cK).
		\]
	\end{proposition}
	
	\begin{remark} \label{rem: monotonicity}
		The analogous monotonicity property does not hold for the original $1$-bounded entropy $h$ of a quantifier-free type, but it does hold for $1$-bounded entropy in the presence.  The monotonicity property holds for the full type and for the existential type of $\bX$ because those types already encode information about how $\bX$ interacts with the ambient algebra.  For more information, see Remark \ref{rem:qfnomonotonicity}.
	\end{remark}
	
	\begin{proof}
		Let $F' \subseteq \N$ finite and $\epsilon' \in (0,1)$ be given.  Because $\mathbf{f}$ is a definable function, it is uniformly continuous by Lemma \ref{lem:deffuncunifcont}, hence there exists a finite $F \subseteq I$ and $\epsilon > 0$ such that for every $\cM \models \rT_{\tr}$ and $\mathbf{X}$, $\mathbf{Y} \in \prod_{j \in \N} D_{r_j}^{\cM}$,
		\begin{equation} \label{eq:pushforwardestimate}
			\norm{X_j - Y_j}_2 < \epsilon \text{ for all } j \in F \implies \norm{f_{j'}(\mathbf{X}) - f_{j'}(\mathbf{Y})}_2 < \epsilon'/3 \text{ for all } j' \in F'.
		\end{equation}
		
		Let $\cO$ be a neighborhood of $\cK$ in $\mathbb{S}_{\mathbf{r}}(\rT_{\tr})$.  By Urysohn's lemma, there exists a continuous function $\psi: \mathbb{S}_{\mathbf{r}}(\rT_{\tr}) \to [0,1]$ such that $\phi = 0$ on $\cK$ and $\phi = 1$ on $\mathbb{S}_{\mathbf{r}}(\rT_{\tr}) \setminus \cO$.  As in Proposition \ref{prop:defpredcontinuous}, there exists a formula $\eta$ such that $|\eta^{\cM} - \phi^{\cM}| < \epsilon'/3$ on $\prod_{j \in \N} D_{r_j}^{\cM}$.  Next, define $\psi^{\cM}: \prod_{j \in \N} S_{r_j'}^{\cM} \to \R$ by
		\[
		\psi^{\cM}(\mathbf{Y}) = \inf_{\mathbf{X} \in \prod_{j \in \N} D_j^{\cM}} \left( \eta^{\cM}(\mathbf{X}) + \max_{j' \in F'} d^{\cM}(f_{j'}^{\cM}(\mathbf{X}),Y_{j'}) \right),
		\]
		which is a definable predicate relative to $\rT_{\tr}$ by Lemma \ref{lem:defpredoperations}.
		
		Viewing $\psi$ as a continuous function on $\mathbb{S}_{\mathbf{r}'}(\rT_{\tr})$, let $\cO' = \psi^{-1}((-\infty,2\epsilon'/3))$.  Note that $\mathbf{f}_*(\cK) \subseteq \cO'$ since if $\mathbf{Y} = \mathbf{f}^{\cM}(\mathbf{X})$, then we can take this value of $\mathbf{X}$ in the infimum defining $\psi$ and obtain that $\psi^{\cM}(\mathbf{Y}) \leq \epsilon' / 3$.  Meanwhile, if $\cM \models \rT_{\tr}$ and $\mathbf{Y} \in \prod_{j \in \N} D_{r_j'}^{\cM}$ with $\tp^{\cM}(\mathbf{Y}) \in \cO'$, then there exists $\mathbf{X} \in \prod_{j \in \N} D_{r_j}^{\cM}$ with
		\[
		\eta^{\cM}(\mathbf{X}) + \max_{j' \in F'} d^{\cM}(f_{j'}^{\cM}(\mathbf{X}), \mathbf{Y}_{j'}) < \frac{2\epsilon'}{3},
		\]
		which implies that $\tp^{\cM}(\mathbf{X}) \in \cO$ and $\max_{i' \in F'} \norm{f_{i'}^{\cM}(\mathbf{Y}) - \mathbf{X}_{i'}} < 2\epsilon'/3$.  Applying this with $\cM = M_n(\C)$, we obtain
		\[
		\Gamma^{(n)}(\cO') \subseteq N_{2\epsilon'/3}(\mathbf{f}^{M_n(\C)}(\Gamma^{(n)}(\cO))).
		\]
		
		If $\Omega$ is an $(F,\epsilon)$-cover of $\Gamma^{(n)}(\cO)$, then by \eqref{eq:pushforwardestimate} and the fact that $\mathbf{f}$ is equivariant with respect to conjugation of an $\N$-tuple by a fixed unitary, $\mathbf{f}_*(\Omega)$ is an orbital $(F',\epsilon'/3)$-cover of $\mathbf{f}_*(\Gamma^{(n)}(\cO))$, and therefore also an orbital $(F',\epsilon')$-cover of $\Gamma^{(n)}(\cO')$.  It follows that
		\[
		K_{\mathbf{r}',F',\epsilon'}^{\orb}(\Gamma^{(n)}(\cO')) \leq K_{\mathbf{r},F,\epsilon}^{\orb}(\Gamma^{(n)}(\cO)).
		\]
		Hence,
		\[
		\Ent_{\mathbf{r}',F',\epsilon'}^{\cU}(\mathbf{f}_*(\cK)) \leq \Ent_{\mathbf{r}',F',\epsilon'}^{\cU}(\cO') \leq \Ent_{\mathbf{r},F,\epsilon}(\cO).
		\]
		Since $\cO$ was an arbitrary neighborhood of $\cK$, we obtain
		\[
		\Ent_{\mathbf{r}',F',\epsilon'}^{\cU}(\mathbf{f}_*(\cK)) \leq \Ent_{\mathbf{r},F,\epsilon}^{\cU}(\cK) \leq \Ent_{\mathbf{r}}^{\cU}(\cK).
		\]
		Since $F'$ and $\epsilon'$ were arbitary, we conclude that $\Ent_{\mathbf{r}'}^{\cU}(\mathbf{f}_*(\cK)) \leq \Ent_{\mathbf{r}}^{\cU}(\cK)$, as desired.
	\end{proof}
	
	\begin{corollary} \label{cor:independentofr}
		If $\cK$ is a closed subset of $\mathbf{S}_{\mathbf{r}}(\rT_{\tr})$, then $\Ent^{\cU}(\cK) = \Ent_{\mathbf{r}}^{\cU}(\cK)$.
	\end{corollary}
	
	\begin{proof}
		By definition, $\Ent^{\cU}(\cK) \geq \Ent_{\mathbf{r}}^{\cU}(\cK)$.  On the other hand, fix some $\mathbf{r}' \in (0,\infty)^{\N}$ and let $\mathbf{r}'' = \max(\mathbf{r}',\mathbf{r})$.  By Observation \ref{obs:monotonicity},
		\[
		\Ent_{\mathbf{r}'}^{\cU}(\cK) \leq \Ent_{\mathbf{r}''}^{\cU}(\cK).
		\]
		Now applying Proposition \ref{prop:pushforward} to the identity map, since $\id$ maps $\prod_{j \in \N} D_{r_j}$ into $\prod_{j \in \N} D_{r_j''}$, it follows that
		\[
		\Ent_{\mathbf{r}''}^{\cU}(\cK) \leq \Ent_{\mathbf{r}}^{\cU}(\cK).
		\]
		Since $\mathbf{r}'$ was arbitrary, $\Ent^{\cU}(\cK) \leq \Ent_{\mathbf{r}}^{\cU}(\cK)$.
	\end{proof}
	
	\begin{corollary} \label{cor:invariance}
		Let $\cM = (M,\tau)$ be a tracial $\mathrm{W}^*$-algebra and $\bX$, $\bY \in M^{\N}$.  If $\bY \in \mathrm{W}^*(\bX)^{\N}$, then
		\[
		\Ent^{\cU}(\tp^{\cM}(\bY)) \leq \Ent^{\cU}(\tp^{\cM}(\bX)).
		\]
		In particular, if $\mathrm{W}^*(\bX) = \mathrm{W}^*(\bY)$, then $\Ent^{\cU}(\tp^{\cM}(\bX)) = \Ent^{\cU}(\tp^{\cM}(\bY))$.
	\end{corollary}
	
	\begin{proof}
		By Proposition \ref{prop:deffuncrealization}, there exists a quantifier-free definable function $\mathbf{f}$ relative to $\rT_{\tr}$ such that $\bY = \mathbf{f}^{\cM}(\bX)$.  Now $\tp^{\cM}(\bY) = \mathbf{f}_* \tp^{\cM}(\bX)$.  Hence, applying Proposition \ref{prop:pushforward} to $\cK = \{\tp^{\cM}(\bX)\}$ (for an appropriate choice of $\mathbf{r}$), we obtain $\Ent^{\cU}(\tp^{\cM}(\bY)) \leq \Ent^{\cU}(\tp^{\cM}(\bX))$.  The second claim follows by symmetry.
	\end{proof}
	
	With this invariance result in hand, it seems natural to define the covering entropy for a separable $\mathrm{W}^*$-subalgebra of $\cM$ as the entropy of any $\N$-tuple of generators.  However, the following definition works even in the non-separable case.
	
	\begin{definition} \label{def:W*entropy}
		If $\cM = (M,\tau)$ is a tracial $\mathrm{W}^*$-algebra and $\cN$ is a $\mathrm{W}^*$-subalgebra, we define
		\[
		\Ent^{\cU}(\cN:\cM) = \sup_{\bX \in L^\infty(\cN)^{\N}} \Ent^{\cU}(\tp^{\cM}(\bX)).
		\]
	\end{definition}
	
	\begin{observation} \label{obs:monotonicity2}
		Let $\cM$ be a tracial $\mathrm{W}^*$-algebra, and let $\cN$ be a $\mathrm{W}^*$-subalgebra.  If $\bX \in L^\infty(\cN)_{\sa}^{\N}$ generates $\cN$, then for any $\bY \in \cN^{\N}$, we have $\Ent^{\cU}(\tp^{\cM}(\bY)) \leq \Ent^{\cU}(\tp^{\cM}(\bX))$ by Corollary \ref{cor:invariance}, and therefore,
		\[
		\Ent^{\cU}(\cN:\cM) = \Ent^{\cU}(\tp^{\cM}(\bX)).
		\]
		Moreover, if $\cP$ is a $\mathrm{W}^*$-subalgebra of $\cN$, then $\Ent^{\cU}(\cP:\cM) \leq \Ent^{\cU}(\cN:\cM)$.
	\end{observation}
	
	\begin{remark}
		Furthermore, it is evident from Definition \ref{def:W*entropy} that $\Ent^{\cU}(\cN:\cM)$ only depends on the set of types in $\cM$ that are realized in $L^\infty(\cN)^{\N}$.  Hence, if two embeddings $\cN \to \cM_1$ and $\cN \to \cM_2$ are elementarily equivalent---meaning that for every definable predicate $\phi$ and $\bX \in L^\infty(\cN)^{\N}$, we have $\phi^{\cM_1}(\bX) = \phi^{\cM_2}(\bX)$---then $\Ent^{\cU}(\cN:\cM_1) = \Ent^{\cU}(\cN: \cM_2)$.
	\end{remark}
	
	\subsection{Entropy and ultraproduct embeddings} \label{subsec:embeddings1}
	
	\begin{lemma}[Ultraproduct realization of types] \label{lem:ultraproductrealization}
		Let $\cQ = \prod_{n \to \cU} M_n(\C)$.  Let $\mu \in \mathbb{S}(\rT_{\tr})$. Then $\Ent^{\cU}(\mu)$ is either nonnegative or it is $-\infty$.  Moreover, $\Ent^{\cU}(\mu) \geq 0$ if and only if there exists $\mathbf{X} \in L^\infty(\cQ)^{\N}$ such that $\tp^{\cQ}(\bX) = \mu$.
	\end{lemma}
	
	\begin{proof}
		Note that $\log K_{F,\epsilon}^{\orb}(\Gamma_{\mathbf{r}}^{(n)}(\cO))$ is either $\geq 0$ or it is $-\infty$.  Therefore, $\Ent_{\mathbf{r}}^{\cU}(\mu)$ is either nonnegative or it is $-\infty$.  It remains to show the second claim of the lemma.
		
		($\implies$) In light of the foregoing argument, if $\Ent^{\cU}(\mu) \geq 0$, then $\Ent_{\mathbf{r}}^{\cU}(\mu) \geq 0$ for some $\mathbf{r}$.  By Observations \ref{obs:weak*metrizable} and \ref{obs:Ltrseparable}, $\mathbb{S}_{\mathbf{r}}(\rT_{\tr})$ is metrizable, hence there is a sequence $(\cO_k)_{k \in \N}$ of neighborhoods of $\mu$ in $\mathbb{S}(\rT)$ such that $\overline{\cO}_{k+1} \subseteq \cO_k$ and $\bigcap_{k \in \N} \cO_k = \{\mu\}$.  For $k \in \N$, let
		\[
		E_k = \{n \in \N: \Gamma_{\mathbf{r}}^{(n)}(\cO_k) \neq \varnothing\}.
		\]
		Now choose $\mathbf{X}^{(n)} \in M_N(\C)^{\N}$ as follows.  For each $n \not \in E_1$, set $\mathbf{X}^{(n)} = 0$.  For each $n \in E_k \setminus E_{k+1}$, let $\mathbf{X}^{(n)}$ be an element of $\Gamma_{\mathbf{r}}^{(n)}(\cO_k)$.  If $n \in \bigcap_{k \in \N} E_k$, that means that $\Gamma^{(n)}(\{\mu\}) \neq \varnothing$, so in this case we may choose $\mathbf{X}^{(n)} \in M_n(\C)^{\N}$ with $\tp^{M_n(\C)}(\mathbf{X}^{(n)}) = \mu$.
		
		Since $\cU$ is an ultrafilter, either $E_k \in \cU$ or $E_k^c \in \cU$.  If we had $E_k^c \in \cU$, then $\lim_{n \to \cU} (1/n^2) \log K_{F,\epsilon}^{\orb}(\Gamma^{(n)}(\cO_k))$ would be $-\infty$ since the set would be empty for $n \in E_k^c$.  Hence, $E_k \in \cU$.  For $n \in E_k$, we have $\tp^{M_n(\C)}(\mathbf{X}^{(n)}) \in \cO_k$.  Therefore, $\lim_{n \to \cU} \tp^{M_n(\C)}(\mathbf{X}^{(n)}) \in \overline{\cO}_k$.  Since this holds for all $k$, $\lim_{n \to \cU} \tp^{M_n(\C)}(\mathbf{X}^{(n)}) = \mu$.  Let $\mathbf{X} = [\mathbf{X}^{(n)}]_{n \in \N} \in L^\infty(\cQ)^{\N}$.  Then
		\[
		\tp^{\cQ}(\mathbf{X}) = \lim_{n \to \cU} \tp^{M_n(\C)}(\mathbf{X}^{(n)}) = \mu.
		\]
		
		($\impliedby$)  Suppose that $\mathbf{X}$ is an element of the ultraproduct with type $\mu$.  Let $r_j = \norm{X_j}_\infty$.   Express $\mathbf{X}$ as $[\mathbf{X}^{(n)}]_{n \in \N}$ for some $\mathbf{X}^{(n)} \in M_n(\C)^{\N}$ with $\norm{X_j^{(n)}} \leq r_j$.  Since the type of $\mathbf{X}^{(n)}$ converges to the type of $\mathbf{X}$, for every neighborhood $\cO$ of $\mu$, there exists $E \in \cU$ such that $\tp^{M_n(\C)}(\mathbf{X}^{(n)}) \in \cO$ for all $n \in E$, and in particular, $\Gamma_{\mathbf{r}}^{(n)}(\cO) \neq \varnothing$ for $n \in E$.  This implies that $\Ent_{F,\epsilon}^{\cU}(\cO) \geq 0$ for every $F$ and $\epsilon$.  Hence, $\Ent^{\cU}(\mu) \geq 0$.
	\end{proof}
	
	Recall that an embedding $\cM \to \cQ$ of tracial $\mathrm{W}^*$-algebras is said to be \emph{elementary} if for every definable predicate $\phi$ and $\bX \in L^\infty(\cM)^{\N}$, we have $\phi^{\cQ}(\bX) = \phi^{\cM}(\bX)$.  This in particular implies that $\cM$ and $\cQ$ are elementarily equivalent, that is, they have the same theory.
	
	\begin{corollary} \label{cor:embeddings}
		Suppose that $\cM$ is a separable tracial $\mathrm{W}^*$-algebra and $\cN \subseteq \cM$ is a $\mathrm{W}^*$-subalgebra.  If $\Ent^{\cU}(\cN:\cM) \geq 0$, then there exists an elementary embedding $\iota: \cM \to \cQ$.
	\end{corollary}
	
	\begin{remark}
		Since the embedding $\iota: \cM \to \cQ$ is elementary, in particular the embeddings $\cN \to \cM$ and $\cN \to \cQ$ are elementarily equivalent, and hence $\Ent^{\cU}(\iota(\cN):\cQ) = \Ent^{\cU}(\cN:\cM)$.
	\end{remark}
	
	\begin{proof}[Proof of Corollary \ref{cor:embeddings}]
		By Observation \ref{obs:monotonicity2}, $\Ent^{\cU}(\cM:\cM) \geq \Ent^{\cU}(\cN:\cM) \geq 0$.  Let $\bX \in L^\infty(\cM)^{\N}$ generate $\cM$.  Then by the previous lemma, there exists $\bX' \in \cQ$ with the same type of $\bX$.  In particular, since $\bX$ and $\bX'$ have the same quantifier-free type, Lemma \ref{lem:lawisomorphism} shows that there is an embedding $\iota: \cM \to \cQ$ with $\iota(\bX) = \bX'$.  To show that $\iota$ is elementary, suppose that $\bY \in L^\infty(\cM)^{\N}$ and $\phi$ is a definable predicate.  By Proposition \ref{prop:deffuncrealization}, there exists a quantifier-free definable function $\mathbf{f}$ such that $\bY = \mathbf{f}^{\cM}(\bX)$.  Since $\mathbf{f}$ is quantifier-free, $d^{\cQ}(\iota(Y_j), f_j^{\cQ}(\bX')) = d^{\cM}(Y_j, f_j^{\cM}(\bX)) = 0$, hence $\iota(\bY) = \mathbf{f}^{\cQ}(\iota(\bX))$.  Therefore, $\phi^{\cQ}(\iota(\bY)) = (\phi \circ \mathbf{f})^{\cQ}(\bX') = (\phi \circ \mathbf{f})^{\cM}(\bX) = \phi^{\cM}(\bY)$, where the middle equality follows because $\tp^{\cQ}(\bX') = \tp^{\cM}(\bX)$, and therefore the embedding is elementary.
	\end{proof}
	
	\subsection{Entropy and Algebraicity} \label{subsec:algebraic}
	
	In this section, we show that $\Ent^{\cU}(\cN:\cM) = \Ent^{\cU}(\acl(\cN): \cM)$, where $\acl(\cN)$ is the \emph{algebraic closure} of continuous model theory.  At present, very little is known about algebraic closures for tracial $\mathrm{W}^*$-algebras.  Nonetheless, it is natural to study how the model-theoretic $1$-bounded entropy behaves under this model-theoretic operation, analogously to how Hayes studied the behavior of $1$-bounded entropy under various $\mathrm{W}^*$-algebraic operations (see \cite[\S 2]{Hayes2018} and \cite[\S 2.3]{HJNS2021}).
	
	First, we explain the definition of algebraic closure.
	
	\begin{definition}[Algebraicity] \label{def:algebraic}
		Let $\cM$ be a structure in some language $\cL$, and let $\cN$ be a substructure.  Let $S$ be a sort in $\cL$.
		\begin{itemize}
			\item A map $\phi: S^{\cM} \to \R$ is a \emph{definable predicate in $\cM$ over $\cN$} if for every $\bD \in \prod_{j \in \N} \cD_{S_j}$ and $\epsilon > 0$, there exists a formula $\psi$ in variables $x_j$ from $S_j$ for $j \in \N$ and $y_j$ from $S_j'$ for $j \in \N$, and there exists $\bY \in \prod_{j \in \N} (S_j')^{\cN}$ such that
			\[
			|\phi(\bX) - \psi^{\cM}(\bX,\bY)| < \epsilon \text{ for all } \bX \in \prod_{j \in \N} D_j^{\cM}.
			\]
			\item If $A \subseteq S^{\cM}$, we say that $A$ is \emph{definable in $\cM$ over $\cN$} if the map $S^{\cM} \to \R: X \mapsto d^{\cM}(X,A)$ is definable in $\cM$ over $\cN$.
			\item If $a \in S^{\cM}$, we say that $a$ is \emph{algebraic over $\cN$} if there exists a compact set $A \subseteq S^{\cM}$ such that $a \in A$ and $A$ is definable in $\cM$ over $\cN$.
		\end{itemize}
	\end{definition}
	
	\begin{remark} \label{rem:definableoverN}
		It will be convenient in our arguments that for tracial $\mathrm{W}^*$-algebras $\cM$ and $\cN$, if a function $\phi: L^\infty(\cM) \to \R$ is definable in $\cM$ over $\cN$, then there exists a definable predicate $\theta$ and $\bY \in L^\infty(\cN)^{\N}$ such that $\phi(X) = \theta(X,\bY)$.  This follows by a forced-limit argument similar to Proposition \ref{prop:deffuncrealization}:  Since $\phi$ is definable in $\cM$ over $\cN$, then for each $k \in \N$, there exists a formula $\theta_k$ and $\bY_k \in L^\infty(\cN)^{\N}$ such that
		\[
		|\phi(X) - \theta_k(X,\bY_k)| < \frac{1}{2^k} \text{ for } X \in D_k^{\cM}.
		\]
		Let $\bY$ by an $\N$-tuple obtained by joining together the $\bY_k$'s into a single tuple, so that $\theta_k$ can be viewed as a definable predicate in $(X,\bY)$.  Similar to the proof of Proposition \ref{prop:deffuncrealization}, there exists a definable predicate $\psi_k$ such that $\psi_k^{\cM}(X,\bY) = 1$ and $\psi_k \cdot (\theta_{k+1} - \theta_k) < 2/2^k$ on $D_k$.  Then
		\[
		\theta := \theta_1 + \sum_{k=1}^\infty \psi_k (\theta_{k+1} - \theta_k)
		\]
		converges uniformly on every domain $D_r$ and satisfies $\phi(X) = \theta^{\cM}(X,\bY)$.
	\end{remark}
	
	\begin{definition}[Algebraic closure]
		Let $\cM$ be an $\cL$-structure and $\cN$ an $\cL$-substructure.  We define $S^{\acl(\cN)}$ to be the set of $a \in S^{\cM}$ that are algebraic in $\cM$ over $\cN$.  We let $\acl(\cN) = (S^{\acl(\cN)})_{S \in \mathcal{S}}$.  (Although we omit $\cM$ from the notation, the algebraic closure a priori depends on the ambient structure $\cM$.)
	\end{definition}
	
	For the properties of algebraic closure, see \cite[\S 10]{BYBHU2008}.  In particular, one can show that if $\cN_1 \subseteq \acl(\cN_2)$, then $\acl(\cN_1) \subseteq \acl(\cN_2)$ (``what is algebraic over the algebraic closure of $\cN_2$ is algebraic over $\cN_2$).  Moreover, one can verify directly from Definition \ref{def:algebraic} that $f$ is a term and $Y_1$, \dots, $Y_k \in \cN$, then $f^{\cM}(Y_1,\dots,Y_k) \in \acl(\cN)$.  By combining these properties, it follows that $\acl(\cN)$ is an $\cL$-substructure of $\cM$.
	
	Thus, in particular, if $\cN \subseteq \cM$ are tracial $\mathrm{W}^*$-algebras, then the algebraic closure $\acl(\cN)$ of $\cN$ in $\cM$ is a tracial $\mathrm{W}^*$-subalgebra of $\cM$ as well.  We will show that $\Ent^{\cU}(\acl(\cN):\cM) = \Ent^{\cU}(\cN:\cM)$.  We first consider the case of adjoining to an $\N$-tuple $\bX$ a single element $Y$ that is algebraic over $\mathrm{W}^*(\bX)$, and this case takes the bulk of the work.
	
	\begin{theorem}
		Let $\bX$ be an $\N$-tuple in $\cM = (M,\tau)$.  Let $Y \in M$ be algebraic over $\mathrm{W}^*(\bX)$.  Then
		\[
		\Ent^{\cU}(\tp^{\cM}(Y,\bX)) = \Ent^{\cU}(\tp^{\cM}(\bX)).
		\]
	\end{theorem}
	
	The inequality $\Ent^{\cU}(\tp^{\cM}(\bX)) \leq \Ent^{\cU}(\tp^{\cM}(Y,\bX))$ follows from Proposition \ref{prop:pushforward}, so we only need to prove the opposite inequality.
	
	The idea of the argument is that $Y$ comes from a definable compact set $A$.  We can cover $A$ by some finite number $k$ of $\epsilon$-balls.  Transferring this to the microstate approximations would tell us that for each matrix approximation $\bX'$ for $\bX$, the possible matrix approximations for $Y$ can be covered by $k$ many $\epsilon$-balls.  So the covering number for the microstate space of $(Y,\bX)$ would be at most $k$ times that of $\bX$; the factor of $k$ is negligible in the large-$n$ limit because we will take the logarithm and divide by $n^2$.
	
	\begin{proof}
		$\Ent^{\cU}(\tp^{\cM}(\bX)) \leq \Ent^{\cU}(\tp^{\cM}(Y,\bX))$ holds by Proposition \ref{prop:pushforward}.
		
		By algebraicity of $Y$ and Remark \ref{rem:definableoverN}, there exists a compact $A \subseteq \cM$, a definable predicate $\phi$ relative to $\rT_{\tr}$, and $\bX' \in \mathrm{W}^*(\bX)^{\N}$ such that $Y \in A$ and $d^{\cM}(Z,A) = \phi^{\cM}(Z,\bX')$.  Since $\bX' = \mathbf{f}(\bX)$ for some quantifier-free definable function $\mathbf{f}$, we have
		\[
		d^{\cM}(Z,A) = \phi^{\cM}(Z,\mathbf{f}(\bX)) = \psi^{\cM}(Z,\bX),
		\]
		where $\psi$ is the definable predicate given by composing $\psi$ with $\mathbf{f}$ in the coordinates $2$, $3$, \dots.
		
		Fix $\mathbf{r} = (r_j)_{j \in \N}$ such that $X_j \in D_{r_j}$ and fix $r$ such that $Z \in D_r$.  We want to show that
		\[
		\sup_{(F',\epsilon')} \inf_{\cO \ni \tp^{\cM}(Y,\bX)} \lim_{n \to \cU} \frac{1}{n^2} \log K_{F',\epsilon'}^{\orb}(\Gamma_{r,\mathbf{r}}^{(n)}(\cO')) \leq \sup_{(F,\epsilon)} \inf_{\cO \ni \tp^{\cM}(Y,\bX)} \lim_{n \to \cU} \frac{1}{n^2} \log K_{F,\epsilon}^{\orb}(\Gamma_{\mathbf{r}}^{(n)}(\cO)) 
		\]
		Here we regard $\N$ as starting at $1$, and we index the tuple $(Y,\bX)$ by $\{0\} \sqcup \N$, where the $0$ index corresponds to $Y$.  Fix $F' \subseteq \{0\} \sqcup \N$ finite and $\epsilon' > 0$.  Since enlarging $F'$ would only increase the quantity inside the $\sup_{(F',\epsilon')}$, assume without loss of generality that $F'$ contains the index $0$ corresponding to the variable $Y$, hence $F' = \{0\} \sqcup F_1$ for some $F_1 \subseteq \N$.
		
		By compactness of $A$, there exists $k \in \N$ and there exist $Y_1$, \dots, $Y_k \in A$ such that the $\epsilon'/4$-balls centered at $Y_1$, \dots, $Y_k$ cover $A$.  This implies that every point within a distance of $\epsilon'/4$ from $A$ is within a distance of $\epsilon'/2$ from one of the points $Y_1$, \dots, $Y_k$, and therefore
		\[
		\sup_{Y \in D_r^{\cM}} \min\Bigl(\epsilon'/4 - \psi^{\cM}(Y,\bX), \min(d^{\cM}(Y,Y_1),\dots,d^{\cM}(Y,Y_k))-\epsilon'/2 \Bigr) \leq 0.
		\]
		Choose $t_j \in (0,\infty)$ such that $Y_j \in D_{t_j}$.  Let $\alpha$ and $\beta$ be the definable predicates
		\begin{align*}
			\alpha(y_1,\dots,y_k,\mathbf{x}) &=\sup_{y \in D_r^{\cM}} \min(\epsilon'/4 - \psi(y,\mathbf{x}), \min(d(y,y_1),\dots,d(y,y_k)) - \epsilon'/2), \\
			\beta(\mathbf{x}) &= \inf_{y_1 \in D_{t_1}} \dots \inf_{y_k \in D_{t_k}} \alpha(y_1,\dots,y_k, \mathbf{x}),
		\end{align*}
		so that $\alpha^{\cM}(Y_1,\dots,Y_k,\bX) \leq 0$ and $\beta^{\cM}(\bX) \leq 0$.
		
		By uniform continuity (Observation \ref{obs:defpredunifcont}), there exists $F_2 \subseteq \N$ finite and $\delta > 0$ such that for all $\cN \models \rT_{\tr}$, all $Y_1' \in D_{t_1}$, \dots, $Y_k' \in D_{t_k}$, and all $\bX'$, $\bX'' \in \prod_{j \in \N} D_{r_j}^{\cN}$, we have
		\[
		\max_{j \in F_2} d^{\cN}(X_j',X_j'') < \delta \implies |\alpha^{\cN}(Y_1',\dots,Y_k',\bX') - \alpha^{\cN}(Y_1',\dots,Y_k',\bX'')| < \frac{\epsilon'}{16}.
		\]
		
		Fix a neighborhood $\cO$ of $\tp^{\cM}(\bX)$, and let
		\[
		\cO' = \{\tp^{\cN}(Y',\bX'): \cN \models \rT_{\tr}, \psi^{\cN}(Y',\bX') < \epsilon'/8, \beta^{\cN}(\bX) < \epsilon'/16, \tp^{\cN}(\bX') \in \cO \},
		\]
		which is a neighborhood of $\tp^{\cM}(Y,\bX)$.  Let $\epsilon = \min(\delta,\epsilon')$ and $F = F_1 \cup F_2$.  We claim that
		\[
		K_{F',\epsilon'}^{\orb}(\Gamma_{r,\mathbf{r}}(\cO')) \leq k \, K_{F,\epsilon}^{\orb}(\Gamma_{\mathbf{r}}(\cO)).
		\]
		
		There exists an orbital $(F,\epsilon)$-cover $\Omega$ of $\Gamma_{\mathbf{r}}^{(n)}(\cO)$ with $\Omega \subseteq \Gamma_{\mathbf{r}}^{(n)}(\cO)$ and $|\Omega| \leq K_{F,\epsilon/2}(\Gamma_{\mathbf{r}}^{(n)}(\cO))$.  Indeed, we can let $\Omega_0$ be any orbital $(F,\epsilon/2)$-cover of $\Gamma_{\mathbf{r}}^{(n)}(\cO)$ not necessarily contained in $\Gamma_{\mathbf{r}}^{(n)}(\cO)$ and let $\Omega$ contain one point $\Gamma_{\mathbf{r}}^{(n)}(\cO) \cap N_{F,\epsilon/2}(\bY)$ for each $\bY \in \Omega'$ where the intersection is nonempty.
		
		For each $\bX' \in \Gamma_{\mathbf{r}}^{(n)}(\cO)$, we have
		\[
		\beta^{M_n(\C)}(\bX') < \frac{\epsilon'}{16},
		\]
		and therefore, there exist $Y_1' \in D_{r_1}^{M_n(\C)}$, \dots, $Y_k' \in D_{r_k}^{M_n(\C)}$ such that
		\[
		\alpha^{M_n(\C)}(Y_1',\dots,Y_k',\bX') < \frac{\epsilon'}{16}.
		\]
		Choose for each $\bX' \in \Omega$ a corresponding $Y_1'(\bX')$, \dots, $Y_k'(\bX')$, and let
		\[
		\Omega' = \{(Y_1'(\bX'),\bX'), \dots, (Y_k'(\bX'),\bX'): \bX' \in \Omega\}.
		\]
		We claim that $\Omega'$ is an orbital $(F',\epsilon')$-cover of $\Gamma_{r,\mathbf{r}}^{(n)}(\cO')$.  Let $(Y'',\bX'') \in \Gamma_{r,\mathbf{r}}^{(n)}(\cO')$.  Then $\bX'' \in \Gamma_{\mathbf{r}}(\cO)$.  Therefore, there exists a unitary $U$ and $\bX' \in \Omega$ such that $U \bX'' U^* \in N_{F,\epsilon}(\bX')$.  Let $Y_1' = Y_1'(\bX')$, \dots, $Y_k' = Y_k'(\bX')$, and note that because $d^{M_n(\C)}(X_j',X_j'') < \delta$ for $j \in F_2$, we have
		\[
		|\alpha^{M_n(\C)}(Y_1',\dots,Y_k',U\bX''U^*) - \alpha^{M_n(\C)}(Y_1',\dots,Y_k',\bX')| < \frac{\epsilon'}{16},
		\]
		hence
		\[
		\alpha^{M_n(\C)}(Y_1',\dots,Y_k',U\bX''U^*) < \frac{\epsilon'}{16} + \frac{\epsilon'}{16} = \frac{\epsilon'}{8}.
		\]
		By definition of $\alpha$, this means that
		\begin{multline*}
			\sup_{Y' \in D_r^{M_n}(\C)} \min\Bigl(\epsilon'/4 - \psi^{M_n(\C)}(Y',U\bX'' U^*), \\ \min(d^{M_n(\C)}(Y',Y_1'),\dots,d^{M_n(\C)}(Y',Y_k')) - \epsilon'/2 \Bigr) < \frac{\epsilon'}{8}.
		\end{multline*}
		Now because $(Y'',\bX'') \in \Gamma_{r,\mathbf{r}}^{(n)}(\cO')$, we have
		\[
		\psi^{M_n(\C)}(UY''U^*,U\bX''U^*) = \psi^{M_n(\C)}(Y'',\bX'') < \epsilon' / 8,
		\]
		it follows that
		\[
		\frac{\epsilon'}{4} - \psi^{M_n(\C)}(UY''U^*,U\bX'' U^*) > \frac{\epsilon'}{8}
		\]
		and therefore
		\[
		\min(d^{M_n(\C)}(UY''U^*,Y_1'),\dots,d^{M_n(\C)}(UY''U^*,Y_k')) - \frac{\epsilon'}{2} < \frac{\epsilon'}{8},
		\]
		hence $d(UY''U^*,Y_i) < \epsilon'/2 + \epsilon'/8 < \epsilon'$ for some $i \in \{1,\dots,k\}$.  Therefore, overall
		\[
		U \bX'' U^* \in N_{F,\epsilon}(\bX') \subseteq N_{F_1,\epsilon'}(\bX') \text{ and } UY''U^* \in N_{\epsilon'}(Y_i'),
		\]
		and thus $(Y'',\bX'') \in N_{F',\epsilon'}(Y_i',\bX')$, which shows that $\Gamma_{r,\mathbf{r}'}^{(n)}(\cO') \subseteq N_{F',\epsilon'}^{\orb}(\Omega')$.
		
		We conclude that
		\[
		K_{F',\epsilon'}(\Gamma_{r,\mathbf{r}}^{(n)}(\cO')) \leq k \,  K_{F,\epsilon/2}(\Gamma_{\mathbf{r}}^{(n)}(\cO)).
		\]
		Hence, applying $\lim_{n \to \cU} (1/n^2) \log$ to both sides,
		\[
		\Ent_{(r,\mathbf{r}'),F',\epsilon'}^{\cU}(\cO') \leq \Ent_{\mathbf{r},F,\epsilon}^{\cU}(\cO).
		\]
		Because for every $\cO$ there exists such an $\cO'$, we obtain
		\[
		\Ent_{(r,\mathbf{r}),F',\epsilon'}^{\cU}(\tp^{\cM}(Y,\bX)) \leq \Ent_{\mathbf{r},F,\epsilon/2}(\tp^{\cM}(\bX)).
		\]
		Then because for every $(F',\epsilon')$ there exists such an $(F,\epsilon)$, we get
		\[
		\Ent_{r,\mathbf{r}}^{\cU}(\tp^{\cM}(Y,\bX)) \leq \Ent_{\mathbf{r}}^{\cU}(\tp^{\cM}(\bX)).
		\]
		Taking the supremum over $\mathbf{r}$ and $r$ completes the proof.
	\end{proof}
	
	\begin{corollary}
		Let $\cM$ be a tracial $\mathrm{W}^*$-algebra and $\cN$ a tracial $\mathrm{W}^*$-subalgebra.  Then
		\[
		\Ent^{\cU}(\acl(\cN):\cM) = \Ent^{\cU}(\cN:\cM).
		\]
	\end{corollary}
	
	\begin{proof}
		The inequality $\Ent^{\cU}(\acl(\cN):\cM) \geq \Ent^{\cU}(\cN:\cM)$ holds by Observation \ref{obs:monotonicity2}.
		
		On the other hand, suppose that $\bY$ is an $\N$-tuple in $\acl(\cN)$.  Using Remark \ref{rem:definableoverN}, each $Y_k$ is algebraic over some separable $\mathrm{W}^*$-subalgebra of $\cN$.  Let $\cN_0 \subseteq \cN$ be the join of all these subalgebras, so that $\cN_0$ is separable and $\bY$ is algebraic over $\cN_0$.  Let $\bX \in L^\infty(\cN_0)^{\N}$ generate $\cN_0$.  Since $Y_1$ is algebraic over $\cN_0$, we have
		\[
		\Ent^{\cU}(\tp^{\cM}(\bX,Y_1)) = \Ent^{\cU}(\tp^{\cM}(\bX)).
		\]
		Similarly, since $Y_2$ is algebraic over $\mathrm{W}^*(\bX,Y_1)$, we have
		\[
		\Ent^{\cU}(\tp^{\cM}(\bX,Y_1,Y_2)) = \Ent^{\cU}(\tp^{\cM}(\bX)).
		\]
		Continuing inductively, for $k \in \N$,
		\[
		\Ent^{\cU}(\tp^{\cM}(\bX,Y_1,\dots,Y_k)) = \Ent^{\cU}(\tp^{\cM}(\bX)).
		\]
		Now to analyze $\Ent^{\cU}(\tp^{\cM}(\bX,\bY))$, suppose $\mathbf{r} \in (0,\infty)^{\cN\sqcup \N}$ and $\epsilon > 0$ and $F \subseteq \N \sqcup \N$ is finite.  Then $F \subseteq \N \sqcup \{1,\dots,k\}$ for some $k \in \N$.  For every neighborhood $\cO$ of $\tp^{\cM}(\bX,Y_1,\dots,Y_k)$, there is a corresponding neighborhood $\cO'$ of $\tp^{\cM}(\bX,\bY)$ given as the preimage of $\cO$ under map restricting the type of an $\N \sqcup \N$-tuple to the type of the $\N \sqcup \{1,\dots,k\}$-subtuple.  Since $F \subseteq \N \sqcup \{1,\dots,k\}$, then
		\[
		\Ent_{\mathbf{r},F,\epsilon}^{\cU}(\cO') = \Ent_{\mathbf{r}',F,\epsilon}^{\cU}(\cO)
		\]
		where $\mathbf{r}'$ is the restriction of $\mathbf{r}$ to $\N \sqcup \{1,\dots,k\}$.  This implies that
		\begin{multline*}
			\Ent_{\mathbf{r},F,\epsilon}^{\cU}(\tp^{\cM}(\bX,\bY)) \\
			\leq \Ent_{\mathbf{r}',F,\epsilon}^{\cU}(\tp^{\cM}(\bX,Y_1,\dots,Y_k)) \leq \Ent^{\cU}(\tp^{\cM}(\bX,Y_1,\dots,Y_k)) = \Ent^{\cU}(\tp^{\cM}(\bX)).
		\end{multline*}
		Since $\mathbf{r}$, $F$, and $\epsilon$ were arbitrary, $\Ent^{\cU}(\tp^{\cM}(\bX,\bY)) \leq \Ent^{\cU}(\tp^{\cM}(\bX))$.  Also, $\Ent^{\cU}(\tp^{\cM}(\bY)) \leq \Ent^{\cU}(\tp^{\cM}(\bX,\bY))$ by Corollary \ref{cor:invariance}.  Since $\bY$ was an arbitrary $\N$-tuple in $\acl(\cN)$, we obtain $\Ent^{\cU}(\acl(\cN):\cM) \leq \Ent^{\cU}(\cN:\cM)$ as desired.
	\end{proof}
	
	\section{Entropy for quantifier-free and existential types} \label{sec:qfentropy}
	
	In this section, we explain how Hayes' $1$-bounded entropy (or covering entropy for non-commutative laws) relates to the entropy for types in this paper.  Specifically, the $1$-bounded entropy for laws corresponds is the version for quantifier-free types and the $1$-bounded entropy of $\cN$ in the presence of a larger $\mathrm{W}^*$-algebra $\cM$ is the version for existential types.
	
	\subsection{Entropy for quantifier-free types}
	
	We begin with the quantifier-free version, essentially the same as orbital version of $h(\cM)$ in \cite[Appendix A]{Hayes2018}.
	
	\begin{definition}[Entropy for quantifier-free types]
		For $\cK \subseteq \mathbb{S}_{\qf}(\rT_{\tr})$ and $\mathbf{r} \in (0,\infty)^{\N}$, we define
		\[
		\Gamma_{\mathbf{r}}^{(n)}(\cK) = \left\{\bX \in \prod_{j \in \N} D_{r_j}^{M_n(\C)}: \tp_{\qf}^{M_n(\C)}(\bX) \in \cK \right\}.
		\]
		Then we define for $F \subseteq \N$ finite and $\epsilon > 0$,
		\[
		\Ent_{\qf,\mathbf{r},F,\epsilon}^{\cU}(\cK) = \inf_{\cO \supseteq \cK \text{ open in } \mathbb{S}_{\mathbf{r}}(\rT_{\tr})} \lim_{n \to \cU} \frac{1}{n^2} \log K_{F,\epsilon}^{\orb}(\Gamma_{\mathbf{r}}^{(n)}(\cO)),
		\]
		and we set
		\[
		\Ent_{\qf}^{\cU}(\cK) = \sup_{\mathbf{r}, F, \epsilon} \Ent_{\qf,\mathbf{r},F,\epsilon}^{\cU}(\cK).
		\]
		For $\mu \in \mathbb{S}_{\qf}(\rT_{\tr})$, let $\Ent_{\qf}^{\cU}(\mu) = \Ent_{\qf}^{\cU}(\{\mu\})$.
	\end{definition}
	
	Some earlier works such as \cite{Hayes2018} phrased the definition of $\Ent_{\qf}^{\cU}(\mu)$ in terms of particular open sets $\cO_k$ (for instance, those defined by looking at moments of order up to $k$ being within some distance $1/k$ of the moments of $\mu$).  This does not change the definition because of the following lemma.
	
	\begin{lemma} \label{lem:sequenceofneighborhoods}
		Let $\mathbf{r} \in (0,\infty)^{\N}$.  Let $\cK \subseteq \mathbb{S}_{\mathbf{r}}(\rT_{\tr})$.  Let $(\cO_k)_{k \in \N}$ be a sequence of open subsets of $\mathbb{S}(\rT_{\tr})$ such that $\overline{\cO}_{k+1} \subseteq \cO_k$ and $\bigcap_{k=1}^\infty \cO_k = \cK$.  Then
		\[
		\Ent_{\mathbf{r},F,\epsilon}^{\cU}(\cK) = \lim_{k \to \infty} \Ent_{\mathbf{r},F,\epsilon}^{\cU}(\cO_k) = \inf_{k \in \N} \Ent_{\mathbf{r},F,\epsilon}^{\cU}(\cO_k).
		\]
		The same holds with $\mathbb{S}(\rT_{\tr})$ and $\Ent^{\cU}$ replaced by their quantifier-free versions.
	\end{lemma}
	
	\begin{proof}
		By Observation \ref{obs:monotonicity},
		\[
		\Ent_{\mathbf{r},F,\epsilon}^{\cU}(\cK) \leq \Ent_{\mathbf{r},F,\epsilon}^{\cU}(\cO_{k+1}) \leq \Ent_{\mathbf{r},F,\epsilon}^{\cU}(\cO_k),
		\]
		so that
		\[
		\Ent_{\mathbf{r},F,\epsilon}^{\cU}(\cK) \leq \inf_{k \in \N} \Ent_{\mathbf{r},F,\epsilon}^{\cU}(\cO_k) = \lim_{k \to \infty} \Ent_{\mathbf{r},F,\epsilon}^{\cU}(\cO_k).
		\]
		For the inequality in the other direction, fix $\cO \supseteq \cK$ open.  Then $\mathbb{S}_{\mathbf{r}}(\rT_{\tr}) \setminus \cO$ is closed and disjoint from $k$.  Moreover, it is contained in $\cK^c = \bigcup_{k \in \N} \cO_k^c = \bigcup_{k \in \N} \overline{\cO}_k^c$.  By compactness, there is a finite subcollection of $\overline{\cO}_k^c$'s that covers $\mathbb{S}_{\mathbf{r}}(\rT_{\tr}) \setminus \cO$.  The $\cO_k$'s are nested, so there exists some $k$ such that $\mathbb{S}_{\mathbf{r}}(\rT_{\tr}) \setminus \cO \subseteq \overline{\cO_k}^c$, hence $\cO_k \cap \mathbb{S}_{\mathbf{r}}(\rT_{\tr}) \subseteq \cO$.  Therefore,
		\[
		\inf_{k \in \N} \Ent_{\mathbf{r},F,\epsilon}^{\cU}(\cO_k) \leq \Ent_{\mathbf{r},F,\epsilon}^{\cU}(\cO).
		\]
		Since $\cO$ was arbitrary,
		\[
		\inf_{k \in \N} \Ent_{\mathbf{r},F,\epsilon}^{\cU}(\cO_k) \leq \Ent_{\mathbf{r},F,\epsilon}^{\cU}(\cK).
		\]
		The argument for the quantifier-free case is identical.
	\end{proof}
	
	This lemma also allows us to relate the entropy $\Ent_{\qf}^{\cU}$ for quantifier-free types directly to the entropy for types $\Ent^{\cU}$.
	
	\begin{lemma} \label{lem:qfentropyrelationship}
		Let $\pi_{\qf}: \mathbb{S}(\rT_{\tr}) \to \mathbb{S}_{\qf}(\rT_{\tr})$ be the canonical restriction map.  Let $\cK \subseteq \mathbb{S}_{\qf}(\rT_{\tr})$ be closed.  Then
		\[
		\Ent_{\qf}^{\cU}(\cK) = \Ent^{\cU}(\pi_{\qf}^{-1}(\cK)).
		\]
	\end{lemma}
	
	\begin{proof}
		Fix $\mathbf{r} \in (0,\infty)^{\cN}$, and let $\cK_{\mathbf{r}} = \cK \cap \mathbb{S}_{\qf,\mathbf{r}}(\rT_{\tr})$.  Since $\mathbb{S}_{\qf,\mathbf{r}}(\rT_{\tr})$ is metrizable, there exists a sequence of open sets $\cO_k$ in $\mathbb{S}_{\qf,\mathbf{r}}(\rT_{\tr})$ such that $\overline{\cO}_{k+1} \subseteq \cO_k$ and $\bigcap_{k \in \N} \cO_k = \cK_{\mathbf{r}}$ (and these can be extended to open sets in $\mathbb{S}_{\qf}(\rT_{\tr})$ since the inclusion of $\mathbb{S}_{\qf,\mathbf{r}}(\rT_{\tr})$ is a topological embedding).  Now $\pi_{\qf}^{-1}(\cO_k)$ is open in $\mathbb{S}_{\mathbf{r}}(\rT_{\tr})$ and $\overline{\pi_{\qf}^{-1}(\cO_{k+1})} \subseteq \pi_{\qf}^{-1}(\overline{\cO}_{k+1}) \subseteq \pi_{\qf}^{-1}(\cO_k)$ and $\bigcap_{k \in \N} \pi_{\qf}^{-1}(\cO_k) = \pi_{\qf}^{-1}(\cK_{\mathbf{r}})$.  Note that $\Gamma_{\mathbf{r}}^{(n)}(\cO_k) = \Gamma_{\mathbf{r}}^{(n)}(\pi_{\qf}^{-1}(\cO_k))$.  Thus, using the previous lemma,
		\begin{align*}
			\Ent_{\qf,\mathbf{r},F,\epsilon}^{\cU}(\cK) &= \Ent_{\qf,\mathbf{r},F,\epsilon}^{\cU}(\cK_{\mathbf{r}}) \\
			&= \inf_{k \in \N} \Ent_{\qf,\mathbf{r},F,\epsilon}(\cO_k) \\
			&= \inf_{k \in \N} \Ent_{\mathbf{r},F,\epsilon}(\pi_{\qf}^{-1}(\cO_k)) \\
			&= \Ent_{\mathbf{r},F,\epsilon}^{\cU}(\pi_{\qf}^{-1}(\cK_{\mathbf{r}})) \\
			&= \Ent_{\mathbf{r},F,\epsilon}^{\cU}(\pi_{\qf}^{-1}(\cK)).
		\end{align*}
		Taking the supremum over $\mathbf{r}$, $F$, and $\epsilon$ completes the argument.
	\end{proof}
	
	In particular, by combining this with the variational principle (Proposition \ref{prop:variational}), we obtain the following corollary.
	
	\begin{corollary} \label{cor:qfvariational}
		Let $\pi_{\qf}: \mathbb{S}(\rT_{\tr}) \to \mathbb{S}_{\qf}(\rT_{\tr})$ be the restriction map.  If $\mu \in \mathbb{S}_{\qf}(\rT_{\tr})$, then
		\[
		\Ent_{\qf}^{\cU}(\mu) = \sup_{\nu \in \pi_{\qf}^{-1}(\mu)} \Ent^{\cU}(\nu).
		\]
	\end{corollary}
	
	This also implies that the quantifier-free entropy of $\tp_{\qf}^{\cM}(\bX)$ only depends on $\mathrm{W}^*(\bX)$, which is an important property of $1$-bounded entropy previously established by Hayes in \cite[Theorem A.9]{Hayes2018}.
	
	\begin{corollary}
		Let $\cM = (M,\tau)$ be a tracial $\mathrm{W}^*$-algebra.  Let $\bX$, $\bY \in M^{\N}$.  If $\mathrm{W}^*(\bX) = \mathrm{W}^*(\bY)$, then
		\[
		\Ent_{\qf}^{\cU}(\tp_{\qf}^{\cM}(\bX)) = \Ent_{\qf}^{\cU}(\tp_{\qf}^{\cM}(\bY)). 
		\]
	\end{corollary}
	
	\begin{proof}
		By Proposition \ref{prop:deffuncrealization}, there exist quantifier-free definable functions $\mathbf{f}$ and $\mathbf{g}$ such that $\mathbf{f}^{\cM}(\bX) = \bY$ and $\mathbf{g}^{\cM}(\bY) = \bX$.  If $\mu \in \pi_{\qf}^{-1}(\tp_{\qf}^{\cM}(\bX))$, then $\mathbf{f}_* \mu \in \pi_{\qf}^{-1}(\tp_{\qf}^{\cM}(\bY))$ since $\pi_{\qf} \circ \mathbf{f}_* = \mathbf{f}_* \circ \pi_{\qf}$.  Similarly, $\mathbf{g}_*$ maps $\pi_{\qf}^{-1}(\tp_{\qf}^{\cM}(\bY))$ into $\pi_{\qf}^{-1}(\tp_{\qf}^{\cM}(\bX))$.
		
		Since $d(\mathbf{f} \circ \mathbf{g}(\mathbf{x}),\mathbf{x})$ is a quantifier-free definable predicate, its value only depends on the quantifier-free type of the input, and thus $\mathbf{g}^{\cN} \circ \mathbf{f}^{\cN}(\bZ) = \bZ$ whenever $\cN \models \rT_{\tr}$ and $\tp_{\qf}^{\cN}(\bZ) = \tp_{\qf}^{\cM}(\bX)$.  In particular, if $\mu \in \pi_{\qf}^{-1}(\tp_{\qf}^{\cM}(\bX))$, then $\mathbf{g}_* \mathbf{f}_* \mu = \mu$.  The same holds for $\mathbf{g} \circ \mathbf{f}$.  Hence, $\mathbf{f}$ and $\mathbf{g}$ define mutually inverse maps between $\pi_{\qf}^{-1}(\tp_{\qf}^{\cM}(\bX))$ and $\pi_{\qf}^{-1}(\tp_{\qf}^{\cM}(\bY))$. Note also that $\pi_{\qf}^{-1}(\tp_{\qf}^{\cM}(\bX))$ is contained in $\mathbb{S}_{\mathbf{r}}(\rT_{\tr})$ for some $\mathbf{r}$ by Remark \ref{rem:qfnorm}. Therefore, by Proposition \ref{prop:pushforward},
		\[
		\Ent^{\cU}(\pi_{\qf}^{-1}(\tp_{\qf}^{\cU}(\bX))) = \Ent^{\cU}(\pi_{\qf}^{-1}(\tp_{\qf}^{\cU}(\bY))),
		\]
		which implies the claimed result by Lemma \ref{lem:qfentropyrelationship}.
	\end{proof}
	
	Furthermore, since the quantifier-free type does not depend on the ambient tracial $\mathrm{W}^*$-algebra $\cM$, it follows that if $\bX$ and $\bY$ in \emph{different} tracial $\mathrm{W}^*$-algebras generate isomorphic tracial $\mathrm{W}^*$-algebras, then their quantifier-free types have the same entropy.  Hence, it is consistent to define for a separable tracial $\mathrm{W}^*$-algebra $\cM$,
	\[
	\Ent_{\qf}^{\cU}(\cM) = \Ent_{\qf}^{\cU}(\tp_{\qf}^{\cM}(\bX)),
	\]
	where $\bX$ is an $\N$-tuple of generators for $\cM$ (for the definition of $\Ent_{\qf}^{\cU}(\cM)$ in the case of non-separable $\cM$, see Remark \ref{rem:nonseparableqfentropy} below).  However, Remark \ref{rem:qfnomonotonicity} shows that there is no quantifier-free analog of monotonicity under pushforward (Proposition \ref{prop:pushforward}).
	
	\subsection{Existential types}
	
	Now we turn our attention to existential types.
	
	\begin{definition}
		An \emph{existential formula} in a language $\cL$ is a formula of the form
		\[
		\phi(\mathbf{x}) = \inf_{y_1 \in D_1, \dots, y_k \in D_k} \psi(\mathbf{x},y_1,\dots,y_k),
		\]
		where $\psi$ is a quantifier-free formula and $D_1$, \dots, $D_k$ are domains of quantification in the appropriate sorts.  Similarly, we say that $\phi$ is an \emph{existential definable predicate relative to $\rT$} if
		\[
		\phi^{\cM}(\bX) = \inf_{\bY \in \prod_{j \in \N} D_j^{\cM}} \psi^{\cM}(\bX,\bY)
		\]
		for $\cM \models \rT$, where $\psi$ is a quantifier-free definable predicate.
	\end{definition}
	
	\begin{observation}
		Any existential definable predicate can be approximated uniformly on each product of domains of quantification by an existential formula.
	\end{observation}
	
	\begin{definition}
		Let $\cM$ be an $\cL$-structure, $\bS$ an $\N$-tuple of sorts, and $\bX \in \prod_{j \in \N} S_j^{\cM}$.  Let $\cF_{\exists,\bS}$ denote the space of existential formulas.  The \emph{existential type} $\tp_{\exists}^{\cM}(\bX)$ is the map
		\[
		\tp_{\exists}^{\cM}(\bX): \cF_{\exists,\bS} \to \R, \phi \mapsto \phi^{\cM}(\bX).
		\]
		If $\rT$ is an $\cL$-theory, we denote the set of existential types that arise in models of $\rT$ by $\mathbb{S}_{\exists,\bS}(\rT)$.
	\end{definition}
	
	The topology for existential types, however, is not simply the weak-$\star$ topology on $\mathbb{S}_{\exists,\bD}(\rT)$ for each tuple of domains.  Rather, we define neighborhoods of a type $\mu = \tp^{\cM}(\bX)$ using sets of the form $\{\nu: \nu(\phi) < \mu(\phi) + \epsilon\}$.  The idea is that if $\phi^{\cM}(\bX) = \inf_{\bY \in \prod_{j \in \N} D_j^{\cM}} \psi^{\cM}(\bX,\bY)$ for some quantifier-free definable predicate $\phi$, then $\mu(\phi) \leq c$ means that there exists $\bY$ such that $\psi^{\cM}(\bX,\bY) < c + \delta$ for any $\delta > 0$.  Thus, a neighborhood corresponds to types $\nu$ where there exists $\bY$ that gets within $\epsilon$ of the infimum achieved by $\mu$.
	
	\begin{definition} \label{def:existentialtopology}
		Let $\rT$ be an $\cL$-theory, $\bS$ an $\N$-tuple of sorts, and $\bD \in \prod_{j \in \N} \cD_{S_j}$.  We say that $\cO \subseteq \mathbb{S}_{\exists,\bD}(\rT)$ is open if for every $\mu \in \cO$, there exist existential formulas $\phi_1$, \dots, $\phi_k$ and $\epsilon_1$, \dots, $\epsilon_k > 0$ such that
		\[
		\{\nu \in \mathbb{S}_{\exists,\bD}(\rT): \nu(\phi_j) < \mu(\phi_j) + \epsilon_j \text{ for } j = 1, \dots, k\} \subseteq \cO.
		\]
		Moreover, we say that $\cO \subseteq \mathbb{S}_{\exists,\bS}(\rT)$ is open if $\cO \cap \mathbb{S}_{\exists,\bD}(\rT)$ is open in $\mathbb{S}_{\exists,\bD}(\rT)$ for all $\bD \in \prod_{j \in \N} \cD_{S_j}$.
	\end{definition}
	
	\begin{observation}~
		\begin{itemize}
			\item Any set of the form $\{\nu: \nu(\phi_1) < c_1, \dots, \nu(\phi_k) < c_k\}$, where $\phi_1$, \dots, $\phi_k$ are existential definable predicates, is open in $\mathbb{S}_{\exists,\bS}(\rT)$.
			\item The same holds if $\phi_j$ is an existential definable predicate rather than existential formula, since it can be uniformly approximated by existential formulas on each product of domains of quantification, hence existential definable predicates may be used in Definition \ref{def:existentialtopology} without changing the definition.
			\item The inclusion $\mathbb{S}_{\exists,\bD}(\rT) \to \mathbb{S}_{\exists,\bS}(\rT)$ is a topological embedding since each of the basic open sets in $\mathbb{S}_{\exists,\bD}(\rT)$ given by $\nu(\phi_j) < \mu(\phi_j) + \epsilon_j$ for $j = 1$, \dots, $k$ extends to an open set in $\mathbb{S}_{\exists,\bS}(\rT)$.
			\item The restriction map $\mathbb{S}_{\bS}(\rT) \to \mathbb{S}_{\exists,\bS}(\rT)$ is continuous.
		\end{itemize}
	\end{observation}
	
	\begin{remark}
		Like the Zariski topology on the space of ideals in a commutative ring, the topology on $\mathbb{S}_{\exists,\bS}(\rT)$ is often non-Hausdorff.  For instance, the closure of a point is given by
		\[
		\overline{\{\mu\}} = \{\nu: \nu(\phi) \geq \mu(\phi) \text{ for all } \phi \in \cF_{\exists,\bS}\}.
		\]
		Meanwhile, the intersection of all neighborhoods of $\mu$ is given by
		\begin{equation} \label{eq:nbhd}
			\mathcal{K}_\mu = \{\nu: \nu(\phi) \leq \mu(\phi) \text{ for all } \phi \in \cF_{\exists,\bS}\}.
		\end{equation}
		We say that $\nu$ \emph{extends} $\mu$ if $\nu(\phi) \leq \mu(\phi)$ for all existential formulas $\phi$, which is equivalent to saying that for $\phi \in \mathcal{F}_{\exists,\bS}$, we have $\mu(\phi) = 0$ implies that $\nu(\phi) = 0$ (since $\max(\phi - c,0)$ is an existential formula if $\phi$ is).  Then $\{\mu\} = \cK_\mu$ if and only if it does have any proper extension, or it is \emph{maximal}.  These closed points correspond to existential types from existentially closed models (see \cite[\S 6.2]{Goldbring2021enforceable}), and such maximal existential types in $\mathbb{S}_{\exists,\bD}(\rT)$ form a compact Hausdorff space.  However, our present goal is to work with general tracial $\mathrm{W}^*$-algebras, not only those that are existentially closed.
	\end{remark}
	
	\subsection{Entropy for existential types}
	
	Here we define the entropy for existential types which corresponds to Hayes' entropy of $\cN$ in the presence of $\cM$.  We explain our definition in this subsection, and in the next one we relate it with Hayes' definition.
	
	\begin{definition}
		For $\cK \subseteq \mathbb{S}_{\exists}(\rT_{\tr})$, let
		\[
		\Gamma_{\mathbf{r}}^{(n)}(\cK) = \{ \bX \in \prod_{j \in \N} D_{r_j}^{M_n(\C)}: \tp_{\exists}^{M_n(\C)}(\bX) \in \cK\},
		\]
		and define for $\mathbf{r} \in (0,\infty)^{\N}$, $F \subseteq \N$, finite, and $\epsilon > 0$,
		\[
		\Ent_{\exists,\mathbf{r},F,\epsilon}^{\cU}(\cK) = \inf_{\cO \supseteq \cK \text{ open}} \lim_{n \to \cU} \frac{1}{n^2} \log K_{F,\epsilon}^{\orb}(\Gamma_{\mathbf{r}}^{(n)}(\cO)).
		\]
		Then let
		\[
		\Ent_{\exists}^{\cU}(\cK) = \sup_{\mathbf{r},F,\epsilon} \Ent_{\exists,\mathbf{r},F,\epsilon}^{\cU}(\cK).
		\]
	\end{definition}
	
	Because of the non-Hausdorff nature of $\mathbb{S}_{\exists}(\rT_{\tr})$, we will be content to focus on the existential entropy for an individual existential type rather than for a closed set of existential types.
	
	\begin{lemma} \label{lem:existentialversusfull}
		Let $\mu \in \mathbb{S}_{\exists}(\rT_{\tr})$, and let $\cK_\mu$ be given by \eqref{eq:nbhd}.  Let $\pi: \mathbb{S}(\rT_{\tr}) \to \mathbb{S}_{\exists}(\rT_{\tr})$ be the canonical restriction map.  Then
		\[
		\Ent_{\exists}^{\cU}(\mu) = \Ent^{\cU}\left(\pi^{-1}\left( \cK_\mu \right) \right) =  \sup_{\nu \in \pi^{-1}(\cK_\mu)} \Ent^{\cU}(\nu).
		\]
	\end{lemma}
	
	\begin{proof}
		Fix $\mathbf{r} \in (0,\infty)^{\N}$, $F \subseteq \N$ finite, and $\epsilon > 0$.  If $\cO$ is a neighborhood of $\mu$ in $\mathbb{S}_{\exists}(\rT_{\tr})$, then it contains $\cK_\mu$, and hence $\pi^{-1}(\cO)$ is a neighborhood of $\pi^{-1}(\cK_\mu)$ in $\mathbb{S}(\rT_{\tr})$.  Moreover, $\Gamma_{\mathbf{r}}^{(n)}(\cO) = \Gamma_{\mathbf{r}}^{(n)}(\pi^{-1}(\cO))$, hence
		\[
		\Ent_{\mathbf{r},F,\epsilon}^{\cU}\left(\pi^{-1}\left(\cK_\mu \right) \right) \leq \Ent_{\exists,\mathbf{r},F,\epsilon}^{\cU}(\mu).
		\]
		
		It remains to show the reverse inequality. Since the space of definable predicates on $\prod_{j \in \N} D_{r_j}$ relative to $\rT_{\tr}$ is separable with respect to the uniform metric, so is the space of existential definable predicates.  Let $(\phi_j)_{j \in \N}$ be a sequence of existential definable predicates that are dense in this space.  Let
		\[
		\cO_k = \left\{\nu \in \mathbb{S}_{\exists,\mathbf{r}}(\rT_{\tr}): \nu(\phi_j) < \mu(\phi_j) + \frac{1}{k} \text{ for } j \leq k \right\}.
		\]
		Note that
		\[
		\bigcap_{k \in \N} \cO_k = \{\nu \in \mathbb{S}_{\exists,\mathbf{r}}(\rT_{\tr}): \nu(\phi_k) \leq \mu(\phi_k) \text{ or } k \in \N\} = \cK_\mu.
		\]
		Moreover,
		\[
		\overline{\pi^{-1}(\cO_{k+1})} \subseteq \left\{\nu \in \mathbb{S}(\rT_{\tr}): \nu(\phi_j) \leq \mu(\phi_j) + \frac{1}{k+1} \text{ for } j \leq k + 1 \right\} \subseteq \pi^{-1}(\cO_k).
		\]
		Therefore, by Lemma \ref{lem:sequenceofneighborhoods} applied to $\pi^{-1}(\cO_k)$, we have
		\[
		\Ent_{\exists,\mathbf{r},F,\epsilon}^{\cU}(\mu) \leq \inf_{k \in \N} \Ent_{\exists,\mathbf{r},F,\epsilon}^{\cU}(\cO_k) = \inf_{k \in \N} \Ent_{\mathbf{r},F,\epsilon}^{\cU}(\pi^{-1}(\cO_k)) = \Ent_{\mathbf{r},F,\epsilon}^{\cU}(\pi^{-1}(\cK_\mu)),
		\]
		where the last equality follows from the density of $\{\phi_k: k \in \N\}$.  Thus, $\Ent_{\exists,\mathbf{r},F,\epsilon}^{\cU}(\mu) = \Ent_{\mathbf{r},F,\epsilon}^{\cU}(\pi^{-1}(\cK_\mu))$.  Taking the supremum over $\mathbf{r}$, $F$, and $\epsilon$ yields the first asserted equality $\Ent_{\exists}^{\cU}(\mu) = \Ent^{\cU}(\pi^{-1}(\cK_\mu))$.  The second equality follows from the applying the variational principle (Proposition \ref{prop:variational}) to the closed set $\pi^{-1}(\cK_\mu)$.
	\end{proof}
	
	Like the entropy for full types, the entropy for existential types satisfies a certain monotonicity under pushforwards.  First, to clarify the meaning of pushforward, note that if $\mathbf{f}$ is a quantifier-free definable function and $\phi$ is an existential definable predicate, say
	\[
	\phi^{\cM}(\bX) = \inf_{\mathbf{Y} \in \prod_{j \in \N} D_{r_j'}^{\cM}} \psi^{\cM}(\bX,\bY) \text{ for } \cM \models \rT_{\tr},
	\]
	where $\psi$ is a quantifier-free definable predicate, then
	\[
	(\phi \circ \mathbf{f})^{\cM}(\bX) = \inf_{\mathbf{Y} \in \prod_{j \in \N} D_{r_j'}^{\cM}} \psi^{\cM}(\mathbf{f}(\bX),\bY)
	\]
	is also an existential definable predicate.  Hence, there is a well-defined pushforward map $\mathbf{f}_*: \mathbb{S}_{\exists}(\rT_{\tr}) \to \mathbb{S}_{\exists}(\rT_{\tr})$ given by $\mathbf{f}_* \mu(\phi) = \mu(\phi \circ \mathbf{f})$.  Furthermore, $\mathbf{f}_*$ is continuous with respect to the topology on $\mathbb{S}_{\exists}(\rT_{\tr})$ for the same reason that $\phi \circ \mathbf{f}$ is an existential definable predicate whenever $\phi$ is an existential definable predicate and $\mathbf{f}$ is a quantifier-free definable function.
	
	The following lemma can be proved directly in a similar way to Proposition \ref{prop:pushforward}, as was essentially done by Hayes in \cite{Hayes2018}; compare also the proof of Proposition \ref{prop:presenceequivalence} below.  However, as one of our main goals is to illuminate the model-theoretic nature of the existential entropy, we will give an argument to deduce this from Proposition \ref{prop:pushforward}.
	
	\begin{lemma} \label{lem:existentialpushforward}
		Let $\mu \in \mathbb{S}_{\exists}(\rT_{\tr})$ and let $\mathbf{f}$ be a quantifier-free definable function relative to $\rT_{\tr}$.  Then
		\[
		\Ent_{\exists}^{\cU}(\mathbf{f}_* \mu) \leq \Ent_{\exists}^{\cU}(\mu).
		\]
	\end{lemma}
	
	\begin{proof}
		Let $\pi: \mathbb{S}(\rT_{\tr}) \to \mathbb{S}_{\exists}(\rT_{\tr})$ be the restriction map.  Let
		\[
		\cK = \cK_\mu = \{\nu \in \mathbb{S}_{\exists}(\rT_{\tr}): \nu(\phi) \leq \mu(\phi) \text{ for existential } \phi\},
		\]
		and similarly, let $\cK' = \cK_{\mathbf{f}_*\mu}$.  By Lemma \ref{lem:existentialversusfull},
		\begin{align*}
			\Ent_{\exists}^{\cU}(\mu) &= \Ent_{\exists}^{\cU}(\cK) = \Ent^{\cU}(\pi^{-1}(\cK)) \\
			\Ent_{\exists}^{\cU}(\mathbf{f}_*\mu) &= \Ent_{\exists}^{\cU}(\cK') = \Ent^{\cU}(\pi^{-1}(\cK')).
		\end{align*}
		Meanwhile, by Proposition \ref{prop:pushforward}, Corollary \ref{cor:independentofr}, and Remark \ref{rem:qfnorm},
		\[
		\Ent^{\cU}(\mathbf{f}_* \pi^{-1}(\cK)) \leq \Ent^{\cU}(\pi^{-1}(\cK)).
		\]
		Therefore, it suffices to show that $\pi^{-1}(\cK') = \mathbf{f}_*(\pi^{-1}(\cK))$.
		
		By continuity of the pushforward on the space of existential types, it follows that $\mathbf{f}_*(\cK) \subseteq \cK'$, and hence
		\[
		\mathbf{f}_*(\pi^{-1}(\cK)) \subseteq \pi^{-1}(\mathbf{f}_*(\cK)) \subseteq \pi^{-1}(\cK').
		\]
		
		To prove the reverse inclusion, fix $\nu \in \pi^{-1}(\cK')$.  Fix $\mathbf{r}$ such that $\cK \subseteq \mathbb{S}_{\exists,\mathbf{r}}(\rT_{\tr})$ and $\mathbf{r}'$ such that $\mathbf{f}$ maps $\prod_{j \in \N} D_{r_j}^{\cM}$ into $\prod_{j \in \N} D_{r_j'}^{\cM}$ for $\cM \models \rT_{\tr}$.  For $F \subseteq \N$ finite and $\phi_1$, \dots, $\phi_k$ existential definable predicates, consider the definable predicate
		\[
		\psi^{\cM}(\bY) = \inf_{\bX \in \prod_{j \in \N} D_{r_j'}^{\cM}} \left[ \sum_{j \in F} d^{\cM}(f_j^{\cM}(\bX),Y_j) + \sum_{j=1}^k \max(0,\phi_j^{\cM}(\bX)-\mu(\phi_j)) \right].
		\]
		Then $\psi$ is an existential definable predicate:  Indeed, if $\phi_i^{\cM}(\bX) = \inf_{\bZ \in \prod_{j \in D_{r_{i,j}}}} \eta_i^{\cM}(\bX,\bZ)$, where $\eta$ is quantifier-free then
		\[
		\psi^{\cM}(\bY) = \inf_{\bX \in \prod_{j \in \N} D_{r_j'}^{\cM}} \inf_{\substack{\bZ_i \in \prod_{j \in \N} D_{r_{i,j}}^{\cM} \\ \text{for } i = 1, \dots, k}} \left[ \sum_{j \in F} d^{\cM}(f_j^{\cM}(\bX),Y_j) + \sum_{i=1}^k \max(0,\eta_i^{\cM}(\bX,\bZ_i)-\mu(\phi_i)) \right].
		\]
		Since $\nu \in \pi^{-1}(\cK')$, it follows that
		\[
		\nu(\psi) \leq \mathbf{f}_*\mu(\psi) = 0;
		\]
		this last equality holds because if $\tp_{\exists}^{\cN}(\bX') = \mu$ and $\bY' = \mathbf{f}_*(\bX')$, then $\psi^{\cN}(\bY') = 0$ since $\bX'$ participates in the infimum defining $\psi^{\cN}(\bY')$.
		
		Unwinding the definition of $\nu(\psi) \leq 0$, we have shown that for every $\epsilon > 0$ and $F \subseteq \N$ finite and $\phi_1$, \dots, $\phi_k$ existential definable predicates, there exist $\cM \models \rT_{\tr}$ and $\bY \in L^\infty(\cM)^{\N}$ and $\bX \in L^\infty(\cM)^{\N}$ with $\tp^{\cM}(\bY) = \nu$ and
		\[
		\sum_{j \in F} d^{\cM}(f_j^{\cM}(\bX),Y_j) + \sum_{j=1}^k \max(0,\phi_j^{\cM}(\bX) - \mu(\phi_j)) < \epsilon.
		\]
		Using an ultraproduct argument (or equivalently using the compactness theorem in continuous model theory, \cite[Theorem 5.8]{BYBHU2008}, \cite[Corollary 2.16]{BYU2010}), there exists some $\cM$ and $\bX$ and $\bY$ such that $\tp^{\cM}(\bY) = \nu$ and
		\[
		d^{\cM}(Y_j, f_j^{\cM}(\bX)) = 0 \text{ and } \phi^{\cM}(\bX) \leq \mu(\phi) \text{ for all existential definable predicates } \phi.
		\]
		This implies that $\tp_{\exists}^{\cM}(\bX) \in \cK_\mu = \cK$, hence $\tp^{\cM}(\bX) \in \pi^{-1}(\cK)$.  Therefore, $\nu = \tp^{\cM}(\bY) = \mathbf{f}_* \tp^{\cM}(\bX) \in \mathbf{f}_*(\pi^{-1}(\cK))$ as desired.
	\end{proof}
	
	The next corollary follows from Lemma \ref{lem:existentialpushforward} and Proposition \ref{prop:deffuncrealization}.
	
	\begin{corollary} \label{cor:existentialmonotonicity}
		If $\cM \models \rT_{\tr}$ and $\bX$, $\bY \in L^\infty(\cM)^{\N}$ and $\mathrm{W}^*(\bY) \subseteq \mathrm{W}^*(\bX)$, then
		\[
		\Ent_{\exists}^{\cU}(\tp^{\cM}(\bY)) \leq \Ent_{\exists}^{\cU}(\tp^{\cM}(\bX)).
		\]
		In particular, if $\mathrm{W}^*(\bX) = \mathrm{W}^*(\bY)$, then $\Ent_{\exists}^{\cU}(\tp^{\cM}(\bY)) = \Ent_{\exists}^{\cU}(\tp^{\cM}(\bX))$.
	\end{corollary}
	
	\begin{remark} \label{rem:qfnomonotonicity}
		The monotonicity property fails for the quantifier-free entropy.  For instance, let $\cM$ be the von Neumann algebra of the free group $\mathbb{F}_2$ and $\cR$ the hyperfinite $\mathrm{II}_1$ factor.  Then $\Ent_{\qf}^{\cU}(\cM) = \infty$ but $\Ent_{\qf}^{\cU}(\cM \overline{\otimes} \cR) = 0$ (by the same reasoning as in Corollary \ref{cor:Jungproperty}).  The proof of Lemma \ref{lem:existentialmonotonocity} breaks down because if $\pi: \mathbb{S}(\rT_{\tr}) \to \mathbb{S}_{\qf}(\rT_{\tr})$ is the restriction map, then $\pi^{-1}(\mathbf{f}_* \mu) \neq \mathbf{f}_*(\pi^{-1}(\mu))$ in general.  Given a $\bY$ with $\tp_{\qf}^{\cM}(\bY) = \mathbf{f}_* \mu$, in order to show the existence of some $\bX$ with $\mathbf{f}^{\cM}(\bX) \approx \bY$, we would have to use an existential formula in $\bY$.
	\end{remark}
	
	Now come to the definition of existential entropy for $\cN \subseteq \cM$, which we will show in \S \ref{subsec:inthepresence} is equivalent to Hayes' $h(\cN:\cM)$.
	
	\begin{definition} \label{def:inthepresence}
		Let $\cM$ be a tracial $\mathrm{W}^*$-algebra and $\cN \subseteq \cM$ a $\mathrm{W}^*$-subalgebra.  Then define
		\[
		\Ent_{\exists}^{\cU}(\cN: \cM) := \sup_{\bX \in L^\infty(\cN)^{\N}} \Ent_{\exists}^{\cU}(\tp_{\exists}^{\cM}(\bX)).
		\]
	\end{definition}
	
	The following is immediate from Corollary \ref{cor:existentialmonotonicity}.
	
	\begin{corollary} \label{cor:existentialgenerators}
		Let $\cM$ be a tracial $\mathrm{W}^*$-algebra and $\cN \subseteq \cM$ a $\mathrm{W}^*$-subalgebra.  If $\bX \in L^\infty(\cN)^{\N}$ generates $\cN$, then $\Ent_{\exists}^{\cU}(\cN:\cM) = \Ent_{\exists}^{\cU}(\tp_{\exists}^{\cM}(\bX))$.
	\end{corollary}

	\begin{lemma} \label{lem:existentialmonotonocity}
		Let $\cM_1 \subseteq \cM_2 \subseteq \cM_3$ be a tracial $\mathrm{W}^*$-algebras.  Then
		\[
		\Ent_{\exists}^{\cU}(\cM_1:\cM_3) \leq \Ent_{\exists}^{\cU}(\cM_2:\cM_3) \text{ and } \Ent_{\exists}^{\cU}(\cM_1:\cM_3) \leq \Ent_{\exists}^{\cU}(\cM_1:\cM_2).
		\]
	\end{lemma}
	
	\begin{proof}
		The first inequality is immediate from Definition \ref{def:inthepresence}.  For the second inequality, note that for every $\bX \in L^\infty(\cM_1)^{\N}$, for every existential formula $\phi$, we have $\phi^{\cM_3}(\bX) \leq \phi^{\cM_2}(\bX)$ since the first is the infimum over a larger set than the second.  In other words, $\tp_{\exists}^{\cM_3}(\bX)$ is an extension of $\tp_{\exists}^{\cM_2}(\bX)$, and hence every neighborhood of $\tp_{\exists}^{\cM_2}(\bX)$ is also a neighborhood of $\tp_{\exists}^{\cM_3}(\bX)$.  This implies that $\Ent_{\exists}^{\cU}(\tp_{\exists}^{\cM_3}(\bX)) \leq \Ent_{\exists}^{\cU}(\tp_{\exists}^{\cM_2}(\bX))$.  Since this holds for all $\bX \in L^\infty(\cM)^{\N}$, we obtain $\Ent_{\exists}^{\cU}(\cM_1:\cM_3) \leq \Ent_{\exists}^{\cU}(\cM_1:\cM_2)$.
	\end{proof}
	
	Next, we show that the quantifier-free entropy can be expressed in terms of the existential entropy.
	
	\begin{lemma} \label{lem:qfversusexistential}
		Let $\cM$ be a separable tracial $\mathrm{W}^*$-algebra.  Then
		\[
		\Ent_{\qf}^{\cU}(\cM) = \Ent_{\exists}^{\cU}(\cM:\cM).
		\]
	\end{lemma}
	
	\begin{proof}
		Suppose $\bX \in L^\infty(\cM)^{\N}$ generates $\cM$.  Fix $\mathbf{r}$ such that $\bX \in \prod_{j \in \N} D_{r_j}^{\cM}$.  Let $\pi: \mathbb{S}_{\exists}(\rT_{\tr}) \to \mathbb{S}_{\qf}(\rT_{\tr})$ be the restriction map.  It suffices to show that $\Ent_{\exists,\mathbf{r},F,\epsilon}(\tp_{\exists}^{\cM}(\bX)) = \Ent_{\qf,\mathbf{r},F,\epsilon}^{\cU}(\tp_{\qf}^{\cM}(\bX))$ for all $\mathbf{r}$, $F$, and $\epsilon$, which in turn will follow if we prove that every neighborhood $\cO'$ of $\tp_{\exists}^{\cM}(\bX)$ in $\mathbb{S}_{\exists,\mathbf{r}}(\rT_{\tr})$ contains $\pi^{-1}(\cO)$ for some neighborhood $\cO$ of $\tp_{\qf}^{\cM}(\bX)$ in $\mathbb{S}_{\qf,\mathbf{r}}(\rT_{\tr})$ and vice versa.
		
		Let $\cO$ be a neighborhood of $\tp_{\qf}^{\cM}(\bX)$.  By the definition / properties of the weak-$\star$ topology, there exist some quantifier-free definable predicates $\phi_1$, \dots, $\phi_k$ and intervals $(a_k, b_k)$ such that
		\[
		\tp_{\qf}^{\cM}(\bX) \in \{\nu \in \mathbb{S}_{\qf,\mathbf{r}}(\rT_{\tr}): \nu(\phi_i) \in (a_i,b_i) \text{ for } i = 1, \dots, k\} \subseteq \cO.
		\]
		Then $\phi_i$ and $-\phi_i$ are both existential definable predicates, hence
		\[
		\pi^{-1}(\cO) \supseteq \cO' := \{\nu \in \mathbb{S}_{\exists,\mathbf{r}}(\rT_{\tr}): \nu(\phi_i) < b_i, \nu(-\phi_k) < -a_i \text{ for } i = 1, \dots, k\}.
		\]
		
		Conversely, let $\cO'$ be a neighborhood of $\tp_{\exists}^{\cM}(\bX)$ in $\mathbb{S}_{\exists,\mathbf{r}}(\bX)$.  Then there exists existential definable predicates $\phi_1$, \dots, $\phi_k$ and $c_1$, \dots, $c_k \in \R$ such that
		\[
		\tp_{\exists}^{\cM}(\bX) \in \{\nu \in \mathbb{S}_{\exists,\mathbf{r}}(\rT_{\tr}): \nu(\phi_i) < c_i \text{ for } i = 1,\dots, k \} \subseteq \cO'.
		\]
		Suppose that
		\[
		\phi_i^{\cN}(\bX') = \inf_{\bY' \in \prod_{j \in \N} D_{r_{i,j}}^{\cN}} \psi_i^{\cN}(\bX',\bY')
		\]
		for all $\cN \models \rT_{\tr}$ and $\bX' \in L^\infty(\cN)^{\N}$, where $\psi_i$ is quantifier-free.  Because $\phi_i^{\cM}(\bX) < c_i$, there exists $\bY_i \in \prod_{j \in \N} D_{r_{i,j}}^{\cM}$ with $\psi_i^{\cM}(\bX,\bY_i) < c_i$.  By Proposition \ref{prop:deffuncrealization}, there exists a quantifier-free definable function $\mathbf{f}_i$ such that $\bY_i = \mathbf{f}_i^{\cM}(\bX)$ and $\mathbf{f}_i^{\cN}$ maps into $\prod_{j \in \N} D_{r_{i,j}}^{\cN}$ for all $\cN \models \rT_{\tr}$.  Let
		\[
		\eta_i^{\cN}(\bX') = \psi_i^{\cN}(\bX',\mathbf{f}^{\cN}(\bX')) \geq \phi_i^{\cN}(\bX').
		\]
		Then $\eta_i$ is quantifier-free.  Thus,
		\[
		\tp_{\qf}^{\cM}(\bX) \in \cO := \{\nu \in \mathbb{S}_{\qf,\mathbf{r}}(\rT_{\tr}): \nu(\eta_i) < c_i \},
		\]
		and
		\[
		\pi^{-1}(\cO) \subseteq \{\nu \in \mathbb{S}_{\exists,\mathbf{r}}(\rT_{\tr}): \nu(\phi_i) < c_i \text{ for } i = 1,\dots, k \} \subseteq \cO'
		\]
		as desired.
	\end{proof}
	
	\begin{remark} \label{rem:nonseparableqfentropy}
		Therefore, it is natural to define $\Ent_{\qf}^{\cU}(\cM)$ for general (not necessarily separable $\cM$) by $\Ent_{\qf}^{\cU}(\cM) := \Ent_{\exists}^{\cU}(\cM:\cM)$.
	\end{remark}
	
	\subsection{Existential entropy and entropy in the presence} \label{subsec:inthepresence}
	
	Let us finally explain why the existential entropy defined here agrees with (the ultrafilter version of) Hayes' $1$-bounded entropy of $\cN$ in the presence of $\cM$ in \cite{Hayes2018}.  The definition is given in terms of Voiculescu's microstate spaces for some $\bX$ in the presence of $\bY$ from \cite{VoiculescuFE3}.
	
	\begin{definition}[Hayes \cite{Hayes2018}] \label{def:Hayespresence}
		Let $\cM$ be a tracial $\mathrm{W}^*$-algebra.  Let $I$ and $J$ be arbitrary index sets and let $\bX \in L^\infty(\cM)^I$ and $\bY \in L^\infty(\cM)^J$.  Let $\mathbf{r} \in (0,\infty)^I$ and $\mathbf{r}' \in (0,\infty)^J$ such that $\norm{X_j} \leq r_j$ and $\norm{Y_j} \leq r_j'$.  Let $\mathbb{S}_{\mathbf{r},\mathbf{r}',\qf}(\rT_{\tr})$ be the set of quantifier-free types of tuples from $\prod_{i \in I} D_{r_i} \times \prod_{j \in J} D_{r_j'}$ equipped with the weak-$\star$ topology.  Let $p: M_n(\C)^{I \sqcup J} \to M_n(\C)^{I}$ be the canonical coordinate projection.  Then we define
		\[
		h_{\mathbf{r},\mathbf{r}'}^{\cU}(\bX:\bY) := \sup_{\epsilon > 0} \sup_{F \subseteq I \text{ finite}} \inf_{\cO \ni \tp_{\qf}^{\cM}(\bX,\bY)} \lim_{n \to \cU} \frac{1}{n^2} \log K_{F,\epsilon}^{\orb}(p[\Gamma_{\mathbf{r},\mathbf{r}'}^{(n)}(\cO)]),
		\]
		where $\cO$ ranges over all neighborhoods of $\tp_{\qf}^{\cM}(\bX,\bY)$ in $\mathbb{S}_{\mathbf{r},\mathbf{r}',\qf}(\rT_{\tr})$.
	\end{definition}
	
	Here we use arbitrary index sets $I$ and $J$ rather than $\N$ because we do not assume that $\cM$ is separable.  This is a technical issue we will have to consider when proving that our definition using $\N$-tuples agrees with Hayes.'  Apart from that, the idea of the proof is that a matrix tuple $\bX'$ is in the projection $p[\Gamma_{\mathbf{r},\mathbf{r}'}^{(n)}(\cO)]$ if and only if \emph{there exists} some $\bY'$ such that $\tp_{\qf}^{M_n(\C)}(\bX',\bY') \in \cO$.  If $\bX', \bY'$ being in $\Gamma_{\mathbf{r},\mathbf{r}'}^{(n)}(\cO)$ can be detected by a quantifier-free formula being less than some $c$ (using Urysohn's lemma), then $\bX'$ being in $p[\Gamma_{\mathbf{r},\mathbf{r}'}^{(n)}(\cO)]$ can be detected by an existential formula.
	
	\begin{proposition} \label{prop:presenceequivalence}
		In the setup of Definition \ref{def:Hayespresence}, we have $h_{\mathbf{r},\mathbf{r}'}^{\cU}(\bX:\bY) = \Ent_{\exists}^{\cU}(\mathrm{W}^*(\bX):\mathrm{W}^*(\bX,\bY))$.
	\end{proposition}
	
	We remark at the start of the proof that all the facts we proved about definable predicates and functions work for arbitrary index sets, so long as they do not invoke metrizability of the type space.  We also leave some details to the reader for the sake of space.
	
	\begin{proof}
		We may assume without loss of generality that $\cM = \mathrm{W}^*(\bX,\bY)$ since restricting to a smaller $\mathrm{W}^*$-algebra does not change the quantifier-free type of $(\bX,\bY)$.
		
		First, let us show that $h_{\mathbf{r},\mathbf{r}'}^{\cU}(\bX:\bY) \leq \Ent_{\exists}^{\cU}(\mathrm{W}^*(\bX):\cM)$.  Let $F \subseteq I$ finite and $\epsilon > 0$.  First, to deal with changing index sets from $I$ to $\N$, let $\alpha: F \to \N$ be an injective function and let $\mathbf{f}$ be the quantifier-free definable function that sends an $I$-tuple $\bX'$ to the $\N$-tuple obtained by putting $X_j'$ into the $\alpha(j)$th entry for $j \in F$ and fills the other entries with zeros.  Let $\bZ = \mathbf{f}^{\cM}(\bX)$, fix some $\rT \in (0,\infty)^{\N}$ with $\bZ \in \prod_{j \in \N} D_{t_j}^{\cN}$, and let $\cO$ be a neighborhood of $\mu = \tp_{\exists}^{\cM}(\bZ)$ in $\mathbb{S}_{\exists,\rT}(\rT_{\tr})$.  Then there exist existential definable predicates $\phi_1$, \dots, $\phi_k$ and $\epsilon_1$, \dots, $\epsilon_k > 0$ such that
		\[
		\{\nu \in \mathbb{S}_{\exists,\rT}(\rT_{\tr}): \nu(\phi_j) \leq \mu(\phi_j) + \epsilon_j \text{ for } j = 1, \dots, k\} \subseteq \cO.
		\]
		There exist quantifier-free definable predicates $\psi_1$, \dots, $\psi_k$ such that
		\[
		\phi_j^{\cN}(\bZ') = \inf_{\bW' \in \prod_{i \in \N} D_{t_{i,j}}^{\cN}} \psi_j^{\cN}(\bZ',\bW') \text{ for all } \cN \text{ and } \bZ' \in L^\infty(\cN)^{\N}.
		\]
		Moreover, for our particular $\cM$ and $\bZ$, there exists $\bW_j \in \prod_{i \in \N} D_{t_{i,j}}^{\cM}$ such that
		\[
		\psi_j^{\cM}(\bZ,\bW_j) < \mu(\phi_j) + \epsilon_j.
		\]
		By Proposition \ref{prop:deffuncrealization}, $\bW_j = \mathbf{g}_j(\bX,\bY)$ for some quantifier-free definable function $\mathbf{g}_j$.  Let $\psi_j(\mathbf{f},\mathbf{g}_j)$ denote the quantifier-free definable predicate defined for $I \sqcup J$-tuples by applying $\mathbf{f}$ to the $I$-tuple and $\mathbf{g}_j$ to the $I \sqcup J$-tuple and then applying $\psi_j$.  Then
		\[
		\cO' := \bigcap_{j=1}^k \{\sigma \in \mathbb{S}_{\qf,\mathbf{r},\mathbf{r}'}(\rT_{\tr}): \sigma(\psi_j(\mathbf{f},\mathbf{g}_j)) < \mu(\phi_j) + \epsilon_j \}
		\]
		is a neighborhood of $\tp_{\qf}^{\cM}(\bX,\bY)$ in $\mathbb{S}_{\qf,\mathbf{r},\mathbf{r}'}(\rT_{\tr})$ such that
		\[
		p[\Gamma_{\mathbf{r},\mathbf{r}'}^{(n)}(\cO')] \subseteq (\mathbf{f}^{M_n(\C)})^{-1}[\Gamma_{\rT}^{(n)}(\cO)].
		\]
		Therefore,
		\[
		K_{F,\epsilon}^{\orb}(p[\Gamma_{\mathbf{r},\mathbf{r}'}^{(n)}(\cO')]) \leq K_{\alpha(F),\epsilon}^{\orb}(\Gamma_{\rT}^{(n)}(\cO)).
		\]
		Because for every such $\cO$, there exists such an $\cO'$, we obtain that
		\[
		\inf_{\cO' \ni \tp_{\qf}^{\cM}(\bX,\bY)} \lim_{n \to \cU} \frac{1}{n^2} \log K_{F,\epsilon}^{\orb}(p[\Gamma_{\mathbf{r},\mathbf{r}'}^{(n)}(\cO')]) \leq \Ent_{\exists,\alpha(F),\epsilon}^{\cU}(\tp_{\exists}^{\cM}(\mathbf{f}(\bX))) \leq \Ent_{\exists}^{\cU}(\mathrm{W}^*(\bX): \cM).
		\]
		Since $F$ and $\epsilon$ were arbitrary, we are done with the first inequality.
		
		To prove the second inequality, we must show that for all $\bZ \in \mathrm{W}^*(\bX)^{\N}$, we have $\Ent_{\exists}(\tp_{\exists}^{\cM}(\bZ)) \leq h_{\mathbf{r},\mathbf{r}'}(\bX: \bY)$.  Fix $\bZ$, let $\rT \in (0,\infty)^{\N}$ with $\norm{Z_j} \leq t_j$, and write $\bZ = \mathbf{f}^{\cM}(\bX)$ for some quantifier-free definable function $\mathbf{f}$ depending on countably many coordinates of $\bX$.  Let $\cO'$ be a neighborhood of $\tp_{\qf}^{\cM}(\bX,\bY)$.  Note that $\cO'$ contains a neighborhood of $\mu$ that depends only on finitely many coordinates of $\bX$ and $\bY$.  By Urysohn's lemma and Remark \ref{rem:predicateextension}, there exists a quantifier-free definable predicate $\psi$ with values in $[0,1]$ (depending on only finitely many coordinates) such that $\psi^{\cM}(\bX,\bY) = 0$ and
		\[
		\{\sigma \in \mathbb{S}_{\qf,\mathbf{r},\mathbf{r}'}(\rT_{\tr}): \sigma(\psi) < 1 \} \subseteq \cO'.
		\]
		Fix $F \subseteq \N$ finite and $\epsilon \in (0,2)$, and consider the existential formula
		\[
		\phi^{\cN}(\bZ') = \inf_{\bX' \in \prod_{i \in I} D_{r_i}^{\cN}} \inf_{\bY' \in \prod_{j \in J} D_{r_j'}^{\cN}} \left( \sum_{k \in F} d^{\cN}(\mathbf{f}_k^{\cN}(\bX'), \bZ_k') + \psi^{\cN}(\bX',\bY') \right).
		\]
		Because $\mathbf{f}$ and $\psi$ only depend on countably many coordinates, the infima can be expressed using only countably many variables, so this expression is a valid existential definable predicate.  Moreover, note that for $\bZ' \in \prod_{k \in \N} D_{t_k}^{M_n(\C)}$, we have
		\[
		\phi^{M_n(\C)}(\bZ') < \frac{\epsilon}{2} \implies \bZ' \in N_{\epsilon/2}(\mathbf{f}^{M_n(\C)} \circ p(\Gamma_{\mathbf{r},\mathbf{r}'}^{(n)})(\cO')).
		\]
		Let $\cO = \{\mu \in \mathbb{S}_{\exists,\rT}(\rT_{\tr}): \mu(\phi) < \epsilon/2\}$.  By applying the uniform continuity of $\mathbf{f}$ (Lemma \ref{lem:deffuncunifcont}) for the given $\epsilon/2$ and $F$ in the target space, we obtain a corresponding $F' \subseteq I$ and $\delta > 0$ such that
		\[
		K_{\epsilon,F}^{\orb}(\Gamma_{\rT}^{(n)}(\cO)) \leq K_{\epsilon/2,F}^{\orb}(\mathbf{f}^{M_n(\C)} \circ p(\Gamma_{\mathbf{r},\mathbf{r}'}^{(n)})(\cO')) \leq K_{\delta,F'}^{\orb}(p(\Gamma_{\mathbf{r},\mathbf{r}'}^{(n)}(\cO'))).
		\]
		Applying the definitions of the appropriate limits, suprema, and infima shows that $\Ent^{\cU}(\tp_{\exists}^{\cM}(\bZ)) \leq h_{\mathbf{r},\mathbf{r}'}(\bX: \bY)$.
	\end{proof}
	
	\subsection{Applications to ultraproduct embeddings}
	
	\begin{theorem} \label{thm:main2}
		Let $\cN \subseteq \cM$ be separable tracial $\mathrm{W}^*$-algebras, and let $\cQ = \prod_{n \to \cU} M_n(\C)$. Suppose that $\Ent_{\exists}^{\cU}(\cN:\cM) \geq 0$.  Then for every $c < \Ent_{\exists}^{\cU}(\cN:\cM)$, there exists an embedding $\iota: \cM \to \cQ$ such that $\Ent_{\exists}^{\cU}(\iota(\cN):\cQ) \geq \Ent^{\cU}(\iota(\cN):\cQ) > c$.
	\end{theorem}
	
	\begin{proof}
		Let $\bX$ be an $\N$-tuple of generators for $\cN$.  Let $\pi: \mathbb{S}(\rT_{\tr}) \to \mathbb{S}_{\exists}(\rT_{\tr})$ be the restriction map.  By Lemma \ref{lem:existentialversusfull},
		\[
		c < \Ent_{\exists}^{\cU}(\tp_{\exists}^{\cM}(\bX)) = \sup_{\mu \in \pi^{-1}(\cK_{\tp_{\exists}^{\cM}(\bX)})} \Ent^{\cU}(\mu),
		\]
		so there exists a type $\mu$ such that $\pi(\mu) \in \cK_{\tp_{\exists}^{\cM}(\bX)}$ and $\Ent^{\cU}(\mu) > c$.  By Lemma \ref{lem:ultraproductrealization}, there exists $\bX' \in \cQ$ with $\tp^{\cQ}(\bX') = \mu$.  As in Corollary \ref{cor:embeddings}, there exists an embedding $\iota: \cN \to \cQ$ with $\iota(\bX) = \bX'$.  Observe that
		\[
		\Ent_{\exists}^{\cU}(\iota(\cN):\cQ) = \Ent_{\exists}^{\cU}(\tp_{\exists}^{\cQ}(\bX')) \geq \Ent^{\cU}(\tp^{\cQ}(\bX')) = \Ent^{\cU}(\iota(\cN):\cQ) = \Ent^{\cU}(\mu) > c,
		\]
		where we apply in order Corollary \ref{cor:existentialgenerators}, Lemma \ref{lem:existentialversusfull}, Observation \ref{obs:monotonicity2}, and the choice of $\mu$ and $\bX'$.
		
		It only remains to show that $\iota$ extends to an embedding of $\cM$.  Let $\bY \in L^\infty(\cM)^{\N}$ be a set of generators.  Let $\mathbf{r}$ and $\mathbf{r}'$ be such that $\bX \in \prod_{j \in \N} D_{r_j}^{\cM}$ and $\bY \in \prod_{j \in \N} D_{r_j'}^{\cM}$.  Since the quantifier-free type space $\mathbb{S}_{\qf,\mathbf{r},\mathbf{r}'}(\rT_{\tr})$ for $\N \sqcup \N$-tuples is metrizable, there exists a nonnegative continuous function on $\mathbb{S}_{\qf,\mathbf{r},\mathbf{r}'}(\rT_{\tr})$ that equals zero at and only at the point $\tp_{\qf}^{\cM}(\bX,\bY)$.  By Remark \ref{rem:predicateextension}, this continuous function extends to a global quantifier-free definable predicate $\phi$.  Let $\psi$ be the existential predicate given by
		\[
		\psi^{\cN}(\bZ) = \inf_{\bW \in \prod_{j \in \N} D_{r_j'}^{\cN}} \phi^{\cN}(\bZ,\bW)
		\]
		for $\cN \models \rT_{\tr}$ and $\bZ \in L^\infty(\cN)^{\N}$.  Thus, $\psi^{\cM}(\bX) \leq \phi^{\cM}(\bX,\bY) = 0$.
		
		Because $\psi$ is existential and $\tp_{\exists}^{\cQ}(\bX') \in \cK_{\tp_{\exists}^{\cM}(\bX)}$, we have $\psi^{\cQ}(\bX') \leq \psi^{\cM}(\bX) = 0$.  We may write $\bX' = [\bX^{(n)}]_{n \in \N}$ where $\bX^{(n)} \in \prod_{j \in \N} D_{r_j}^{M_n(\C)}$ (this follows for instance from the construction of $\bX'$ through Lemma \ref{lem:ultraproductrealization}).  Then
		\[
		\lim_{n \to \cU} \psi^{M_n(\C)}(\bX^{(n)}) = \psi^{\cQ}(\bX') = 0,
		\]
		hence there exists $\bY^{(n)} \in \prod_{j \in \N} D_{r_j'}^{M_n(\C)}$ such that
		\[
		\lim_{n \to \cU} \phi^{M_n(\C)}(\bX^{(n)},\bY^{(n)}) = 0.
		\]
		Let $\bY' = [\bY^{(n)}]_{n \in \N} \in \prod_{j \in \N} D_{r_j'}^{\cQ}$.  Then $\phi^{\cQ}(\bX',\bY') = 0$, and therefore, $\tp_{\qf}^{\cQ}(\bX',\bY') = \tp_{\qf}^{\cM}(\bX,\bY)$.  Hence, by Lemma \ref{lem:lawisomorphism}, there exists an embedding $\iota': \cM \to \cQ$ with $\iota'(\bX,\bY) = (\bX',\bY')$.  This is the desired extension of $\iota$.
	\end{proof}
	
	\begin{remark}
		Note that $\Ent_{\exists}^{\cU}(\iota(\cN):\cQ) \leq \Ent_{\exists}^{\cU}(\cN:\cM)$ for any such embedding $\iota$.  Thus, the point of the theorem is that \emph{some} $\iota$ can be chosen to make this inequality close to an equality.  It is not obvious that there is an existential type in $\cQ$ extending the $\tp_{\exists}^{\cM}(\bX)$ with close to the same amount of entropy of $\mu$.  The key ingredient is the variational principle (Proposition \ref{prop:variational}) applied through Lemma \ref{lem:existentialversusfull}, which gives us not only an existential type $\tp_{\exists}^{\cQ}(\bX')$ extending $\tp_{\exists}^{\cM}(\bX)$ with large entropy, but even the full type $\tp^{\cQ}(\bX')$ with large entropy.
	\end{remark}
	
	In particular, the theorem shows that if $\Ent_{\qf}^{\cU}(\cM) > 0$, then there exists an embedding of $\cM$ into $\cQ$ with $\Ent^{\cU}(\cM:\cQ) > 0$, and hence also $h(\cM:\cQ) = \Ent_{\exists}^{\cU}(\cM:\cQ) > 0$.

	\section{Remarks on conditional entropy} \label{sec:conditional}
	
	In this section, we sketch how the previous results could be adapted to the setting of entropy relative to a $\mathrm{W}^*$-subalgebra.  However, we will not give the arguments in detail because we will not be giving any new applications of the conditional version of entropy.  Our goal is mainly to complete our translation between the different flavors of microstate spaces in free entropy theory and the different flavors of types in the conditional setting.
	
	Hayes' original definition of $1$-bounded entropy used microstate spaces relative to a fixed microstate sequence for some self-adjoint element with diffuse spectrum.  He then showed that this was equivalent to the $1$-bounded entropy defined through unitary orbits (the definition that we have used so far in this paper).  As remarked in \cite[\S 4.1]{HJKE2021}, the same reasoning shows that orbital $1$-bounded entropy is equivalent to $1$-bounded entropy relative to fixed microstates for any diffuse amenable $\mathrm{W}^*$-subalgebra $\cP$ of $\cM$.  In fact, one can formulate the definition of $1$-bounded entropy of $\cM$ relative to any $\mathrm{W}^*$-subalgebra $\cA$ with a fixed choice of microstates $\bY^{(n)}$ with $\lim_{n \to \cU} \tp_{\qf}^{M_n(\C)}(\bY^{(n)}) = \tp_{\qf}^{\cM}(\bY)$.  Unlike the case where $\cA$ is amenable, the $1$-bounded entropy relative to $\cA$ may, as far as we know, depend on the choice of microstates for $\cA$, and in general it will not coincide with the orbital $1$-bounded entropy.  Nonetheless, relative $1$-bounded entropy for general $\cA$ has a natural motivation in terms of ultraproduct embeddings:  Fixing $\cA \subseteq \cN \subseteq \cM$ and an embedding $\iota: \cA \to \cQ := \prod_{n \to \cU} M_n(\C)$, a relative $1$-bounded entropy $h(\cN:\cM | \cA, \iota)$ would quantify the amount of embeddings of $\iota': \cN \to \cQ$ that extend $\iota$ and which admit some extension $\iota'': \cM \to \cQ$.
	
	Just as we have interpreted the entropy in the presence as corresponding to existential types in the model-theoretic framework, relative entropy naturally corresponds to types over $\cA$.  Types over $\cA$ represent types in a language $\cL_{\tr,\cA}$ described as follows.  Let $\cA$ be a separable tracial $\mathrm{W}^*$-algebra.  Let $\cL_{\tr,\cA}$ be the language obtained by adding to $\cL_{\tr}$ a constant symbol $\alpha(a) \in D_{\norm{a}}$ for each $a \in \cA$.
	
	Let $\rT_{\tr,\cA}$ be the $\cL_{\tr,\cA}$ theory obtained from $\rT_{\tr}$ by adding the (infinite family of) axioms
	\begin{itemize}
		\item $\alpha(a+b) = \alpha(a) + \alpha(b)$ for each $a, b \in \cA$.
		\item $\alpha(\lambda a) = \lambda \alpha(a)$ for $a \in \cA$ and $\lambda \in \C$.
		\item $\alpha(ab) = \alpha(a) \alpha(b)$ for $a, b \in \cA$.
		\item $\alpha(a^*) = \alpha(a)^*$ for $a \in \cA$.
		\item $\alpha(1) = 1$.
		\item $\re \tr \alpha(a) = \tau_{\cA}(a)$ where $\tau_{\cA}$ is the given trace on the tracial $\mathrm{W}^*$-algebra $\cA$.
	\end{itemize}
	We leave it as an exercise to the reader to verify that every model of $\rT_{\tr,\cA}$ is given by a tracial $\mathrm{W}^*$-algebra $\cM$ together with an embedding (unital, trace-preserving $*$-homomorphism) $\alpha: \cA \to \cM$, and conversely every such embedding defines a model of $\rT_{\tr,\cA}$.  Given a tracial $\mathrm{W}^*$-algebra $\cM$ and an inclusion $\alpha: \cA \to \cM$, the $\cL_{\tr,\cA}$-type of a tuple $\bX$ is also known as the \emph{type of $\bX$ over $\cA$} and denoted $\tp^{\cM}(\bX / \cA)$.
	
	Next, we want to define versions of entropy for quantifier-free, full, and existential types over $\cA$, using covering numbers for microstate spaces corresponding to neighborhoods of the type over $\cA$.  Unfortunately, we cannot use neighborhoods in the space of $\cL_{\tr,\cA}$-types $\mathbb{S}_{\cA}(\rT_{\tr,\cA})$ because the matrix algebra $M_n(\C)$ could never be a model of $\rT_{\tr,\cA}$ since it cannot contain a copy of $\cA$ unless $\cA$ is finite-dimensional.  In other words, the issue is that we must work with approximate embeddings $\alpha_n: \cA \to M_n(\C)$ rather than literal embeddings, since the latter may not exist.  Thus, we will look at $\cL_{\tr,\cA}$ structures that satisfy $\rT_{\tr}$ but not necessarily $T_{\tr,\cA}$, which are tracial von Neumann algebras together with a function $\alpha: \cA \to \cM$ that is not necessarily is a $*$-homomorphism or even linear but does satisfy $\norm{\alpha(a)} \leq \norm{a}$ for $a \in \cA$.  We will denote by $\mathbb{S}_{\cA}(T_{\tr})$ the set of $\cL_{\tr,\cA}$-types that arise from models of $T_{\tr}$, so that $\mathbb{S}_{\cA}(\rT_{\tr}) \supseteq S_{\cA}(\rT_{\tr,\cA})$.
	
	Given a sequence of functions $\alpha_n: \cA \to M_n(\C)$ and $\cO \subseteq S_{\cA}(T_{\tr})$, we define the microstate space
	\[
	\Gamma_{\mathbf{r}}^{(n)}(\cO \mid \alpha_n) = \{\bX \in M_n(\C)^{\N}: \tp^{M_n(\C),\alpha_n}(\bX) \in \cO \},
	\]
	where $\tp^{M_n(\C),\alpha_n}(\bX)$ is the $\cL_{\cA}$ type of $\bX$ in the $\cL_{\cA}$ structure given by $M_n(\C)$ and $\alpha_n: \cA \to M_n(\C)$.  We are interested only in the case when $(\alpha_n)_{n\in \N}$ defines a trace-preserving $*$-homomorphism $\alpha: \cA \to \cQ = \prod_{n \to \cU} M_n(\C)$.  Then for a closed set $\cK \subseteq \mathbb{S}_{\cA}(\rT_{\tr,\cA}) \subseteq \mathbb{S}_{\cA}(\rT_{\tr})$, we define
	\[
	\Ent_{\mathbf{r},F,\epsilon}^{\cU}(\cK \mid \alpha) = \inf_{\cO \supseteq \cK \text{ open}} \lim_{n \to \cU} \frac{1}{n^2} \log K_{F,\epsilon}(\Gamma_{\mathbf{r}}^{(n)}(\cO \mid \alpha),
	\]
	where the infimum is over all open neighborhoods of $\cK$ in $S_{\cA}(\rT_{\tr})$, and then let $\Ent^{\cU}(\cK \mid \alpha)$ be the supremum over $\mathbf{r}$, $F$, and $\epsilon$.
	
	As the notation above suggests, it turns out that this quantity only depends on the embedding $\alpha: \cA \to \cQ$, not on the particular lift $(\alpha_n)_{n \in \N}$.  To see this, suppose $\beta_n$ is another such lift, so that for every $a \in \cA$ we have $d^{M_n(\C)}(\alpha_n(a),\beta_n(a)) \to 0$ as $n \to \cU$.  Using Urysohn's lemma, taking a smaller neighborhood if necessary, we can assume the neighborhood $\cO$ is given by $\phi < \delta$ for some nonnegative formula $\phi(x_1,x_2,\dots)$ in $\cL_{\tr,\cA}$.  Then $\phi$ can be equivalently viewed as an $\cL_{\tr}$ formula in the variables $x_j$ together with additional variables corresponding to the elements of $\cA$.  By uniform continuity of the formulas, $|\phi(\bX,\alpha_n(a))_{a \in \cA}) - \phi(\bX,(\beta_n(a))_{a \in \cA})| < \delta/2$ for $n$ in a small enough neighborhood of $\cU$.  Thus, if the neighborhood $\cO'$ is given by $\phi < \delta/2$, we get $\Gamma_{\mathbf{r}}^{(n)}(\cO' \mid \beta_n) \subseteq \Gamma_{\mathbf{r}}^{(n)}(\cO \mid \alpha_n)$.  The argument is finished by taking the appropriate infima over $\cO$ and limits.\footnote{For the analog of this argument in the existential case, we would work only with the case when $\cK= \cK_\mu$ for a single existential $\cL_{\tr,\cA}$-type.  The only issue adapting the above argument to the existential case is in finding, for a given a neighborhood $\cO$ of $\cK_\mu$, a sub-neighborhood of the form $\phi^{-1}((-\infty,\epsilon))$ for an existential formula $\phi$.  Since the space of existential types is not Hausdorff, we cannot apply Urysohn's lemma, but rather must work with the existential formulas directly to construct such a neighborhood.}
	
	We remark that the approximate embedding $\alpha_n: \cA \to M_n(\C)$ can be thought of as a choice of microstates for \emph{every} element of $\cA$.  But, as in Hayes original description of relative $1$-bounded entropy, we could instead fix a generating set $\bA$ for $\cA$, fix microstates $\bA^{(n)}$ for that \emph{generating set}, and define microstate spaces of matrix tuples $\bX$ such that the $\cL_{\tr}$-type of $(\cA^{(n)},\bX)$ is in a certain neighborhood $\cO$ of the set $\cK$.  It is a technical exercise to show that these definitions are equivalent, the key point being that every element of $a$ can be expressed as a quantifier-free definable function of the generating tuple $\bA$.
	
	Most of the properties we showed for $\Ent^{\cU}$ adapt to the relative version with the same method of proof.   For instance, it satisfies the analog of the variational principle (Proposition \ref{prop:variational}) and monotonicity under pushforward (\ref{prop:pushforward}).  Thus, given $\cA \subseteq \cN \subseteq \cM$ and an embedding $\alpha: \cA \to \cQ$, we can define $\Ent^{\cU}(\cN: \cM \mid \alpha)$ as the supremum of $\Ent^{\cU}(\tp^{\cM,\alpha}(\bX))$ for tuples $\bX$ from $\cN$.  Analogously to Lemma \ref{lem:ultraproductrealization}, if $\Ent^{\cU}(\cM \mid \alpha) \geq 0$, then there is an embedding of $\cN$ into $\cQ$ that restricts to $\alpha$ on $\cA$ and extends to an elementary embedding of $\cM$.  The quantifier-free and existential version of conditional entropy are defined in a similar way, and the relationship between them works the same way as it does for the unconditional version.

	%
	%
	
	\bibliographystyle{plain}
	\bibliography{modelmicrostate.bib}

\begin{thebibliography}{10}

\bibitem{AGKE2022}
Scott Atkinson, Isaac Goldbring, and Srivatsav {Kunnawalkam Elayavalli}.
\newblock Factorial relative commutants and the generalized jung property for
  ii1 factors.
\newblock {\em Advances in Mathematics}, 396:108107, 2022.

\bibitem{AKE2021}
Scott Atkinson and Srivatsav {Kunnawalkam Elayavalli}.
\newblock On ultraproduct embeddings and amenability for tracial von neumann
  algebras.
\newblock {\em International Mathematics Research Notices}, 2021(4):2882--2918,
  10 2020.

\bibitem{BYBHU2008}
Ita{\"i} {Ben Yaacov}, Alexander Berenstein, C.~Ward Henson, and Alexander
  Usvyatsov.
\newblock Model theory for metric structures.
\newblock In Z.~Chatzidakis et~al., editor, {\em Model Theory with Applications
  to Algebra and Analysis, Vol. II}, volume 350 of {\em London Mathematical
  Society Lecture Notes Series}, pages 315--427. Cambridge University Press,
  2008.

\bibitem{BYU2010}
Itai{\"i} {Ben Yaacov} and Alexander Usvyatsov.
\newblock Continuous first order logic and local stability.
\newblock {\em Transactions of the American Mathematical Society},
  362(10):5213--5259, 10 2010.

\bibitem{Blackadar2006}
Bruce Blackadar.
\newblock {\em Operator Algebras: Theory of ${C}^*$-algebras and von {N}eumann
  algebras}, volume 122 of {\em Encyclopaedia of Mathematical Sciences}.
\newblock Springer-Verlag, Berlin, Heidelberg, 2006.

\bibitem{Dixmier1969}
Jacques Dixmier.
\newblock {\em Les ${C}^*$-alg{\`e}bres et leurs repr{\'e}sentations},
  volume~29 of {\em Cahiers Scientifiques}.
\newblock Gauthier-Villars, Paris, 2 edition, 1969.
\newblock Reprinted by {E}ditions {J}acques {G}abay, {P}aris, 1996.
  {T}ranslated as ${C}^*$-algebras, North-Holland, Amsterdam, 1977. First Edi-
  tion 1964.

\bibitem{FHS2013}
Ilias Farah, Bradd Hart, and David Sherman.
\newblock Model theory of operator algebras {I}: stability.
\newblock {\em Bulletin of the London Mathematical Society}, 45(4):825--838,
  2013.

\bibitem{FHS2014}
Ilias Farah, Bradd Hart, and David Sherman.
\newblock Model theory of operator algebras {II}: model theory.
\newblock {\em Israel Journal of Mathematics}, 201(1):477--505, 2014.

\bibitem{FHS2014b}
Ilias Farah, Bradd Hart, and David Sherman.
\newblock Model theory of operator algebras {III}: elementary equivalence and
  $\mathrm{II}_1$ factors.
\newblock {\em Bulletin of the London Mathematical Society}, 46(3):609--628,
  2014.

\bibitem{GP2015}
A.~Galatan and S.~Popa.
\newblock Smooth bimodules and cohomology of $\mathrm{II}_1$ factors.
\newblock {\em J. Inst. Math. Jussieu}, pages 1--33, 2015.

\bibitem{Gao2020}
Weichen Gao.
\newblock Relative embeddability of von neumann algebras and amalgamated free
  products.
\newblock With an appendix by Weichen Gao and Marius Junge. arXiv:2012.07940,
  2020.

\bibitem{GePrime}
Liming Ge.
\newblock Applications of free entropy to finite von {N}eumann algebras. {II}.
\newblock {\em Ann. of Math. (2)}, 147(1):143--157, 1998.

\bibitem{Goldbring2021nonembeddable}
Isaac Goldbring.
\newblock Non-embeddable $\mathrm{II}_1$ factors resembling the hyperfinite
  $\mathrm{II}_1$ factor.
\newblock Preprint arXiv:2101.10467.

\bibitem{Goldbring2020}
Isaac Goldbring.
\newblock On {P}opa's factorial commutant embedding problem.
\newblock {\em Proc. Amer. Math. Soc.}, 148:5007--5012, 2020.

\bibitem{Goldbring2021enforceable}
Isaac Goldbring.
\newblock Enforceable operator algebras.
\newblock {\em Journal of the Institute of Mathematics of Jussieu}, 20:31--63,
  2021.

\bibitem{Hayes2018}
Ben Hayes.
\newblock 1-bounded entropy and regularity problems in von {N}eumann algebras.
\newblock {\em International Mathematics Research Notices}, 1(3):57--137, 2018.

\bibitem{HJKE2021}
Ben Hayes, David Jekel, and Srivatsav {Kunnawalkam Elayavalli}.
\newblock Property {(T)} and strong $1$-boundedness for von {N}eumann algebras.
\newblock Preprint, arXiv:2107.03278.

\bibitem{HJNS2021}
Ben Hayes, David Jekel, Brent Nelson, and Thomas Sinclair.
\newblock A random matrix approach to absorption theorems for free products.
\newblock {\em International Mathematics Research Notices}, 2021(3):1919--1979,
  2021.

\bibitem{HS2011}
Neil Hindman and Dona Strauss.
\newblock {\em Algebra in the {S}tone-{\v C}ech Compactification: Theory and
  Applications}.
\newblock De Gruyter, Berlin, Boston, 2011.

\bibitem{IS2021}
Adrian Ioana and Pieter Spaas.
\newblock $\mathrm{II}_1$ factors with exotic central sequence algebras.
\newblock {\em Journal of the Institute of Mathematics of Jussieu},
  20(5):1671--1696, 2021.

\bibitem{James1990}
Ioan~M. James.
\newblock {\em Introduction to Uniform Spaces}, volume 144 of {\em London
  Mathematical Society Lecture Note Series}.
\newblock Cambridge University Press, 1990.

\bibitem{JekelThesis}
David Jekel.
\newblock {\em Evolution equations in non-commutative probability}.
\newblock PhD thesis, University of California, Los Angeles, 2020.

\bibitem{JekelModelEntropy}
David Jekel.
\newblock Free probability and model theory of tracial $\mathrm{W}^*$-algebras.
\newblock preprint, arXiv:2208.13867, to appear in \emph{Model Theory of
  Operator Algebras}, ed. I. Goldbring, published by DeGruyter, 2022.

\bibitem{Jung2007embeddings}
Kenley Jung.
\newblock Amenability, tubularity, and embeddings into {$\mathcal{R}^\omega$}.
\newblock {\em Math. Ann.}, 338(1):241--248, 2007.

\bibitem{Jung2007S1B}
Kenley Jung.
\newblock Strongly $1$-bounded von {N}eumann algebras.
\newblock {\em Geom. Funct. Anal.}, 17(4):1180--1200, 2007.

\bibitem{KadisonRingroseI}
Richard~V. Kadison and John~R. Ringrose.
\newblock {\em Fundamentals of the Theory of Operator Algebras I}, volume~15 of
  {\em Graduate Studies in Mathematics}.
\newblock American Mathematical Society, Providence, 1983.

\bibitem{Kaplansky1953}
Irving Kaplansky.
\newblock Modules over operator algebras.
\newblock {\em American Journal of Mathematics}, 75:839--858, 1953.

\bibitem{Popa2019}
Sorin Popa.
\newblock Asymptotic orthogonalization of subalgebras in $\mathrm{II}_1$ 1
  factors.
\newblock {\em Publ. Res. Inst. Math. Sci.}, 55(4):795--809, 2019.

\bibitem{Sakai1971}
Sh{\^o}ichir{\^o} Sakai.
\newblock {\em $\mathrm{C}^*$-algebras and $\mathrm{W}^*$-algebras}, volume~60
  of {\em Ergebnisse der {M}athematik und ihrer {G}renzgebiete}.
\newblock Springer-Verlag, Berlin Heidelberg, 1971.

\bibitem{Shlyakhtenko2009}
Dimitri Shlyakhtenko.
\newblock Lower estimates on microstates free entropy dimension.
\newblock {\em Analysis \& PDE}, 2(2):119--146, 2009.

\bibitem{TakesakiI}
Masamichi Takesaki.
\newblock {\em Theory of Operator Algebras I}, volume 124 of {\em Encyclopaedia
  of Mathematical Sciences}.
\newblock Springer-Verlag, Berlin Heidelberg, 2002.

\bibitem{VoiculescuFE2}
Dan-Virgil Voiculescu.
\newblock The analogues of entropy and of {F}isher's information in free
  probability, {II}.
\newblock {\em Inventiones Mathematicae}, 118:411--440, 1994.

\bibitem{VoiculescuFE3}
Dan-Virgil Voiculescu.
\newblock The analogues of entropy and of {F}isher's information measure in
  free probability, {III}: Absence of {C}artan subalgebras.
\newblock {\em Geometric and Functional Analysis}, 6:172--199, 1996.

\bibitem{Zhu1993}
Kehe Zhu.
\newblock {\em An Introduction to Operator Algebras}.
\newblock Studies in Advanced Mathematics. CRC Press, Ann Arbor, 1993.

\end{thebibliography}

\end{document}